\documentclass[a4paper]{amsart}
\usepackage[foot]{amsaddr}

\usepackage[latin1]{inputenc}
\usepackage{amssymb,amsmath,amsthm}
\usepackage{amsmath}
\usepackage{amsfonts}
\usepackage{enumitem}
\usepackage{booktabs}
\usepackage{ifthen}
\usepackage{tikz}
\usetikzlibrary{math,calc}
\usepackage{url}
\usepackage{hyphenat}
\usepackage{hyperref}
\usepackage{graphicx}
\usepackage{calc,etoolbox}
\usepackage{rotating}
\usepackage{caption}
\usepackage{subfig}
\usepackage{caption}

\theoremstyle{plain}
\newtheorem{theorem}{Theorem}
\newtheorem{proposition}[theorem]{Proposition}
\newtheorem{lemma}[theorem]{Lemma}

\theoremstyle{definition}
\newtheorem{definition}[theorem]{Definition}

\newtheorem{remark}[theorem]{Remark}

\newtheorem{claim}{Claim}
\numberwithin{claim}{theorem}
\newenvironment{pfclaim}[1][Proof]{\begin{trivlist}
\item[\hskip \labelsep {\textit{{#1}.}}]} {{\footnotesize $\blacksquare$}\end{trivlist}}

\numberwithin{figure}{section}
\numberwithin{table}{section}

\newcommand{\veewedge}{%
  \mathord{{\vee}\mkern-7.35mu{\wedge}}%
}

\newcommand{\IN}{\ensuremath{\mathbb{N}}}
\newcommand{\ZZ}{\ensuremath{\mathbb{Z}}}

\newcommand{\nset}[1]{\ensuremath{[{#1}]}}
\newcommand{\card}[1]{\ensuremath{\lvert{#1}\rvert}}

\newcommand{\gen}[2][\gendefault]{\ensuremath{\langle{#2}\rangle_{#1}}}
\newcommand{\clonegen}[1]{\gen[]{#1}}
\newcommand{\vect}[1]{\ensuremath{\mathbf{#1}}}
\newcommand{\lhs}{\hspace{2em}&\hspace{-2em}}

\newcommand{\clIntVal}[3]{\ensuremath{#1_{\ifthenelse{\equal{#2}{}}{\mathord{*}}{#2}\ifthenelse{\equal{#3}{}}{\mathord{*}}{#3}}}}
\newcommand{\clIntEq}[1]{\ensuremath{#1_{=}}}
\newcommand{\clIntNeq}[1]{\ensuremath{#1_{\neq}}}
\newcommand{\clIntLeq}[1]{\ensuremath{#1_{\leq}}}
\newcommand{\clIntGeq}[1]{\ensuremath{#1_{\geq}}}
\newcommand{\clIntNeqOO}[1]{\ensuremath{#1_{{\neq},00}}}
\newcommand{\clIntNeqII}[1]{\ensuremath{#1_{{\neq},11}}}

\newcommand{\clAll}{\ensuremath{\mathsf{\Omega}}}
\newcommand{\clEioo}{\ensuremath{\clIntNeqII{\clAll}}}
\newcommand{\clEioi}{\ensuremath{\clIntGeq{\clAll}}}
\newcommand{\clEiio}{\ensuremath{\clIntLeq{\clAll}}}
\newcommand{\clEiii}{\ensuremath{\clIntNeqOO{\clAll}}}
\newcommand{\clEq}{\ensuremath{\clIntEq{\clAll}}}
\newcommand{\clNeq}{\ensuremath{\clIntNeq{\clAll}}}
\newcommand{\clOX}{\ensuremath{{\clIntVal{\clAll}{0}{}}}}
\newcommand{\clIX}{\ensuremath{{\clIntVal{\clAll}{1}{}}}}
\newcommand{\clXO}{\ensuremath{{\clIntVal{\clAll}{}{0}}}}
\newcommand{\clXI}{\ensuremath{{\clIntVal{\clAll}{}{1}}}}

\newcommand{\clOXCI}{\ensuremath{\clOX \cup \clVaki}}

\newcommand{\clIXCO}{\ensuremath{\clIX \cup \clVako}}

\newcommand{\clXOCI}{\ensuremath{\clXO \cup \clVaki}}

\newcommand{\clXICO}{\ensuremath{\clXI \cup \clVako}}
\newcommand{\clOO}{\ensuremath{\clIntVal{\clAll}{0}{0}}}
\newcommand{\clII}{\ensuremath{\clIntVal{\clAll}{1}{1}}}
\newcommand{\clOI}{\ensuremath{\clIntVal{\clAll}{0}{1}}}
\newcommand{\clIO}{\ensuremath{\clIntVal{\clAll}{1}{0}}}

\newcommand{\clOOCI}{\ensuremath{\clOO \cup \clVaki}}
\newcommand{\clIICO}{\ensuremath{\clII \cup \clVako}}
\newcommand{\clOICO}{\ensuremath{\clOI \cup \clVako}}
\newcommand{\clOICI}{\ensuremath{\clOI \cup \clVaki}}
\newcommand{\clOIC}{\ensuremath{\clOI \cup \clVak}}
\newcommand{\clIOCO}{\ensuremath{\clIO \cup \clVako}}
\newcommand{\clIOCI}{\ensuremath{\clIO \cup \clVaki}}
\newcommand{\clIOC}{\ensuremath{\clIO \cup \clVak}}
\newcommand{\clS}{\ensuremath{\mathsf{S}}}
\newcommand{\clSc}{\ensuremath{\clIntVal{\clS}{0}{1}}}

\newcommand{\clSM}{\ensuremath{\mathsf{SM}}}

\newcommand{\clSmin}{\ensuremath{\clS^{-}}}
\newcommand{\clSmaj}{\ensuremath{\clS^{+}}}

\newcommand{\clSminOX}{\ensuremath{\clIntVal{\clSmin}{0}{}}}
\newcommand{\clSminXO}{\ensuremath{\clIntVal{\clSmin}{}{0}}}
\newcommand{\clSminOO}{\ensuremath{\clIntVal{\clSmin}{0}{0}}}
\newcommand{\clSminOI}{\ensuremath{\clIntVal{\clSmin}{0}{1}}}
\newcommand{\clSminIO}{\ensuremath{\clIntVal{\clSmin}{1}{0}}}
\newcommand{\clSminOICO}{\ensuremath{\clSminOI \cup \clVako}}
\newcommand{\clSminIOCO}{\ensuremath{\clSminIO \cup \clVako}}
\newcommand{\clSmajIX}{\ensuremath{\clIntVal{\clSmaj}{1}{}}}
\newcommand{\clSmajXI}{\ensuremath{\clIntVal{\clSmaj}{}{1}}}
\newcommand{\clSmajOI}{\ensuremath{\clIntVal{\clSmaj}{0}{1}}}
\newcommand{\clSmajIO}{\ensuremath{\clIntVal{\clSmaj}{1}{0}}}
\newcommand{\clSmajII}{\ensuremath{\clIntVal{\clSmaj}{1}{1}}}
\newcommand{\clSmajOICI}{\ensuremath{\clSmajOI \cup \clVaki}}
\newcommand{\clSmajIOCI}{\ensuremath{\clSmajIO \cup \clVaki}}
\newcommand{\clM}{\ensuremath{\mathsf{M}}}
\newcommand{\clMo}{\ensuremath{\clIntVal{\clM}{0}{}}}
\newcommand{\clMi}{\ensuremath{\clIntVal{\clM}{}{1}}}
\newcommand{\clMc}{\ensuremath{\clIntVal{\clM}{0}{1}}}
\newcommand{\clMneg}{\ensuremath{{\overline{\clM}}}}
\newcommand{\clMoneg}{\ensuremath{\clIntVal{\clMneg}{1}{}}}
\newcommand{\clMineg}{\ensuremath{\clIntVal{\clMneg}{}{0}}}
\newcommand{\clMcneg}{\ensuremath{\clIntVal{\clMneg}{1}{0}}}
\newcommand{\clUk}[1]{\ensuremath{\mathsf{U}^{#1}}}

\newcommand{\clTcUk}[1]{\ensuremath{\clIntVal{\clUk{#1}}{0}{1}}}

\newcommand{\clMUk}[1]{\ensuremath{\mathsf{M}\clUk{#1}}}
\newcommand{\clMcUk}[1]{\ensuremath{\clIntVal{\clMUk{#1}}{0}{1}}}
\newcommand{\clWk}[1]{\ensuremath{\mathsf{W}^{#1}}}

\newcommand{\clTcWk}[1]{\ensuremath{\clIntVal{\clWk{#1}}{0}{1}}}

\newcommand{\clMWk}[1]{\ensuremath{\mathsf{M}\clWk{#1}}}
\newcommand{\clMcWk}[1]{\ensuremath{\clIntVal{\clMWk{#1}}{0}{1}}}

\newcommand{\clRefl}{\ensuremath{\mathsf{R}}}  
\newcommand{\clReflOO}{\ensuremath{\clIntVal{\clRefl}{0}{0}}}
\newcommand{\clReflII}{\ensuremath{\clIntVal{\clRefl}{1}{1}}}
\newcommand{\clReflOOC}{\ensuremath{\clReflOO \cup \clVak}}
\newcommand{\clReflIIC}{\ensuremath{\clReflII \cup \clVak}}
\newcommand{\clVak}{\ensuremath{\mathsf{C}}}
\newcommand{\clVaka}[1]{\ensuremath{\clVak_{#1}}}
\newcommand{\clVakaAB}[2]{\ensuremath{\clVak^{#2}_{#1}}}
\newcommand{\clVako}{\ensuremath{\clVaka{0}}}
\newcommand{\clVaki}{\ensuremath{\clVaka{1}}}
\newcommand{\clEmpty}{\ensuremath{\mathsf{\emptyset}}}

\newcommand{\clL}{\ensuremath{\mathsf{L}}}

\newcommand{\clLo}{\ensuremath{\mathsf{L}_0}}
\newcommand{\clLi}{\ensuremath{\mathsf{L}_1}}
\newcommand{\clLc}{\ensuremath{\mathsf{L}_\mathrm{c}}}
\newcommand{\clLS}{\ensuremath{\mathsf{LS}}}
\newcommand{\clLambda}{\ensuremath{\mathsf{\Lambda}}}
\newcommand{\clV}{\ensuremath{\mathsf{V}}}
\newcommand{\clLambdac}{\ensuremath{\clIntVal{\clLambda}{0}{1}}}
\newcommand{\clVc}{\ensuremath{\clIntVal{\clV}{0}{1}}}
\newcommand{\clLambdao}{\ensuremath{\clIntVal{\clLambda}{0}{}}}
\newcommand{\clVo}{\ensuremath{\clIntVal{\clV}{0}{}}}
\newcommand{\clLambdai}{\ensuremath{\clIntVal{\clLambda}{}{1}}}
\newcommand{\clVi}{\ensuremath{\clIntVal{\clV}{}{1}}}
\newcommand{\clOmegaOne}{\ensuremath{\clAll(1)}}
\newcommand{\clIstar}{\ensuremath{\mathsf{I}^{*}}}

\newcommand{\clI}{\ensuremath{\mathsf{I}}}
\newcommand{\clIo}{\ensuremath{\mathsf{I}_0}}
\newcommand{\clIi}{\ensuremath{\mathsf{I}_1}}
\newcommand{\clIc}{\ensuremath{\mathsf{J}}}

\newcommand{\id}{\ensuremath{\mathrm{id}}}
\newcommand{\threshold}[2]{\ensuremath{\mathrm{th}^{#1}_{#2}}}

\DeclareMathOperator{\pr}{pr}

\newcommand{\Alt}{\ensuremath{\mathrm{Alt}}}
\newcommand{\clAlt}[1]{\ensuremath{\mathsf{A}^{#1}}}
\newcommand{\closys}[1]{\ensuremath{\mathcal{L}_{#1}}}
\newcommand{\clProj}[1]{\ensuremath{\mathsf{J}_{#1}}}

\DeclareMathOperator{\range}{Im}

\makeatletter
\newcommand{\displaybump}{\hbox to \@totalleftmargin{\hfil}}
\makeatother

\tikzstyle{every node}=[circle, draw, fill=black, inner sep=0pt, minimum size=6pt, scale=1, font=\normalsize]
\tikzstyle{uusi}=[fill=black]
\tikzstyle{vanha}=[fill=none]

\begin{document}
\title[Clonoids with a monotone or discriminator source clone]{Clonoids of Boolean functions with a monotone or discriminator source clone}

\author{Erkko Lehtonen}

\address{%
    Department of Mathematics \\
    Khalifa University of Science and Technology \\
    P.O. Box 127788 \\
    Abu Dhabi \\
    United Arab Emirates
}

\date{June 28, 2024}

\begin{abstract}
Extending Sparks's theorem, we determine the cardinality of the lattice of $(C_1,C_2)$\hyp{}clonoids of Boolean functions in the cases where the target clone $C_2$ is the clone of projections. Moreover, we explicitly describe the $(C_1,C_2)$\hyp{}clonoids of Boolean functions in the cases where the source clone $C_1$ is one of the four clones of monotone functions or contains the discriminator function.
\end{abstract}

\maketitle


\section{Introduction}
\label{sec:introduction}
\numberwithin{theorem}{section}

Clones are a fundamental concept in universal algebra.
A \emph{clone} is a set of operations on a set $A$ that contains all projections and is closed under composition.
Clonoids are a generalization of clones.
For fixed clones $C_1$ (the \emph{source}) and $C_2$ (the \emph{target}) on sets $A$ and $B$, respectively, a \emph{$(C_1,C_2)$\hyp{}clonoid} is a set of functions of several arguments from $A$ to $B$ that is stable under right composition with $C_1$ and stable under left composition with $C_2$.

Given a function $f \colon A^n \to B$, the functions obtained by composing $f$ from the right by projections on $B$ are called \emph{minors} of $f$.
In other words, the minors of $f$ are those functions that are obtained from $f$ by permutation of arguments, introduction or deletion of fictitious arguments, and identification of arguments.
Classes of functions that are closed under formation of minors are called \emph{minor\hyp{}closed classes} or simply \emph{minions}.
Clones and clonoids are examples of minions.
Minions arise naturally in universal algebra as sets of operations induced by terms of height 1 on an algebra.
In the context of constraint satisfaction problems (CSPs), clones and minions have played a significant role in computational complexity analysis and classification (see the survey article by Barto et al.\ \cite{BarBulKroOpr}).

Clones, clonoids, minors, and minions have been studied from various perspectives in the past decades.
Although the terminology we use here is quite modern, these concepts have been present in the literature much earlier.
To the best of the author's knowledge, the term ``clone'' first appeared in the 1965 monograph of Cohn~\cite{Cohn}, who attributed it to Philip Hall, the term ``clonoid'' was introduced in the 2016 paper by Aichinger and Mayr \cite{AicMay}, and ``minion'' was coined by Opr\v{s}al around the year 2018 (see \cite[Definition~2.20]{BarBulKroOpr}, \cite{BulKroOpr}).
As examples of earlier work in which these concepts appear under different names, we would like to highlight Post's classification of clones on a two\hyp{}element set \cite{Post}, the paper of Rosenberg and Szendrei \cite{RosSze} where minors are called ``polymers'', and descriptions of minions and $(C_1,C_2)$\hyp{}clonoids as sets of polymorphisms of relation pairs, due to Pippenger \cite{Pippenger} and Couceiro and Foldes \cite{CouFol-2005,CouFol-2009}.
More recent studies on clonoids include the papers by Fioravanti \cite{Fioravanti-AU,Fioravanti-IJAC}, Kreinecker \cite{Kreinecker}, and Mayr and Wynne \cite{MayWyn}.

Our specific focus regarding clonoids is on systematically counting and enumerating all $(C_1,C_2)$\hyp{}clonoids.
An opportune starting point for our investigation was provided by the following remarkable result by Sparks about the cardinality of the lattice $\closys{(C_1,C_2)}$ of $(C_1,C_2)$\hyp{}clonoids, in the cases where $C_1$ is the clone of projections on a finite set $A$, denoted by $\clProj{A}$, and $C_2$ is an arbitrary clone on $\{0,1\}$.

\begin{theorem}[{Sparks~\cite[Theorem~1.3]{Sparks-2019}}]
\label{thm:Sparks}
Let $A$ be a finite set with $\card{A} > 1$, and let $B = \{0,1\}$.
Let $C$ be a clone on $B$.
Then the following statements hold.
\begin{enumerate}[label={\upshape(\roman*)}]
\item\label{thm:Sparks:finite} $\closys{(\clProj{A},C)}$ is finite if and only if $C$ contains a near\hyp{}unanimity operation.
\item\label{thm:Sparks:countable} $\closys{(\clProj{A},C)}$ is countably infinite if and only if $C$ contains a Mal'cev operation but no majority operation.
\item\label{thm:Sparks:uncountable} $\closys{(\clProj{A},C)}$ has the cardinality of the continuum if and only if $C$ contains neither a near\hyp{}unanimity operation nor a Mal'cev operation.\footnote{Recall that an operation $f \colon A^n \to A$ is a \emph{near\hyp{}unanimity operation} if it satisfies the identities $f(x, \dots, x, y, x, \dots, x) \approx x$, where the single $y$ on the left side may be at any argument position.
A ternary near\hyp{}unanimity operation is called a \emph{majority operation}.
A \emph{Mal'cev operation} is a ternary operation $f$ satisfying $f(x,x,y) \approx f(y,x,x) \approx y$.}
\end{enumerate}
\end{theorem}

In a series of earlier papers of the author's \cite{CouLeh-Lcstability,Lehtonen-SM,Lehtonen-nu}, statements \ref{thm:Sparks:finite} and \ref{thm:Sparks:countable} of Theorem~\ref{thm:Sparks} were sharpened and the $(C_1,C_2)$\hyp{}clonoids of Boolean functions were described in the cases where the clonoid lattice $\closys{(C_1,C_2)}$ is finite or countably infinite.
The $(C_1,C_2)$\hyp{}clonoids of Boolean functions were explicitly described in the cases where $C_2$ contains either a Mal'cev operation \cite{CouLeh-Lcstability} or a majority operation \cite{Lehtonen-SM}, i.e., $C_2 \supseteq \clLc$ or $C_2 \supseteq \clSM$, where $\clLc$ is the clone of idempotent linear functions, generated by the Mal'cev operation $x_1 + x_2 + x_3$, and $\clSM$ is the clone of self\hyp{}dual monotone functions, generated by the majority operation.
As for the remaining clones $C_1$ containing a near\hyp{}unanimity operations, a description was obtained in terms of ideals of the minorant\hyp{}minor poset \cite{Lehtonen-nu}.
However, this description is not as explicit as those mentioned above, because the minorant\hyp{}minor poset is not yet very well understood.

As for statement \ref{thm:Sparks:uncountable} of Theorem~\ref{thm:Sparks}, such an explicit description may be unattainable for uncountable clonoid lattices.
Even so, the theorem could still be refined in such cases by classifying arbitrary pairs $(C_1,C_2)$ of clones according to whether the clonoid lattice $\closys{(C_1,C_2)}$ is finite, countably infinite, or uncountable; note that in Theorem~\ref{thm:Sparks}, the source clone is always the clone of projections.
In this paper, we are going to take a few modest first steps towards this goal.

A starting point is provided by the following theorem that can be derived from results from a series earlier papers of the author's \cite{Lehtonen-ULM,Lehtonen-k-posets,LehNes-clique,LehSze-discriminator,LehSze-quasilinear}; see Sections~\ref{sec:review}--\ref{sec:discriminator} for details.
Note that this is somewhat orthogonal to Sparks's Theorem~\ref{thm:Sparks}; while in Theorem~\ref{thm:Sparks} the source clone $C_1$ is always the clone of projections and the target clone $C_2$ is arbitrary, in the following theorem it is the target clone $C_2$ that is fixed as the clone of projections and the source clone $C_1$ is arbitrary.

\begin{theorem}
\label{thm:C-minors}
Let $C$ be a clone on $\{0,1\}$, and let $\clIc$ be the clone of projections on $\{0,1\}$.
Then the following statements hold.
\begin{enumerate}[label={\upshape(\roman*)}]
\item\label{thm:C-minors:F}
$\closys{(C,\clIc)}$ is finite if and only if $C$ contains the discriminator function.\footnote{The \emph{discriminator function} on $A$ is the ternary operation $t$ defined by the rule
$t(x,y,z) = z$ if $x = y$, and $t(x,y,z) = x$ otherwise.}
\item\label{thm:C-minors:C}
$\closys{(C,\clIc)}$ is countably infinite if and only if $\clonegen{\mathord{\wedge}, \mathord{\vee}} \subseteq C \subseteq \clonegen{\mathord{\wedge}, \mathord{\vee}, 0, 1}$.
\item\label{thm:C-minors:U}
$\closys{(C,\clIc)}$ has the cardinality of the continuum otherwise.
\end{enumerate}
\end{theorem}

\begin{proof}
This puts together Theorems \ref{thm:clique}, \ref{thm:mon-clonoids}, and \ref{thm:disc-clonoids}.	
\end{proof}

In this paper, we are going to sharpen statements \ref{thm:C-minors:F} and \ref{thm:C-minors:C} of Theorem~\ref{thm:C-minors} and find explicit descriptions of the $(C_1,C_2)$\hyp{}clonoids when $C_1$ contains the discriminator function or $\clonegen{\mathord{\wedge}, \mathord{\vee}} \subseteq C \subseteq \clonegen{\mathord{\wedge}, \mathord{\vee}, 0, 1}$ and $C_2$ is an arbitrary clone on $\{0,1\}$.

The paper is organized as follows.

\begin{itemize}
\item
Section~\ref{sec:prel}:
We recall basic notions and some known results from the theory of minors, minions, clones, and clonoids.

\item
Section~\ref{sec:Bf}:
We define properties of Boolean functions and introduce notation for certain sets of Boolean functions that we need to present our results.

\item
Section~\ref{sec:review}:
We review the author's earlier results on a notion called ``$C$\hyp{}minor'', and we explain how such results can be translated to results about $(C_1,C_2)$\hyp{}clonoids in cases where the target clone $C_2$ is the clone $\clIc$ of projections.

\item
Section~\ref{sec:uncountable}:
Making use of those earlier results on $C$\hyp{}minors, we first consider $(C_1,C_2)$\hyp{}clonoids in cases where the source clone $C_1$ does not contain the discriminator function and is not a clone of monotone functions.
In such cases we have an uncountable infinitude of $(C_1,\clIc)$\hyp{}clonoids, and it might not be possible to explicitly describe them all.
It remains out of the scope of the current paper to delve more deeply into this situation.

\item
Section~\ref{sec:monotone}:
We describe the $(C_1,C_2)$\hyp{}clonoids in cases where the source clone $C_1$ is a clone of monotone functions.
Our key tool is the homomorphism order of $k$\hyp{}posets.

\item
Section~\ref{sec:discriminator}:
We describe the $(C_1,C_2)$\hyp{}clonoids in cases where the source clone $C_1$ contains the discriminator function.

\item
Section~\ref{sec:remarks}:
We conclude the paper with a few remarks on directions for further research.
\end{itemize}


\section{Preliminaries}
\label{sec:prel}
\subsection{General}
\label{subs:prel:general}

The set of natural numbers is $\IN = \{0, 1, 2, \ldots{}\}$, and the set of integers is $\ZZ = \{\ldots{}, -2, -1, 0, 1, 2, \ldots{}\}$.
For $a, b \in \ZZ$, the \emph{interval} from $a$ to $b$ is $[a,b] := \{ \, n \in \ZZ \mid a \leq n \leq b \, \}$.
Note that if $a > b$, then $[a,b] = \emptyset$.
We also make use of the abbreviated notation $\nset{n} := [1,n] = \{ 1, 2, \dots, n \}$.

We denote $n$\hyp{}tuples by boldface letters and their components with the corresponding italic letters.
For example, we may write $\vect{a} = (a_1, \dots, a_n)$.

\subsection{Quasi\hyp{}orders, partial orders, lattices, and closure systems}

A reflexive and transitive binary relation $\leq$ on a set $Q$ is called a \emph{quasi\hyp{}order},
and the structure $(Q,{\leq})$ is called a \emph{quasi\hyp{}ordered set} or, briefly, a \emph{qoset}.
When there is no risk of confusion, we may refer to the structure $(Q,{\leq})$ simply as $Q$.
An antisymmetric quasi\hyp{}order is called a \emph{partial order}.
A structure $(P,{\leq})$, where $\leq$ is a partial order on $P$, is called a \emph{partially ordered set} or, briefly, a \emph{poset.}

Two elements $a$ and $b$ of a quasi\hyp{}ordered set $Q$ are \emph{equivalent} if $a \leq b$ and $b \leq a$, and we write $a \equiv b$ to denote this fact.
The relation $\equiv$ is indeed an equivalence relation.
The quasi\hyp{}order $\leq$ induces a partial order, also denoted by $\leq$, on the set $Q / {\equiv}$ of $\equiv$\hyp{}equivalence classes: $a / {\equiv} \leq b / {\equiv}$ if and only if $a \leq b$.

Let $(Q,{\leq})$ be a quasi\hyp{}ordered set.
A subset $D$ of $Q$ is called a \emph{downset}, or an \emph{initial segment}, if for all $d \in D$ and $q \in Q$, $q \leq d$ implies $q \in D$.
For any subset $S$, the downset generated by $S$, denoted by ${\downarrow} S$, is the least downset that includes $S$, i.e., ${\downarrow} S = \{ \, q \in Q \mid \exists s \in S \,\, q \leq s \, \}$.
When $S$ is a singleton $\{a\}$, we write simply ${\downarrow} a$ for ${\downarrow} \{a\}$.
The notion of \emph{upset} (or \emph{final segment}) is defined dually, and the upset generated by $S$ is denoted by ${\uparrow} S$.

For subsets $A, B \subseteq Q$, the \emph{interval} between $A$ and $B$ is $[A,B] := {\uparrow} A \cap {\downarrow} B = \{ \, x \in Q \mid \exists a \in A \, \exists b \in B \, ( a \leq x \leq b ) \, \}$.
Again, if $A$ or $B$ is a singleton, we may simplify the notation and write, for example $[a,b]$ for $[\{a\},\{b\}]$.
Note that this is consistent with the notation for integer intervals introduced in Subsection~\ref{subs:prel:general}; in that case we consider intervals in the poset of integers with their natural order.

An element $a$ of a poset $(P,{\leq})$ is the \emph{least element} of $P$ if $a \leq x$ for all $x \in P$, and, dually, $a$ is the \emph{greatest element} of $P$ if $x \leq a$ for all $x \in P$.
An element $b \in P$ is an \emph{upper bound} of a subset $S \subseteq P$ if $s \leq b$ for all $s \in S$, and $b$ is a \emph{lower bound} of $S$ if $b \leq s$ for all $s \in S$.
A poset $P$ is called a \emph{lattice} if every two\hyp{}element subset $\{a, b\}$ of $P$ has a least upper bound and a greatest lower bound (i.e., the set of upper bounds of $\{a,b\}$ has a least element and the set of lower bounds of $\{a,b\}$ has a greatest element).
A lattice is \emph{complete} if every one of its subsets has a least upper bound and a greatest lower bound.

A collection $\mathcal{S}$ of subsets of a set $A$ is called a \emph{closure system} on $A$ if $A \in \mathcal{S}$ and for all $\mathcal{T} \subseteq \mathcal{S}$, $\bigcap \mathcal{T} \in \mathcal{S}$.
The subsets of $A$ belonging to a closure system $\mathcal{S}$ are called the \emph{closed sets} of $\mathcal{S}$.
A function $c \colon \mathcal{P}(A) \to \mathcal{P}(A)$ is called a \emph{closure operator} on $A$ if it is \emph{extensive} ($X \subseteq c(X)$), \emph{monotone} ($X \subseteq Y$ implies $c(X) \subseteq c(Y)$) and \emph{idempotent} ($c(c(X)) = c(X)$).
The sets of the form $c(X)$ for some $X \subseteq A$ are called the \emph{closed sets} of $c$.
We say that $c(X)$ is the \emph{closure} of $X$ and that $X$ \emph{generates} $c(X)$ or $X$ is a \emph{generating set} of $c(X)$.
The closed sets of $c$ form a closure system on $A$.
Conversely, given a closure system $\mathcal{S}$ on $A$, the mapping $c_\mathcal{S} \colon \mathcal{P}(A) \to \mathcal{P}(A)$, $c_\mathcal{S}(X) = \bigcap \{ \, Y \in \mathcal{S} \mid X \subseteq Y \, \}$ is a closure operator whose closed sets are precisely the members of $\mathcal{S}$.
A closure system, ordered by subset inclusion, constitutes a complete lattice.

\subsection{Function class composition, minors, and minions}

Let $A$ and $B$ be sets.
A mapping of the form $f \colon A^n \to B$ for some $n \in \IN_{+}$ is called a \emph{function of several arguments} from $A$ to $B$, or, briefly, a \emph{function} from $A$ to $B$.
The number $n$ is called the \emph{arity} of $f$.
When $A = B$, we call such a function an \emph{operation} on $A$.
We denote by $\mathcal{F}_{AB}^{(n)}$ the set of all $n$\hyp{}ary functions from $A$ to $B$, i.e., $\mathcal{F}_{AB} = B^{A^n}$, and we let $\mathcal{F}_{AB} := \bigcup_{n \in \IN_{+}} \mathcal{F}_{AB}^{(n)}$ be the set of all finitary functions from $A$ to $B$.
Similarly, we denote by $\mathcal{O}_A^{(n)}$ the set of all $n$\hyp{}ary operations on $A$, and we let $\mathcal{O}_A := \bigcup_{n \in \IN_{+}} \mathcal{O}_A^{(n)}$.
For any set $K \subseteq \mathcal{F}_{AB}$ and $n \in \IN_{+}$, the \emph{$n$\hyp{}ary part} of $K$ is $K^{(n)} := K \cap \mathcal{F}_{AB}^{(n)}$.

The $i$\hyp{}th $n$\hyp{}ary \emph{projection} on $A$ is the operation $\pr_i^{(n)} \colon A^n \to A$ defined by $\pr_i^{(n)}(a_1, \dots, a_n) = a_i$.
We denote by $\clProj{A}$ the set of all projections on $A$.

Let $f \in \mathcal{F}_{AB}^{(n)}$.
The $i$\hyp{}th argument is \emph{essential} in $f$ if there exist $a_1, \dots, a_n, a'_i \in A$ such that $f(a_1, \dots, a_{i-1}, a_i, a_{i+1}, \dots, a_n) \neq f(a_1, \dots, a_{i-1}, a'_i, a_{i+1}, \dots, a_n)$.
An argument that is not essential is \emph{fictitious.}
The number of essential arguments in $f$ is called the \emph{essential arity} of $f$.
We say that a set $F \subseteq \mathcal{F}_{AB}$ of functions is \emph{essentially at most unary} if every member of $F$ is essentially at most unary.

For $f \in \mathcal{F}_{BC}^{(n)}$ (the \emph{outer function}) and $g_1, \dots, g_n \in \mathcal{F}_{AB}^{(m)}$ (the \emph{inner functions}), the composition of $f$ with $g_1, \dots, g_n$ is the function $f(g_1, \dots, g_n) \in \mathcal{F}_{AC}^{(m)}$ defined by the rule
$f(g_1, \dots, g_n)(\vect{a}) := f(g_1(\vect{a}), \dots, g_n(\vect{a}))$ for all $\vect{a} \in A^n$.

The concept of functional composition can be extended to sets of functions as follows.
If $I \subseteq \mathcal{F}_{BC}$ and $J \subseteq \mathcal{F}_{AB}$, then the \emph{composition} of $I$ with $J$ is
\[
I J := \{ \, f(g_1, \dots, g_n) \mid m \in \IN_{+}, \, n \in \IN_{+} \, f \in I^{(n)}, \, g_1, \dots, g_n \in J^{(m)} \, \}.
\]

Function class composition allows us to define many useful concepts in a compact way.
The first such definition is that of minors and minions.
For $f \in \mathcal{F}_{AB}^{(m)}$ and $g \in \mathcal{F}_{AB}^{(n)}$, we say that $f$ is a \emph{minor} of $g$ if $f \in \{g\} \clProj{A}$, or, equivalently, if there exists a map $\sigma \colon \nset{n} \to \nset{m}$ such that $f = g(\pr_{\sigma(1)}^{(m)}, \dots, \pr_{\sigma(n)}^{(m)})$.
In other words, the minors of $g$ are those functions that can be obtained from $g$ by permutation of arguments, introduction or deletion of fictitious arguments, and identification of arguments.
We say that a class $K \subseteq \mathcal{F}_{AB}$ is \emph{minor\hyp{}closed} or that it is a \emph{minion} if $K \, \clProj{A} \subseteq K$.

Let us review some basic properties of function class composition.
Function class composition is monotone, that is,
if $I, I' \subseteq \mathcal{F}_{BC}$ and $J, J' \subseteq \mathcal{F}_{AB}$ satisfy $I \subseteq I'$ and $J \subseteq J'$, then $I J \subseteq I' J'$.
Function class composition is not associative.
Associativity nevertheless holds for triples satisfying special conditions.

\begin{lemma}[{Couceiro, Foldes~\cite[Associativity Lemma]{CouFol-2007,CouFol-2009}}]
\label{lem:CF-associativity}
Let $A$, $B$, $C$, and $D$ be arbitrary nonempty sets, and let $I \subseteq \mathcal{F}_{CD}$, $J \subseteq \mathcal{F}_{BC}$, $K \subseteq \mathcal{F}_{AB}$.
Then the following statements hold.
\begin{enumerate}[label=\upshape{(\roman*)}]
\item $(IJ)K \subseteq I(JK)$.
\item If $J$ is minor\hyp{}closed, then $(IJ)K = I(JK)$.
\end{enumerate}
\end{lemma}

\begin{lemma}[Unary Associativity Lemma]
\label{lem:UAL}
Let $A$, $B$, $C$, and $D$ be arbitrary nonempty sets, and let $I \subseteq \mathcal{F}_{CD}$, $J \subseteq \mathcal{F}_{BC}$, $K \subseteq \mathcal{F}_{AB}$.
If $I$ is essentially at most unary and contains the unary minors of its members, then $(IJ)K = I(JK)$.
\end{lemma}

\begin{proof}
By Lemma~\ref{lem:CF-associativity}, $(IJ)K \subseteq I(JK)$, so it remains to show that $I(JK) \subseteq (IJ)K$.
Let $f \in I(JK)$.
Then $f = \varphi ( \gamma_1, \dots, \gamma_n )$ for some $\varphi \in I$ and $\gamma_1, \dots, \gamma_n \in JK$.
Because $\varphi$ is essentially at most unary, $\varphi = \varphi' \circ \pr^{(n)}_i$ for some $i \in \nset{n}$ and for the unary minor $\varphi' := \varphi( \pr^{(1)}_1, \dots, \pr^{(1)}_1)$ of $\varphi$; note that $\varphi' \in I$ because $I$ contains the unary minors of its members.
Thus $f = (\varphi' \circ \pr^{(n)}_i) (\gamma_1, \dots, \gamma_n) = \varphi' \circ (\pr^{(n)}_i (\gamma_1, \dots, \gamma_n)) = \varphi' \circ \gamma_i$.
Because $\gamma_i \in JK$, we have $\gamma_i = \lambda ( \kappa_1, \dots, \kappa_m )$ for some $\lambda \in J$ and $\kappa_1, \dots, \kappa_m \in K$, so
\[
f
= \varphi' ( \lambda ( \kappa_1, \dots, \kappa_m ) )
= ( \varphi' \circ \lambda  ) ( \kappa_1, \dots, \kappa_m )
\in (IJ)K.
\qedhere
\]
\end{proof}

Function class composition is right\hyp{}distributive over arbitrary unions.
Left\hyp{}distributivity does not hold in general (not even for finite unions), but we would like to highlight a particular case where left\hyp{}distributivity still holds.

\begin{lemma}
\label{lem:CompUnion}
Let $A$, $B$, and $C$ be arbitrary nonempty sets.
\begin{enumerate}[label=\upshape{(\roman*)}]
\item\label{lem:CompUnion:L}
Let $F_i \subseteq \mathcal{F}_{BC}$ \textup{(}$i \in I$\textup{)} and $G \subseteq \mathcal{F}_{AB}$.
Then $(\bigcup_{i \in I} F_i) G = \bigcup_{i \in I} (F_i G)$.
\item\label{lem:CompUnion:R}
Let $F \subseteq \mathcal{F}_{BC}$ and $G_j \subseteq \mathcal{F}_{AB}$ \textup{(}$j \in J$\textup{)}.
If $F$ is essentially at most unary, then $F ( \bigcup_{j \in J} G_j ) = \bigcup_{j \in J} (F G_j)$.
\end{enumerate}
\end{lemma}

\begin{proof}
\ref{lem:CompUnion:L}
By the monotonicity of function class composition, we have $\bigcup_{i \in I} (F_i G) \subseteq \bigcup_{i \in I} ( ( \bigcup_{j \in I} F_j) G ) = ( \bigcup_{j \in I} F_j) G = ( \bigcup_{i \in I} F_i) G $.
In order to prove the converse inclusion, let $f \in ( \bigcup_{i \in I} F_i) G $.
Then $f = \varphi ( \gamma_1, \dots, \gamma_n )$ for some $\varphi \in \bigcup_{i \in I} F_i$ and $\gamma_1, \dots, \gamma_n \in G$.
Then $\varphi \in F_j$ for some $j \in I$, and it follows that $f \in F_j G$; thus $f \in \bigcup_{i \in I} (F_i G)$.

\ref{lem:CompUnion:R}
By the monotonicity of function class composition, we have $\bigcup_{i \in I} (F G_i) \subseteq \bigcup_{i \in I} (F ( \bigcup_{j \in I} G_j )) = F ( \bigcup_{j \in I} G_j ) = F ( \bigcup_{i \in I} G_i )$.
In order to prove the converse inclusion, let $f \in F ( \bigcup_{i \in I} G_i )$.
Then $f = \varphi ( \gamma_1, \dots, \gamma_n )$ for some $\varphi \in F$ and $\gamma_1, \dots, \gamma_n \in \bigcup_{i \in I} G_i$.
Because $\varphi$ is essentially at most unary, we have $\varphi = \varphi' \circ \pr^{(n)}_j$ for some $j \in \nset{n}$ and $\varphi' = \varphi(\pr^{(1)}_1, \dots, \pr^{(1)}_1)$.
Then
\begin{align*}
f
&
= \varphi(\gamma_1, \dots, \gamma_n)
= (\varphi' \circ \pr^{(n)}_j)(\gamma_1, \dots, \gamma_n)
= \varphi' ( \pr^{(n)}_j (\gamma_1, \dots, \gamma_n))
= \varphi'(\gamma_j)
\\ &
= \varphi' ( \pr^{(n)}_j (\gamma_j, \dots, \gamma_j))
= (\varphi' \circ \pr^{(n)}_j)(\gamma_j, \dots, \gamma_j)
= \varphi(\gamma_j, \dots, \gamma_j)
\in F G_k
\end{align*}
for some $k \in I$.
Consequently, $f \in \bigcup_{i \in I} (F G_i)$.
\end{proof}

\subsection{Clones}

A set of operations on $A$ is a \emph{clone} on $A$ if it contains all projections and is closed under composition.
In symbols, $C \subseteq \mathcal{O}_A$ is a \emph{clone} on $A$ if $\clProj{A} \subseteq C$ and $C C \subseteq C$.

The clones on $A$ constitute a closure system on $\mathcal{O}_A$.
For $F \subseteq \mathcal{O}_A$, we denote by $\clonegen{F}$ the clone generated by $F$, i.e., the smallest clone on $A$ that includes $F$.

The main open problem in clone theory is to characterize all clones.
The situation is entirely trivial on a one\hyp{}element base set; there is just one such clone, the clone of all operations.
The clones on a two\hyp{}element set are well known.
There are countably infinitely many such clones, and they were characterized by Post~\cite{Post}; see Section~\ref{sec:Bf}, Figure~\ref{fig:Post}.
It is known that on finite sets with at least three elements, there are an uncountably infinitude of clones, but a complete description eludes us.

\subsection{Clonoids}

Let $C_1$ and $C_2$ be clones on sets $A$ and $B$, respectively.
A set $K \subseteq \mathcal{F}_{AB}$ is \emph{stable under right composition} with $C_1$ if $K C_1 \subseteq K$,
and $K$ is \emph{stable under left composition} with $C_2$ if $C_2 K \subseteq K$.
We say that $K$ is \emph{$(C_1,C_2)$\hyp{}stable} or that $K$ is a \emph{$(C_1,C_2)$\hyp{}clonoid} if $K$ is stable under right composition with $C_1$ and stable under left composition with $C_2$.
We refer to $C_1$ and $C_2$ as the \emph{source clone} and the \emph{target clone} of $K$, respectively.
We denote by $\closys{(C_1,C_2)}$ the set of all $(C_1,C_2)$\hyp{}clonoids.

For fixed source and target clones $C_1$ and $C_2$, the $(C_1,C_2)$\hyp{}clonoids constitute a closure system on $\mathcal{F}_{AB}$.
For $F \subseteq \mathcal{F}_{AB}$, we denote by $\gen[(C_1,C_2)]{F}$ the $(C_1,C_2)$\hyp{}clonoid generated by $F$, i.e., the smallest $(C_1,C_2)$\hyp{}clonoid that includes $F$.

\begin{lemma}[{\cite[Lemma~2.5]{Lehtonen-SM}}]
\label{lem:gen}
Let $F \subseteq \mathcal{F}_{AB}$, and let $C_1$ and $C_2$ be clones on $A$ and $B$, respectively.
Then
$\gen[(C_1,C_2)]{F} = C_2 ( F C_1 )$.
\end{lemma}

The monotonicity of function class composition gives the following relationship between clonoids when we have an inclusion between source clones and between target clones.

\begin{lemma}[{\cite[Lemma~2.16]{CouLeh-Lcstability}}]
\label{lem:clonoid-inclusion}
Let $C_1$ and $C'_1$ be clones on $A$ and $C_2$ and $C'_2$ clones on $B$ such that $C_1 \subseteq C'_2$ and $C_2 \subseteq C'_2$.
Then every $(C'_1,C'_2)$\hyp{}clonoid is a $(C_1,C_2)$\hyp{}clonoid.
\end{lemma}

Moreover, for every minion $K$, there are largest clones with which $K$ is stable under left and right composition.

\begin{lemma}[{\cite[Proposition~2.8]{Lehtonen-SM}}]
\label{lem:CAKCBK}
For any minion $K \subseteq \mathcal{F}_{AB}$, there exist clones $C_A^K$ on $A$ and $C_B^K$ on $B$ such that for all clones $C_1$ on $A$ and $C_2$ on $B$, it holds that $K$ is $(C_1,C_2)$\hyp{}stable if and only if $C_1 \subseteq C_A^K$ and $C_2 \subseteq C_B^K$.
\end{lemma}

The union of $(C_1,C_2)$\hyp{}clonoids is not in general a $(C_1,C_2)$\hyp{}clonoid.
However, if the target clone is essentially at most unary, then $\closys{(C_1,C_2)}$ is closed under unions.

\begin{lemma}
\label{lem:C-essun}
Let $F, G \subseteq \mathcal{F}_{AB}$, and let $C_1$ and $C_2$ be clones on $A$ and $B$, respectively.
If $C_2$ is essentially at most unary, then $\gen[(C_1,C_2)]{F \cup G} = \gen[(C_1,C_2)]{F} \cup \gen[(C_1,C_2)]{G}$.
\end{lemma}

\begin{proof}
By Lemma~\ref{lem:gen}, we have
\begin{align*}
\gen[(C_1,C_2)]{F \cup G}
& = C_2 ( ( F \cup G ) C_1
= C_2 ( F C_1 \cup G C_1 )
\\
& = C_2 ( F C_1 ) \cup C_2 ( G C_1 )
= \gen[(C_1,C_2)]{F} \cup \gen[(C_1,C_2)]{G},
\end{align*}
where the second and the third equalities hold by Lemma~\ref{lem:CompUnion}, statements \ref{lem:CompUnion:L} and \ref{lem:CompUnion:R}, respectively.
\end{proof}

In order to test stability under left or right composition with a clone, we only need to consider a generating set of that clone.
For this, we need to use the binary composition operation $\ast$ on $\mathcal{O}_A$ that is defined as follows.
For $f \in \mathcal{O}_A^{(m)}$ and $g \in \mathcal{O}_A^{(n)}$, $f \ast g \in \mathcal{O}_A^{(m+n-1)}$ is defined by the rule
\[
(f \ast g)(a_1, \dots, a_{m+n-1}) = f(g(a_1, \dots, a_n), a_{n+1}, \dots, a_{m+n-1}).
\]

\begin{lemma}[{\cite[Lemma~3.2]{CouLeh-Lcstability}}]
\label{lem:CL-Lc-3.2}
Let $F \subseteq \mathcal{O}_A$.
Let $C$ be a clone on $A$, and let $G$ be a generating set of $C$.
Then the following conditions are equivalent.
\begin{enumerate}[label=\upshape{(\roman*)}]
\item $F C \subseteq F$.
\item $F$ is minor\hyp{}closed and $f \ast g \in F$ whenever $f \in F$ and $g \in C$.
\item $F$ is minor\hyp{}closed and $f \ast g \in F$ whenever $f \in F$ and $g \in G$.
\end{enumerate}
\end{lemma}

\begin{lemma}[{\cite[Lemma~3.3]{CouLeh-Lcstability}}]
\label{lem:CL-Lc-3.3}
Let $F \subseteq \mathcal{O}_A$.
Let $C$ be a clone on $A$, and let $G$ be a generating set of $C$.
Then the following conditions are equivalent.
\begin{enumerate}[label=\upshape{(\roman*)}]
\item $C F \subseteq F$.
\item $g(f_1, \dots, f_n) \in F$ whenever $g \in C^{(n)}$ and $f_1, \dots, f_n \in F^{(m)}$ for some $n, m \in \IN$.
\item $g(f_1, \dots, f_n) \in F$ whenever $g \in G^{(n)}$ and $f_1, \dots, f_n \in F^{(m)}$ for some $n, m \in \IN$.
\end{enumerate}
\end{lemma}

For $c \in B$, let $\clVakaAB{c}{AB}$ be the set of all constant functions in $\mathcal{F}_{AB}$ taking value $c$.
For a subset $S \subseteq A$, let $\clVakaAB{S} := \bigcup_{c \in S} \clVakaAB{c}{AB}$.
If $A = B$, we write simply $\clVakaAB{c}{A}$ and $\clVakaAB{S}{A}$ for $\clVakaAB{c}{AA}$ and $\clVakaAB{S}{AA}$, respectively, or we may simply omit the superscripts if the sets $A$ and $B$ are clear from the context.

\begin{lemma}
\label{lem:C1C2Vak}
Let $C_1$ and $C_2$ be clones on $A$ and $B$, respectively, and let $S \subseteq B$.
Assume that $C_2 \cup \clVakaAB{S}{B}$ is a clone on $B$.
\begin{enumerate}[label=\upshape{(\roman*)}]
\item\label{lem:C1C2Vak:i}
If $F \subseteq \mathcal{F}_{AB}$ is a $(C_1,C_2)$\hyp{}clonoid,
then
$F \cup \clVakaAB{S}{AB}$ is a $(C_1, C_2 \cup \clVakaAB{S}{B})$\hyp{}clonoid.
\item\label{lem:C1C2Vak:ii}
The nonempty $(C_1, C_2 \cup \clVakaAB{S}{B})$\hyp{}clonoids are precisely the $(C_1,C_2)$\hyp{}clonoids $K$ satisfying $\clVakaAB{S}{AB} \subseteq K$.
\end{enumerate}
\end{lemma}

\begin{proof}
\ref{lem:C1C2Vak:i}
Because $F$ is a $(C_1,C_2)$\hyp{}clonoid, we have $F = \gen{(C_1,C_2)} =  C_2 ( F C_1 )$.
Note also that, because $C_2$ and $C_2 \cup \clVakaAB{S}{B}$ are assumed to be clones on $B$, it holds that
\begin{equation}
C_2 ( \clProj{B} \cup \clVakaAB{S}{B} )
\subseteq ( C_2 \cup \clVakaAB{S}{B} ) ( C_2 \cup \clVakaAB{S}{B} ) 
\subseteq C_2 \cup \clVakaAB{S}{B}.
\label{eq:C2CS}
\end{equation}
By applying Lemmata~\ref{lem:CF-associativity} and \ref{lem:CompUnion}, we get
\begin{align*}
\lhs
F \cup \clVakaAB{S}{AB}
\subseteq 
\clProj{B} ((F \cup \clVakaAB{S}{AB}) \clProj{A} )
\subseteq
(C_2 \cup \clVakaAB{S}{B}) ((F \cup \clVakaAB{S}{AB}) C_1)
\\ &
\stackrel{\text{\ref{lem:CompUnion}}}{=} 
C_2 (F C_1  \cup \clVakaAB{S}{AB} C_1) \cup \clVakaAB{S}{B} ((F \cup \clVakaAB{S}{AB}) C_1) 
=
C_2 (F C_1 \cup \clVakaAB{S}{AB}) \cup \clVakaAB{S}{AB}
\\ &
=
C_2 (\clProj{B} (F C_1) \cup \clVakaAB{S}{B} (F C_1)) \cup \clVakaAB{S}{AB}
\stackrel{\text{\ref{lem:CompUnion}}}{=} 
C_2 (( \clProj{B} \cup \clVakaAB{S}{B} )(F C_1)) \cup \clVakaAB{S}{AB}
\\ &
\stackrel{\text{\ref{lem:CF-associativity}}}{=} 
(C_2 ( \clProj{B} \cup \clVakaAB{S}{B} ))(F C_1)) \cup \clVakaAB{S}{AB}
\stackrel{\text{\eqref{eq:C2CS}}}{\subseteq}
(C_2 \cup \clVakaAB{S}{B} )(F C_1)) \cup \clVakaAB{S}{AB}
\\ &
\stackrel{\text{\ref{lem:CompUnion}}}{=} 
(C_2 (F C_1) \cup \clVakaAB{S}{B} (F C_1)) \cup \clVakaAB{S}{AB}
=
( F \cup \clVakaAB{S}{AB} ) \cup \clVakaAB{S}{AB}
=
F \cup \clVakaAB{S}{AB}.
\end{align*}
This shows that $(C_2 \cup \clVakaAB{S}{B})((F \cup \clVakaAB{S}{AB}) C_1) = F \cup \clVakaAB{S}{AB}$, i.e., $F \cup \clVakaAB{S}{AB}$ is a $(C_1, C_2 \cup \clVakaAB{S}{B})$\hyp{}clonoid.

\ref{lem:C1C2Vak:ii}
If $K$ is a $(C_1,C_2)$\hyp{}clonoid satisfying $\clVakaAB{S}{AB} \subseteq K$, then, by \ref{lem:C1C2Vak:i}, $K \cup \clVakaAB{S}{AB} = K$ is a $(C_1, C_2 \cup \clVakaAB{S}{B})$\hyp{}clonoid.

Assume now that $K$ is a nonempty $(C_1, C_2 \cup \clVakaAB{S}{B})$\hyp{}clonoid.
Then
\[
K \supseteq (C_2 \cup \clVakaAB{S}{B}) K = C_2 K \cup \clVakaAB{S}{B} K = C_2 K \cup \clVakaAB{S}{AB} = K \cup \clVakaAB{S}{AB} \supseteq K,
\]
so $K = K \cup \clVakaAB{S}{AB}$, which is equivalent to $\clVakaAB{S}{AB} \subseteq K$.
\end{proof}


\section{Classes of Boolean functions}
\label{sec:Bf}

Operations on $\{0,1\}$ are called \emph{Boolean functions.}
In this section, we are going to define properties of Boolean functions and introduce notation for the clones and other classes of Boolean functions that are relevant to our work.

We are going to make use of some well\hyp{}known Boolean functions that are defined by the operation tables in Figure~\ref{fig:Bf}:
the constant functions $0$ and $1$,
identity $\id$,
negation $\neg$,
conjunction $\wedge$,
disjunction $\vee$,
addition $+$,
implication $\rightarrow$,
biconditional $\leftrightarrow$,
majority $\mu$,
triple sum $+_3$.
We also need the function
$\veewedge \colon \{0,1\}^3 \to \{0,1\}$, $\veewedge(a_1,a_2,a_3) = a_1 \vee (a_2 \wedge a_3)$,
as well as the \emph{threshold functions} $\threshold{n}{k} \colon \{0,1\}^n \to \{0,1\}$,
\[
\threshold{n}{k}(\vect{a}) =
\begin{cases}
1, & \text{if $\card{\{ \, i \mid a_i = 1 \, \}}$,} \\
0, & \text{otherwise.}
\end{cases}
\]

\begin{figure}
\newcommand{\FIGunary}{$\begin{array}[t]{c|cccc}
x_1 & 0 & 1 & \id & \neg \\
\hline
0 & 0 & 1 & 0 & 1 \\
1 & 0 & 1 & 1 & 0
\end{array}$}
\newcommand{\FIGbinary}{$\begin{array}[t]{cc|ccccc}
x_1 & x_2 & \wedge & \vee & + & \rightarrow & \leftrightarrow \\
\hline
0 & 0 & 0 & 0 & 0 & 1 & 1 \\
0 & 1 & 0 & 1 & 1 & 1 & 0 \\
1 & 0 & 0 & 1 & 1 & 0 & 0 \\
1 & 1 & 1 & 1 & 0 & 1 & 1
\end{array}$}
\newcommand{\FIGternary}{$\begin{array}[t]{ccc|ccc}
x_1 & x_2 & x_3 & \mu & \mathord{+_3} & \veewedge \\
\hline
0 & 0 & 0 & 0 & 0 & 0 \\
0 & 0 & 1 & 0 & 1 & 0 \\
0 & 1 & 0 & 0 & 1 & 0 \\
0 & 1 & 1 & 1 & 0 & 1 \\
1 & 0 & 0 & 0 & 1 & 1 \\
1 & 0 & 1 & 1 & 0 & 1 \\
1 & 1 & 0 & 1 & 0 & 1 \\
1 & 1 & 1 & 1 & 1 & 1
\end{array}$}
\newlength{\FIGternaryheight}
\settototalheight{\FIGternaryheight}{\FIGternary}
\newlength{\FIGbinarywidth}
\settowidth{\FIGbinarywidth}{\FIGbinary}
\begin{minipage}[t][\FIGternaryheight][t]{\FIGbinarywidth}
\FIGunary
\vfill
\FIGbinary
\end{minipage}
\qquad\qquad
\FIGternary

\caption{Some Boolean functions, some of which well known.}
\label{fig:Bf}
\end{figure}

The set of all Boolean functions is denoted by $\clAll$.

For any set $K \subseteq \clAll$ and $a, b \in \{0,1\}$, we let
\begin{align*}
\clIntVal{K}{a}{} &:= \{ \, f \in K \mid f(\vect{0}) = a \, \}, \\
\clIntVal{K}{}{b} &:= \{ \, f \in K \mid f(\vect{1}) = b \, \}, \\
\clIntVal{K}{a}{b} &:= \clIntVal{K}{a}{} \cap \clIntVal{K}{}{b}.
\end{align*}
We introduce some shorthands:
\begin{align*}
\clEq &:= \{ \, f \in \clAll \mid f(\vect{0}) = f(\vect{1}) \, \} = \clOO \cup \clII, \\
\clNeq &:= \{ \, f \in \clAll \mid f(\vect{0}) \neq f(\vect{1}) \, \} = \clOI \cup \clIO, \\
\clEiio &:= \{ \, f \in \clAll \mid f(\vect{0}) \leq f(\vect{1}) \, \} = \clOO \cup \clOI \cup \clII, \\
\clEioi &:= \{ \, f \in \clAll \mid f(\vect{0}) \geq f(\vect{1}) \, \} = \clOO \cup \clIO \cup \clII, \\
\clEiii &:= \clOI \cup \clOI \cup \clOO, \\
\clEioo &:= \clOI \cup \clIO \cup \clII.
\end{align*}

For $a \in \{0,1\}$, we often write $\overline{a}$ for $\neg(a)$, i.e., $\overline{0} = 1$ and $\overline{1} = 0$.
We extend this to tuples: for $\vect{a} = (a_1, \dots, a_n) \in \{0,1\}^n$, we write $\overline{\vect{a}}$ for $(\overline{a_1}, \dots, \overline{a_n})$.
For any $K \subseteq \clAll$, let $\overline{K} := \{\neg\} K = \{ \, \neg (f) \mid f \in K \, \}$.

Let $\clVak$ be the set of all constant functions, and for $a \in \{0,1\}$, let $\clVaka{a}$ be the set of all constant functions taking value $a$.
Let $\clIc$ be the set of all projections, and let
\begin{align*}
\clIo &:= \clIc \cup \clVako, &
\clIi &:= \clIc \cup \clVaki, &
\clI &:= \clIc \cup \clVak, &
\clIstar &:= \clIc \cup \overline{\clIc}, &
\clOmegaOne &:= \clIstar \cup \clVak.
\end{align*}

A Boolean function $f \colon \{0,1\}^n \to \{0,1\}$ is \emph{monotone} if $f(\vect{a}) \leq f(\vect{b})$ whenever $\vect{a} \leq \vect{b}$.
We denote by $\clM$ the set of all monotone functions.

We denote by $\clL$ the polynomial operations of the group $(\{0,1\}, \mathord{+})$,
by $\clV$ the polynomial operations of the join\hyp{}semilattice $(\{0,1\}, \mathord{\vee})$,
and by $\clLambda$ the polynomial operations of the meet\hyp{}semilattice $(\{0,1\}, \mathord{\wedge})$.

The \emph{dual} of a Boolean function $f \colon \{0,1\}^n \to \{0,1\}$ is $f^\mathrm{d} \colon \{0,1\}^n \to \{0,1\}$ defined by
$f^\mathrm{d}(\vect{a}) = \neg(f(\neg(a_1), \dots, \neg(a_n)))$.
The \emph{dual} of a set $F \subseteq \clAll$ is $F^\mathrm{d} := \{ \, f^\mathrm{d} \mid f \in F \, \}$.
The dual of a clone is a clone.
If $K$ is a $(C_1,C_2)$\hyp{}clonoid, then $K^\mathrm{d}$ is a $(C_1^\mathrm{d},C_2^\mathrm{d})$\hyp{}clonoid.

A Boolean function $f \colon \{0,1\}^n \to \{0,1\}$ is \emph{self\hyp{}dual} if $f = f^\mathrm{d}$, or, equivalently, if for all $\vect{a} \in \{0,1\}^n$, $f(\vect{a}) \neq f(\overline{\vect{a}})$.
We denote by $\clS$ the set of all self\hyp{}dual functions.
Let $\clSM := \clS \cap \clM$, the set of all self\hyp{}dual monotone functions,
and let $\clLS := \clL \cap \clS$, the set of all self\hyp{}dual linear functions.

Let $f, g \colon \{0,1\}^n \to \{0,1\}$.
If $f(\vect{a}) \leq g(\vect{a})$ for all $\vect{a} \in \{0,1\}^n$, we say that $f$ is a \emph{minorant} of $g$ or that $g$ is a \emph{majorant} of $f$.
We denote by $\clSmin$ the set of all minorants of self\hyp{}dual functions and by $\clSmaj$ the set of all majorants of self\hyp{}dual functions.

A Boolean function $f \colon \{0,1\}^n \to \{0,1\}$ is \emph{reflexive} if for all $\vect{a} \in \{0,1\}^n$, $f(\vect{a}) = f(\overline{\vect{a}})$.
We denote by $\clRefl$ the set of all reflexive functions.

For $a \in \{0,1\}$, a set $S \subseteq \{0,1\}^n$ is \emph{$a$\hyp{}separating} if there is an index $i \in \nset{n}$ such that for every $(a_1, \dots, a_n) \in S$ we have $a_i = a$.
In other words, $S$ is $0$\hyp{}separating if $\bigvee S \neq \vect{1}$, and $S$ is $1$\hyp{}separating if $\bigwedge S \neq \vect{0}$.
For $m \geq 2$, a function $f \colon \{0,1\}^n \to \{0,1\}$ is \emph{$a$\hyp{}separating of rank $m$} if every subset of $f^{-1}(a)$ of cardinality at most $m$ is $a$\hyp{}separating.
A function $f$ is \emph{$a$\hyp{}separating} if $f^{-1}(a)$ is $a$\hyp{}separating.
We denote by $\clWk{m}$ and $\clUk{m}$ the sets of $0$- and $1$\hyp{}separating functions of rank $m$, respectively, and by $\clWk{\infty}$ and $\clUk{\infty}$ the sets of $0$- and $1$\hyp{}separating functions, respectively.
For $2 \leq m \leq \infty$, let $\clMWk{m} := \clM \cap \clWk{m}$ and $\clMUk{m} := \clM \cap \clUk{m}$.

The clones on $\{0,1\}$ were described by Post~\cite{Post}.
The lattice of clones of Boolean functions, also known as \emph{Post's lattice,} is presented in Figure~\ref{fig:Post}.
Many (but not all) of the classes of Boolean functions defined above are clones.
In what follows, we make use of the following generating sets for some of the clones:
\begin{align*}
& \clonegen{0} = \clVako, &
& \clonegen{1} = \clVaki, &
& \clonegen{0,1} = \clVak, &
& \clonegen{\neg} = \clIstar, \\
& \clonegen{\vee} = \clVc, &
& \clonegen{\vee, 0} = \clVo, &
& \clonegen{\vee, 1} = \clVi, &
& \clonegen{\vee, 0, 1} = \clV, \\
& \clonegen{\mathord{\vee}, \mathord{\wedge}} = \clMc, &
& \clonegen{\mathord{\vee}, \mathord{\wedge}, 0} = \clMo, &
& \clonegen{\mathord{\vee}, \mathord{\wedge}, 1} = \clMi, &
& \clonegen{\mathord{\vee}, \mathord{\wedge}, 0, 1} = \clM, \\
& \clonegen{+_3} = \clLc, &
& \clonegen{\mathord{\rightarrow}, \threshold{4}{2}} = \clWk{3}, &
& \clonegen{\mu, \veewedge} = \clMcWk{2}, &
& \clonegen{\veewedge} = \clMcWk{\infty}.
\end{align*}

\begin{figure}
\begin{center}
\scalebox{0.375}{%
\tikzstyle{every node}=[circle, draw, fill=black, scale=1, font=\LARGE]
\begin{tikzpicture}[baseline, scale=1]
   \node [label = below:$\clIc$] (Ic) at (0,-1) {};
   \node [label = left:$\clIstar$] (Istar) at (0,0.5) {};
   \node [label = below right:$\clIo$] (I0) at (4.5,0.5) {};
   \node [label = below left:$\clIi$] (I1) at (-4.5,0.5) {};
   \node [label = below:$\clI$] (I) at (0,2) {};
   \node [label = above:$\clOmegaOne$] (Omega1) at (0,5) {};
   \node [label = below:$\clLc$] (Lc) at (0,7.5) {};
   \node [label = right:$\clLS$] (LS) at (0,9) {};
   \node [label = right:$\clLo$] (L0) at (3,9) {};
   \node [label = left:$\clLi$] (L1) at (-3,9) {};
   \node [label = above:$\clL$] (L) at (0,10.5) {};
   \node [label = below:$\clSM$] (SM) at (0,13.5) {};
   \node [label = left:$\clSc$] (Sc) at (0,15) {};
   \node [label = above:$\clS$] (S) at (0,16.5) {};
   \node [label = below:$\clMc$] (Mc) at (0,23) {};
   \node [label = left:$\clMo$] (M0) at (2,24) {};
   \node [label = right:$\clMi$] (M1) at (-2,24) {};
   \node [label = above:$\clM$] (M) at (0,25) {};
   \node [label = below:$\clLambdac$] (Lamc) at (7.2,6.7) {};
   \node [label = left:$\clLambdai$] (Lam1) at (5,7.5) {};
   \node [label = right:$\clLambdao$] (Lam0) at (8.7,7.5) {};
   \node [label = below:$\clLambda$] (Lam) at (6.5,8.3) {};
   \node [label = left:$\clMcUk{\infty}$] (McUi) at (7.2,11.5) {};
   \node [label = left:$\clMUk{\infty}$] (MUi) at (8.7,13) {};
   \node [label = right:$\clTcUk{\infty}$] (TcUi) at (10.2,12) {};
   \node [label = right:$\clUk{\infty}$] (Ui) at (11.7,13.5) {};
   \node [label = left:$\clMcUk{3}$] (McU3) at (7.2,16) {};
   \node [label = left:$\clMUk{3}$] (MU3) at (8.7,17.5) {};
   \node [label = right:$\clTcUk{3}$] (TcU3) at (10.2,16.5) {};
   \node [label = right:$\clUk{3}$] (U3) at (11.7,18) {};
   \node [label = left:$\clMcUk{2}$] (McU2) at (7.2,19) {};
   \node [label = left:$\clMUk{2}$] (MU2) at (8.7,20.5) {};
   \node [label = right:$\clTcUk{2}$] (TcU2) at (10.2,19.5) {};
   \node [label = right:$\clUk{2}$] (U2) at (11.7,21) {};
   \node [label = below:$\clVc$] (Vc) at (-7.2,6.7) {};
   \node [label = right:$\clVo$] (V0) at (-5,7.5) {};
   \node [label = left:$\clVi$] (V1) at (-8.7,7.5) {};
   \node [label = below:$\clV$] (V) at (-6.5,8.3) {};
   \node [label = right:$\clMcWk{\infty}$] (McWi) at (-7.2,11.5) {};
   \node [label = right:$\clMWk{\infty}$] (MWi) at (-8.7,13) {};
   \node [label = left:$\clTcWk{\infty}$] (TcWi) at (-10.2,12) {};
   \node [label = left:$\clWk{\infty}$] (Wi) at (-11.7,13.5) {};
   \node [label = right:$\clMcWk{3}$] (McW3) at (-7.2,16) {};
   \node [label = right:$\clMWk{3}$] (MW3) at (-8.7,17.5) {};
   \node [label = left:$\clTcWk{3}$] (TcW3) at (-10.2,16.5) {};
   \node [label = left:$\clWk{3}$] (W3) at (-11.7,18) {};
   \node [label = right:$\clMcWk{2}$] (McW2) at (-7.2,19) {};
   \node [label = right:$\clMWk{2}$] (MW2) at (-8.7,20.5) {};
   \node [label = left:$\clTcWk{2}$] (TcW2) at (-10.2,19.5) {};
   \node [label = left:$\clWk{2}$] (W2) at (-11.7,21) {};
   \node [label = above:$\clOI$] (Tc) at (0,28) {};
   \node [label = right:$\clOX$] (T0) at (5,29.5) {};
   \node [label = left:$\clXI$] (T1) at (-5,29.5) {};
   \node [label = above:$\clAll$] (Omega) at (0,31) {};
   \draw [thick] (Ic) -- (Istar) to[out=135,in=-135] (Omega1);
   \draw [thick] (I) -- (Omega1);
   \draw [thick] (Omega1) to[out=135,in=-135] (L);
   \draw [thick] (Ic) -- (I0) -- (I);
   \draw [thick] (Ic) -- (I1) -- (I);
   \draw [thick] (Ic) to[out=128,in=-134] (Lc);
   \draw [thick] (Ic) to[out=58,in=-58] (SM);
   \draw [thick] (I0) -- (L0);
   \draw [thick] (I1) -- (L1);
   \draw [thick] (Istar) to[out=60,in=-60] (LS);
   \draw [thick] (Ic) -- (Lamc);
   \draw [thick] (I0) -- (Lam0);
   \draw [thick] (I1) -- (Lam1);
   \draw [thick] (I) -- (Lam);
   \draw [thick] (Ic) -- (Vc);
   \draw [thick] (I0) -- (V0);
   \draw [thick] (I1) -- (V1);
   \draw [thick] (I) -- (V);
   \draw [thick] (Lamc) -- (Lam0) -- (Lam);
   \draw [thick] (Lamc) -- (Lam1) -- (Lam);
   \draw [thick] (Lamc) -- (McUi);
   \draw [thick] (Lam0) -- (MUi);
   \draw [thick] (Lam1) -- (M1);
   \draw [thick] (Lam) -- (M);
   \draw [thick] (Vc) -- (V0) -- (V);
   \draw [thick] (Vc) -- (V1) -- (V);
   \draw [thick] (Vc) -- (McWi);
   \draw [thick] (V0) -- (M0);
   \draw [thick] (V1) -- (MWi);
   \draw [thick] (V) -- (M);
   \draw [thick] (McUi) -- (TcUi) -- (Ui);
   \draw [thick] (McUi) -- (MUi) -- (Ui);
   \draw [thick,loosely dashed] (McUi) -- (McU3);
   \draw [thick,loosely dashed] (MUi) -- (MU3);
   \draw [thick,loosely dashed] (TcUi) -- (TcU3);
   \draw [thick,loosely dashed] (Ui) -- (U3);
   \draw [thick] (McU3) -- (TcU3) -- (U3);
   \draw [thick] (McU3) -- (MU3) -- (U3);
   \draw [thick] (McU3) -- (McU2);
   \draw [thick] (MU3) -- (MU2);
   \draw [thick] (TcU3) -- (TcU2);
   \draw [thick] (U3) -- (U2);
   \draw [thick] (McU2) -- (TcU2) -- (U2);
   \draw [thick] (McU2) -- (MU2) -- (U2);
   \draw [thick] (McU2) -- (Mc);
   \draw [thick] (MU2) -- (M0);
   \draw [thick] (TcU2) to[out=120,in=-25] (Tc);
   \draw [thick] (U2) -- (T0);

   \draw [thick] (McWi) -- (TcWi) -- (Wi);
   \draw [thick] (McWi) -- (MWi) -- (Wi);
   \draw [thick,loosely dashed] (McWi) -- (McW3);
   \draw [thick,loosely dashed] (MWi) -- (MW3);
   \draw [thick,loosely dashed] (TcWi) -- (TcW3);
   \draw [thick,loosely dashed] (Wi) -- (W3);
   \draw [thick] (McW3) -- (TcW3) -- (W3);
   \draw [thick] (McW3) -- (MW3) -- (W3);
   \draw [thick] (McW3) -- (McW2);
   \draw [thick] (MW3) -- (MW2);
   \draw [thick] (TcW3) -- (TcW2);
   \draw [thick] (W3) -- (W2);
   \draw [thick] (McW2) -- (TcW2) -- (W2);
   \draw [thick] (McW2) -- (MW2) -- (W2);
   \draw [thick] (McW2) -- (Mc);
   \draw [thick] (MW2) -- (M1);
   \draw [thick] (TcW2) to[out=60,in=-155] (Tc);
   \draw [thick] (W2) -- (T1);

   \draw [thick] (SM) -- (McU2);
   \draw [thick] (SM) -- (McW2);

   \draw [thick] (Lc) -- (LS) -- (L);
   \draw [thick] (Lc) -- (L0) -- (L);
   \draw [thick] (Lc) -- (L1) -- (L);
   \draw [thick] (Lc) to[out=120,in=-120] (Sc);
   \draw [thick] (LS) to[out=60,in=-60] (S);
   \draw [thick] (L0) -- (T0);
   \draw [thick] (L1) -- (T1);
   \draw [thick] (L) to[out=125,in=-125] (Omega);
   \draw [thick] (SM) -- (Sc) -- (S);
   \draw [thick] (Sc) to[out=142,in=-134] (Tc);
   \draw [thick] (S) to[out=42,in=-42] (Omega);
   \draw [thick] (Mc) -- (M0) -- (M);
   \draw [thick] (Mc) -- (M1) -- (M);
   \draw [thick] (Mc) to[out=120,in=-120] (Tc);
   \draw [thick] (M0) -- (T0);
   \draw [thick] (M1) -- (T1);
   \draw [thick] (M) to[out=55,in=-55] (Omega);
   \draw [thick] (Tc) -- (T0) -- (Omega);
   \draw [thick] (Tc) -- (T1) -- (Omega);
\end{tikzpicture}
}
\end{center}
\caption{Post's lattice.}
\label{fig:Post}
\end{figure}
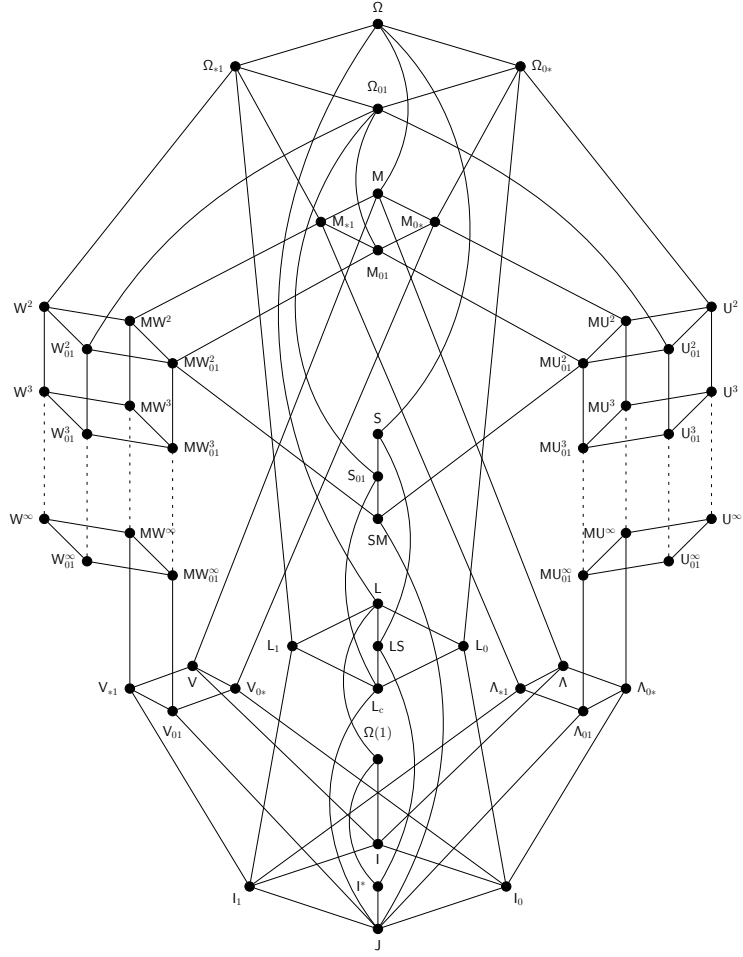

We conclude this section with some auxiliary results.

\begin{lemma}
\label{lem:AllSmaj}
For all $f \in \clAll$ and $g \in \clSmaj$ of the same arity, we have $f \vee g \in \clSmaj$.
\end{lemma}

\begin{proof}
Let $\vect{a} \in \{0,1\}^n$.
Because $g \in \clSmaj$, we have $g(\vect{a}) \vee g(\overline{\vect{a}}) = 1$.
Then
\begin{align*}
(f \vee g)(\vect{a}) \vee (f \vee g)(\overline{\vect{a}}) 
&= (f(\vect{a}) \vee g(\vect{a})) \vee (f(\overline{\vect{a}}) \vee g(\overline{\vect{a}}))
\\ &= (f(\vect{a}) \vee f(\overline{\vect{a}})) \vee (g(\vect{a}) \vee g(\overline{\vect{a}}))
= (f(\vect{a}) \vee f(\overline{\vect{a}})) \vee 1
= 1,
\end{align*}
which shows that $f \vee g \in \clSmaj$.
\end{proof}

\begin{lemma}
\label{lem:uncle}
Let $C$ be a clone on $\{0,1\}$, and let $K$ be a $(C,\clIc)$\hyp{}clonoid.
Then the following statements hold.
\begin{enumerate}[label=\upshape{(\roman*)}]
\item\label{lem:uncle:compl} $\overline{K}$ is a $(C,\clIc)$\hyp{}clonoid.
\item\label{lem:uncle:C0} $K \cup \clVako$ is a $(C,\clIo)$\hyp{}clonoid.
\item\label{lem:uncle:C1} $K \cup \clVaki$ is a $(C,\clIi)$\hyp{}clonoid.
\item\label{lem:uncle:C} $K \cup \clVak$ is a $(C,\clI)$\hyp{}clonoid.
\item\label{lem:uncle:Istar} $K \cup \overline{K}$ is a $(C,\clIstar)$\hyp{}clonoid.
\item\label{lem:uncle:O1} $K \cup \overline{K} \cup \clVak$ is a $(C,\clOmegaOne)$\hyp{}clonoid.
\end{enumerate}
\end{lemma}

\begin{proof}
Because $K$ is a $(C,\clIc)$\hyp{}clonoid, it holds that
\begin{equation}
K = \clIc ( K C ) = K C.
\label{eq:K}
\end{equation}

\ref{lem:uncle:compl}
By making use of the Unary Associativity Lemma (Lemma~\ref{lem:UAL}), we get
$\overline{K} = \{\neg\} K = \{\neg\} (K C) \stackrel{\text{\ref{lem:UAL}}}{=} (\{\neg\} K) C = \overline{K} C = \clIc (\overline{K} C)$.
Therefore, $\overline{K}$ is a $(C,\clIc)$\hyp{}clonoid.

\ref{lem:uncle:C0}
We have
\begin{align*}
\lhs
\clIo ( ( K \cup \clVako ) C )
= ( \clIc \cup \clVako) ( ( K \cup \clVako ) C )
\stackrel{\text{\ref{lem:CompUnion}}}{=}  \clIc ( ( K \cup \clVako ) C ) \cup \underbrace{\clVako ( ( K \cup \clVako ) C )}_{= \clVako}
\\ &
= ( K \cup \clVako ) C  \cup \clVako
\stackrel{\text{\ref{lem:CompUnion}}}{=}  (  K C \cup \clVako C ) \cup \clVako
\stackrel{\text{\eqref{eq:K}}}{=}  (  K  \cup \clVako ) \cup \clVako
= K \cup \clVako.
\end{align*}
Therefore $K \cup \clVako$ is a $(C,\clIo)$\hyp{}clonoid.

\ref{lem:uncle:C1}
The proof is analogous to part \ref{lem:uncle:C0}.

\ref{lem:uncle:C}
We have
\begin{align*}
\clI ( ( K \cup \clVak ) C )
&
= (\clIc \cup \clVak) ( ( K \cup \clVak ) C )
\stackrel{\text{\ref{lem:CompUnion}}}{=} \clIc ( ( K \cup \clVak ) C ) \cup \clVak ( ( K \cup \clVak ) C )
\\ &
= ( K \cup \clVak ) C  \cup \clVak 
\stackrel{\text{\ref{lem:CompUnion}}}{=} ( K C \cup \clVak C ) \cup \clVak
= ( K \cup \clVak ) \cup \clVak
= K \cup \clVak.
\end{align*}
Therefore $K \cup \clVak$ is a $(C,\clI)$\hyp{}clonoid.

\ref{lem:uncle:Istar}
We have
\begin{align*}
\lhs
\clIstar ( ( K \cup \overline{K} ) C )
= ( \clIc \cup \overline{\clIc} ) ( ( K \cup \overline{K} ) C )
= I ( ( K \cup \overline{K} ) C ) \cup ( \{\neg\} I ) ( ( K \cup \overline{K} ) C )
\\ &
\stackrel{\text{\ref{lem:UAL}}}{=} I ( ( K \cup \overline{K} ) C ) \cup \{\neg\} ( I ( ( K \cup \overline{K} ) C ) )
= ( K C \cup \overline{K} C) \cup \{\neg\} ( K C \cup \overline{K} C )
\\ &
\stackrel{\text{\ref{lem:uncle:compl}}}{=} ( K \cup \overline{K} ) \cup \{\neg\} ( K \cup \overline{K} )
= ( K \cup \overline{K} ) \cup ( \overline{K} \cup \overline{\overline{K}} )
= K \cup \overline{K}.
\end{align*}
Therefore $K \cup \overline{K}$ is a $(C,\clIstar)$\hyp{}clonoid.

\ref{lem:uncle:O1}
We have
\begin{align*}
\lhs
\clOmegaOne ( (K \cup \overline{K} \cup \clVak) C )
= (\clIstar \cup \clVak) ( ( K \cup \overline{K} \cup \clVak ) C )
\\ &
= \clIstar ( ( K \cup \overline{K} \cup \clVak ) C ) \cup \clVak ( ( K \cup \overline{K} \cup \clVak ) C )
= \clIstar ( ( K \cup \overline{K} \cup \clVak ) C ) \cup \clVak
\\ &
\stackrel{\text{\ref{lem:CompUnion}}}{=} \clIstar ( ( K \cup \overline{K} ) C \cup \clVak  C ) \cup \clVak
\stackrel{\text{\ref{lem:CompUnion}}}{=} \clIstar ( ( K \cup \overline{K} ) C ) \cup \clIstar ( \clVak  C ) \cup \clVak
\\ &
\stackrel{\text{\ref{lem:uncle:Istar}}}{=} ( K \cup \overline{K} ) \cup \clIstar ( \clVak  C ) \cup \clVak
= ( K \cup \overline{K} ) \cup \clIstar \clVak \cup \clVak
= ( K \cup \overline{K} ) \cup \clVak \cup \clVak
\\ &
= K \cup \overline{K} \cup \clVak.
\end{align*}
Therefore $K \cup \overline{K} \cup \clVak$ is a $(C,\clOmegaOne)$\hyp{}clonoid.
\end{proof}

\begin{lemma}
\label{lem:C1Istar}
Let $C$ be a clone on $\{0,1\}$.
The $(C,\clIstar)$\hyp{}clonoids are precisely the $(C,\clIc)$\hyp{}clonoids $K$ satisfying $K = \overline{K}$.
\end{lemma}

\begin{proof}
If $K$ is a $(C,\clIc)$\hyp{}clonoid satisfying $K = \overline{K}$, then, by Lemma~\ref{lem:uncle}\ref{lem:uncle:Istar}, $K \cup \overline{K} = K$ is a $(C,\clIstar)$\hyp{}clonoid.

Assume now that $K$ is a $(C,\clIstar)$\hyp{}clonoid and hence a $(C,\clIc)$\hyp{}clonoid.
Then $K \supseteq \clIstar K = K \cup \overline{K}$, from which it follows that $\overline{K} \subseteq K$.
Taking outer negations on both side of this inclusion, we get $K = \overline{\overline{K}} \subseteq \overline{K}$.
We conclude that $K = \overline{K}$.
\end{proof}


\section{Review of earlier results on $C$-minors}
\label{sec:review}

In a series of earlier papers of the author's
\cite{Lehtonen-threshold,Lehtonen-ULM,Lehtonen-semilattices,LehNes-clique,LehSze-discriminator,LehSze-submaximal,LehSze-finite,LehSze-quasilinear},
a notion called ``$C$\hyp{}minor'' was introduced and investigated.
In our current terminology, this concept is defined as follows.

\begin{definition}
Let $f, g \in \mathcal{F}_{AB}$, and let $C$ be a clone on $A$.
We say that $f$ is a \emph{$C$\hyp{}minor} of $g$, and we write $f \leq_C g$, if $f = g(h_1, \dots, h_m)$ for some $h_1, \dots, h_m \in C$, or, equivalently,  if $f \in \{g\} C$.
This condition is, in fact, equivalent to $f \in \gen[(C,\clProj{B})]{g}$, because $\{g\} C = \clProj{B} ( \{g\} C ) = \gen[(C,\clProj{B})]{g}$.
\end{definition}

For a fixed clone $C$, the $C$\hyp{}minor relation $\leq_C$ is a quasi\hyp{}order (a reflexive and transitive relation) on $\mathcal{F}_{AB}$.
As for all quasi\hyp{}orders, it induces an equivalence relation ($f$ and $g$ are \emph{$C$\hyp{}equivalent,} denoted $f \equiv_C g$, if $f \leq_C g$ and $g \leq_C f$) and a partial order on $\mathcal{F}_{AB} / \mathord{\equiv_C}$ ($f / \mathord{\equiv_C} \leq_C g / \mathord{\equiv_C}$ if and only if $f \leq_C g$).

The main focus of the earlier work was on describing  the $C$\hyp{}equivalence classes and the structure of the $C$\hyp{}minor poset $(\mathcal{F}_{AB} / \mathord{\equiv_C}, \mathord{\leq_C})$.
What is relevant to the purposes of our current paper is the fact that such results translate immediately into results about $(C,\clProj{B})$\hyp{}clonoids.

\begin{lemma}
\label{lem:clon-C-min}
Let $C$ be a clone on $A$, and let $F \subseteq \mathcal{F}_{AB}$.
\begin{enumerate}[label=\upshape{(\roman*)}]
\item\label{lem:clon-C-min:1}
$F$ is a downset of the $C$\hyp{}minor quasi\hyp{}order if and only if $F C = F$.
\item\label{lem:clon-C-min:2}
$F$ is a $(C,\clProj{B})$\hyp{}clonoid if and only if $F$ is a downset of the $C$\hyp{}minor quasi\hyp{}order $(\mathcal{F}_{AB}, \mathord{\leq_C})$.
\item\label{lem:clon-C-min:3}
$F$ is a $(C,\clProj{B})$\hyp{}clonoid if and only if $F = \bigcup D$ for some downset $D$ of the $C$\hyp{}minor partial order $(\mathcal{F}_{AB} / \mathord{\equiv_C}, \mathord{\leq_C})$.
\end{enumerate}
\end{lemma}

\begin{proof}
\ref{lem:clon-C-min:1}
Assume $F$ is a downset of $(\mathcal{F}_{AB}, \mathord{\leq_C})$.
The inclusion $F \subseteq F C$ is clear because $C$ contains all projections and therefore for every $f \in F$ we have $f = f(\pr_1^{(n)}, \dots, \pr_n^{(n)}) \in F C$.
In order to show $F C \subseteq F$, let $\varphi \in F C$.
Then $\varphi = f(g_1, \dots, g_n)$ for some $f \in F$ and $g_1, \dots, g_n \in C$, i.e., $\varphi \leq_C f$.
Because $f \in F$ and $F$ is a downset of $\leq_C$, it follows that $\varphi \in F$.

Assume now that $F C = F$.
Let $f \in F$, and assume $\varphi \leq_C f$.
Then $\varphi \in \{f\} C \subseteq F C = F$,
which shows that $F$ is a downset of $\leq_C$.

\ref{lem:clon-C-min:2}
The fact that
$F$ is a $(C,\clProj{B})$\hyp{}clonoid is equivalent to $F = \gen[(C,\clProj{B})]{F} = \clProj{B} (F C) = F C$.
By part \ref{lem:clon-C-min:1}, this is equivalent to the fact that $F$ is a downset of $(\mathcal{F}_{AB}, \mathord{\leq_C})$.

\ref{lem:clon-C-min:3}
This follows immediately from \ref{lem:clon-C-min:2} and from the way how the $C$\hyp{}minor quasi\hyp{}order $\leq_C$ induces a partial order on $\mathcal{F}_{AB} / \mathord{\equiv_C}$.
\end{proof}

In the subsequent sections, we are going to review the earlier results on $C$\hyp{}minors of Boolean functions, and we translate them to results about $(C,\clIc)$\hyp{}clonoids.
Furthermore, for those source clones $C$ for which the number of $(C,\clIc)$\hyp{}clonoids is finite or countably infinite, we are also going to explicitly describe the $(C,D)$\hyp{}clonoids, for each target clone $D$.


\section{Source clones giving rise to an uncountable infinitude of clonoids}
\label{sec:uncountable}

Let $\mathcal{Q}$ be a class of partial orders.
A poset $P$ is \emph{universal} for $\mathcal{Q}$ if $P \in \mathcal{Q}$ and for every $Q \in \mathcal{Q}$, $Q$ embeds into $P$.

\begin{theorem}[{\cite[Theorems~3.1, 6.1]{LehSze-quasilinear}}]
\label{thm:LambdaVL}
If $C$ is a subclone of $\clLambda$, $\clV$, or $\clL$, then the $C$\hyp{}minor partial order is universal for the class of countable posets with finite initial segments.
\end{theorem}

\begin{theorem}[{\cite[Theorems~12, 14]{LehNes-clique}}]
\label{thm:clique}
If $C$ is a clone belonging to the interval $[ \{ \clSM, \clMcUk{\infty}, \clMcWk{\infty} \}, \{ \clUk{2}, \clWk{2} \} ]$, then the $C$\hyp{}minor partial order is universal in the class of countable posets.
\end{theorem}

\begin{theorem}
If $C$ is a subclone of $\clLambda$, $\clV$, $\clL$, $\clUk{2}$, or $\clWk{2}$,
then there are an uncountable infinitude of $(C,\clIc)$\hyp{}clonoids.
\end{theorem}

\begin{proof}
If $C$ is a subclone of $\clLambda$, $\clV$, $\clL$, $\clUk{2}$, or $\clWk{2}$, then,
by Theorems~\ref{thm:LambdaVL} and \ref{thm:clique}, the $C$\hyp{}minor partial order is universal for the class of countable posets with finite initial segments or for the class of all countable posets.
Both classes of posets contain, in particular, a countably infinite antichain; therefore, the $C$\hyp{}minor poset contains a countably infinite antichain $A$.
Distinct subsets of $A$ generate distinct downsets of $(\clAll / \mathord{\equiv_C}, \mathord{\leq_C})$.
Consequently, $(\clAll / \mathord{\equiv_C}, \mathord{\leq_C})$ has uncountably many downsets, and hence, by Lemma~\ref{lem:clon-C-min}, there are an uncountable infinitude of $(C,\clIc)$\hyp{}clonoids.
\end{proof}

It is out of the scope of the current paper to investigate in greater detail the $(C_1,C_2)$\hyp{}clonoids in the cases where $C_1$ is a subclone of $\clLambda$, $\clV$, $\clL$, $\clUk{2}$, or $\clWk{2}$.


\section{Clonoids with a monotone source clone}
\label{sec:monotone}

For the four clones $C$ in the interval $[\clMc, \clM]$ (the clones of monotone functions), the $C$\hyp{}minor poset was described in earlier papers of the author's.
The proof was, however, scattered across several papers (\cite[Section~6]{Lehtonen-ULM}, \cite[Section~7]{LehNes-clique}, and Kosub and Wagner \cite[Proposition~38]{KosWag-IC}) and was, admittedly, somewhat implicit.
For this reason, we are going to present a complete proof here.
Our main tool is the homomorphism order of $k$\hyp{}posets.

For $k \in \IN$, a \emph{partially ordered set labeled with $k$ colours} (in short, a \emph{$k$\hyp{}poset}) is a structure $\mathbf{P} = (P, {\leq}, c)$, where $(P, {\leq})$ is a poset (the \emph{underlying poset}) and $c \colon P \to [0, k-1]$ is a \emph{labeling}.
If $(P, {\leq})$ is a lattice or a chain, then we may speak of a \emph{$k$\hyp{}lattice} or a \emph{$k$\hyp{}chain}, respectively.

Let $\mathbf{P} = (P, {\leq}, c)$ and $\mathbf{P}' = (P', {\leq}', c')$ be $k$\hyp{}posets.
A mapping $h \colon P \to P'$ is a \emph{homomorphism} of $\mathbf{P}$ to $\mathbf{P}'$, if $h$ preserves both the order and the labels, i.e.,
$h(x) \leq' h(y)$ whenever $x \leq y$, and $c = c' \circ h$.
We write $h \colon \mathbf{P} \to \mathbf{P}'$ to denote that $h$ is a homomorphism of $\mathbf{P}$ to $\mathbf{P}'$,
and we write $\mathbf{P} \to \mathbf{P}'$ to denote that there exists a homomorphism of $\mathbf{P}$ to $\mathbf{P}'$.
If $\mathbf{P} \to \mathbf{P}'$ and $\mathbf{P}' \to \mathbf{P}$, we say that $\mathbf{P}$ and $\mathbf{P}'$ are \emph{homomorphically equivalent.}
The existence of homomorphism relation $\to$ is a quasi\hyp{}order on the class of all $k$\hyp{}posets, and it induces a partial order on the homomorphical equivalence classes.
In general, the homomorphism order of $k$\hyp{}posets has a very rich structure; see \cite{Lehtonen-k-posets} for further details.
However, as we will see, the restriction of the homomorphism order to $2$\hyp{}lattices is very easy to describe.

If $\mathbf{P}$ and $\mathbf{P}'$ have least elements $\bot$ and $\bot'$, respectively, then a homomorphism $h \colon \mathbf{P} \to \mathbf{P}'$ is \emph{$\bot$\hyp{}preserving} if $h(\bot) = \bot'$.
We define \emph{$\top$\hyp{}preserving} homomorphisms analogously when $\mathbf{P}$ and $\mathbf{P}'$ have greatest elements $\top$ and $\top'$.
We call $\bot$\hyp{}preserving, $\top$\hyp{}preserving, and both $\bot$- and $\top$\hyp{}preserving homomorphisms also \emph{$\bot$\hyp{}homomorphisms}, \emph{$\top$\hyp{}homomorphisms}, and \emph{$\bot\top$\hyp{}homomorphisms}, respectively.
The composition of ($\bot$\hyp{}preserving, $\top$\hyp{}preserving) homomorphisms is again a homomorphism of the same kind.

Let $(A,{\leq})$ be a poset.
An operation $f \colon A^n \to A$ is \emph{monotone} (\emph{order\hyp{}preserving}) with respect to $\leq$ if $f(\vect{a}) \leq f(\vect{b})$ whenever $\vect{a} \leq' \vect{b}$ in $(A,{\leq})^n = (A^n,{\leq'})$,
where $\leq'$ is the component\hyp{}wise partial order induced by $\leq$, given by the rule $(a_1, \dots, a_n) \leq' (b_1, \dots, b_n)$ if and only if $a_i \leq b_i$ for all $i \in \nset{n}$.
Denote by $M$ the set of all monotone functions with respect to $\leq$ (the partial order is implicit in the notation and is understood from the context).
If $(A,{\leq})$ has a least element $\bot$, then denote by $M_\bot$ the set of all monotone functions $f$ that are $\bot$\hyp{}preserving, i.e., $f(\bot,\dots,\bot) = \bot$.
Similarly, if $(A,{\leq})$ has a greatest element $\top$, then denote by $M_\top$ the set of all $\top$\hyp{}preserving monotone functions.
If $(A,{\leq})$ has both a least element $\bot$ and a greatest element $\top$, let $M_{\bot\top} := M_\bot \cap M_\top$.
The sets $M$, $M_\bot$, $M_\top$, and $M_{\bot \top}$ are clones on $A$.

If $A = [0, k-1]$, we associate with any $f \colon A^n \to A$ the $k$\hyp{}poset $P(f,{\leq}) := ((A,{\leq})^n,f) = ((A^n,{\leq'}),f)$.

\begin{proposition}[{\cite[Proposition~6.1]{Lehtonen-ULM}}]
\label{prop:mon}
Let $(A,{\leq})$ be a poset.
\begin{enumerate}[label=\upshape{(\roman*)}]
\item\label{prop:mon:mon}
$f \leq_M g$ if and only if there is a homomorphism of $P(f,{\leq})$ to $P(g,{\leq})$.
\item\label{prop:mon:0}
If $(A,{\leq})$ has a least element $\bot$, then
$f \leq_{M_\bot} g$ if and only if there is a $\bot$\hyp{}homomorphism of $P(f,{\leq})$ to $P(g,{\leq})$.
\item\label{prop:mon:1}
If $(A,{\leq})$ has a greatest element $\top$, then
$f \leq_{M_\top} g$ if and only if there is a $\top$\hyp{}homomorphism of $P(f,{\leq})$ to $P(g,{\leq})$.
\item\label{prop:mon:01}
If $(A,{\leq})$ has a least element $\bot$ and a greatest element $\bot$, then
$f \leq_{M_{\bot\top}} g$ if and only if there is a $\bot\top$\hyp{}homomorphism of $P(f,{\leq})$ to $P(g,{\leq})$.
\end{enumerate}
\end{proposition}

\begin{proof}
Let $C$ be one of the clones $M$, $M_\bot$, $M_\top$, $M_{\bot\top}$.
Assume first that $f \leq_C g$.
Then $f = g(h_1, \dots, h_m)$ for some $h_1, \dots, h_m \in C$.
Clearly, $h = (h_1, \dots, h_m)$ is an order\hyp{}preserving map from $(A,{\leq})^n$ to $(A,{\leq})^m$ and $f = g \circ h$, so $h$ is a homomorphism of $P(f,{\leq})$ to $P(g,{\leq})$.
Moreover, if $C \subseteq M_\bot$, then $h(\bot,\dots,\bot) = (\bot,\dots,\bot)$, i.e., $h$ is $\bot$\hyp{}preserving; and if $C \subseteq M_\top$, then $h(\top,\dots,\top) = (\top,\dots,\top)$, i.e., $h$ is $\top$\hyp{}preserving.

Conversely, assume that there exists a homomorphism $h$ of $P(f,{\leq})$ to $P(g,{\leq})$.
Then $f = g \circ h$, and in $h = (h_1, \dots, h_m)$, each component function $h_i$ is in $M$.
Moreover, if $h$ is $\bot$\hyp{}preserving, then each $h_i$ is in $M_\bot$; and if $h$ is $\top$\hyp{}preserving, then each $h_i$ is in $M_\top$.
Therefore $f \leq_C g$.
\end{proof}

A $k$\hyp{}chain $(A, {\leq}, c)$ with $A = \{a_0, a_1, a_2, \dots, a_d\}$, $a_0 < a_1 < a_2 < \dots < a_d$, is \emph{alternating} if $c(a_i) \neq c(a_{i+1})$ for all $i \in [0,d-1]$.
The number $d \in \IN$ is the \emph{length} of this chain.

Note that an alternating $2$\hyp{}chain is uniquely determined, up to isomorphism, by its length and the label $c(a_0)$ of its least element; we have $c(a_i) = c(a_0)$ if and only if $i \equiv 0 \pmod{2}$.
Denote by $C^k_a$ ($k \in \IN$, $a \in \{0,1\}$) the alternating $2$\hyp{}chain of length $k$ in which the label of the least element is $a$.

The homomorphism order of alternating $2$\hyp{}chains is very easy to describe, and so are the
restrictions of the homomorphism order in which we confine ourselves to $\bot$-, $\top$-, or $\bot\top$\hyp{}homomorphisms; see also Figure~\ref{fig:hom2chain}.

\begin{proposition}
\label{prop:bottophom}
\leavevmode
\begin{enumerate}[label=\upshape{(\roman*)}]
\item
There exists a homomorphism of $C^k_a$ to $C^\ell_b$ if and only if $k < \ell$ or $(k,a) = (\ell,b)$.
\item
There exists a $\bot$\hyp{}homomorphism of $C^k_a$ to $C^\ell_b$ if and only if $k \leq \ell$ and $a = b$.
\item
There exists a $\top$\hyp{}homomorphism of $C^k_a$ to $C^\ell_b$ if and only if $k \leq \ell$ and $a + k \equiv b + \ell \pmod{2}$.
\item
There exists a $\bot\top$\hyp{}homomorphism of $C^k_a$ to $C^\ell_b$ if and only if $k \leq \ell$, $a = b$, and $k \equiv \ell \pmod{2}$.
\end{enumerate}
\end{proposition}

\begin{proof}
Straightforward verification.
\end{proof}

\begin{figure}
\begin{center}
\scalebox{0.9}{%
\begin{tikzpicture}[baseline, scale=0.6]
   \node [label = left:$C^0_0$] (00) at (0,0) {};
   \node [label = right:$C^0_1$] (10) at (2,0) {};
   \node [label = left:$C^1_0$] (01) at (0,2) {};
   \node [label = right:$C^1_1$] (11) at (2,2) {};
   \node [label = left:$C^2_0$] (02) at (0,4) {};
   \node [label = right:$C^2_1$] (12) at (2,4) {};
   \node [label = left:$C^3_0$] (03) at (0,6) {};
   \node [label = right:$C^3_1$] (13) at (2,6) {};
   \coordinate (04) at (0,8);
   \coordinate (14) at (2,8);
   \draw (00) -- (01) -- (02) -- (03);
   \draw (10) -- (11) -- (12) -- (13);
   \draw (00) -- (11) -- (02) -- (13);
   \draw (10) -- (01) -- (12) -- (03);
   \draw[dashed] (03) -- (04);
   \draw[dashed] (13) -- (14);
   \draw[dashed] (03) -- (14);
   \draw[dashed] (13) -- (04);
   \node[fill=none,draw=none] at (1,-1.5) {(a)};
\end{tikzpicture}
}
\qquad
\scalebox{0.9}{%
\begin{tikzpicture}[baseline, scale=0.6]
   \node [label = left:$C^0_0$] (00) at (0,0) {};
   \node [label = right:$C^0_1$] (10) at (2,0) {};
   \node [label = left:$C^1_0$] (01) at (0,2) {};
   \node [label = right:$C^1_1$] (11) at (2,2) {};
   \node [label = left:$C^2_0$] (02) at (0,4) {};
   \node [label = right:$C^2_1$] (12) at (2,4) {};
   \node [label = left:$C^3_0$] (03) at (0,6) {};
   \node [label = right:$C^3_1$] (13) at (2,6) {};
   \coordinate (04) at (0,8);
   \coordinate (14) at (2,8);
   \draw (00) -- (01) -- (02) -- (03);
   \draw (10) -- (11) -- (12) -- (13);
   \draw[dashed] (03) -- (04);
   \draw[dashed] (13) -- (14);
   \node[fill=none,draw=none] at (1,-1.5) {(b)};
\end{tikzpicture}
}
\qquad
\scalebox{0.9}{%
\begin{tikzpicture}[baseline, scale=0.6]
   \node [label = left:$C^0_0$] (00) at (0,0) {};
   \node [label = right:$C^0_1$] (10) at (2,0) {};
   \node [label = left:$C^1_0$] (01) at (0,2) {};
   \node [label = right:$C^1_1$] (11) at (2,2) {};
   \node [label = left:$C^2_0$] (02) at (0,4) {};
   \node [label = right:$C^2_1$] (12) at (2,4) {};
   \node [label = left:$C^3_0$] (03) at (0,6) {};
   \node [label = right:$C^3_1$] (13) at (2,6) {};
   \coordinate (04) at (0,8);
   \coordinate (14) at (2,8);
   \draw (00) -- (11) -- (02) -- (13);
   \draw (10) -- (01) -- (12) -- (03);
   \draw[dashed] (03) -- (14);
   \draw[dashed] (13) -- (04);
   \node[fill=none,draw=none] at (1,-1.5) {(c)};
\end{tikzpicture}
}
\qquad
\scalebox{0.9}{%
\begin{tikzpicture}[baseline, scale=0.6]
   \node [label = left:$C^0_0$] (00) at (0,0) {};
   \node [label = right:$C^0_1$] (10) at (2,0) {};
   \node [label = left:$C^1_0\,\,$] (01) at (0,2) {};
   \node [label = right:$C^1_1$] (11) at (2,2) {};
   \node [label = left:$C^2_0$] (02) at (0,4) {};
   \node [label = right:$\,\,C^2_1$] (12) at (2,4) {};
   \node [label = left:$C^3_0\,\,$] (03) at (0,6) {};
   \node [label = right:$C^3_1$] (13) at (2,6) {};
   \coordinate (04) at (0,8);
   \coordinate (14) at (2,8);
   \coordinate (05) at (0,10);
   \coordinate (15) at (2,10);
   \draw (00) to[out=110,in=250] (02);
   \draw (01) to[out=70,in=290] (03);
   \draw (10) to[out=110,in=250] (12);
   \draw (11) to[out=70,in=290] (13);
   \draw[dashed] (02) to[out=110,in=250] (04);
   \draw[dashed] (12) to[out=110,in=250] (14);
   \begin{scope}
   \clip ($(03)+(-1,0)$) rectangle ($(14)+(1,0)$);
   \draw[dashed] (03) to[out=70,in=290] (05);
   \draw[dashed] (13) to[out=70,in=290] (15);
   \end{scope}
   \node[fill=none,draw=none] at (1,-1.5) {(d)};
\end{tikzpicture}
}
\end{center}
\caption{Homomorphism orders of $2$\hyp{}chains:
(a) unrestricted,
(b) $\bot$\hyp{}preserving,
(c) $\top$\hyp{}preserving,
and
(d) $\bot$- and $\top$\hyp{}preserving homomorphisms.}
\label{fig:hom2chain}
\end{figure}
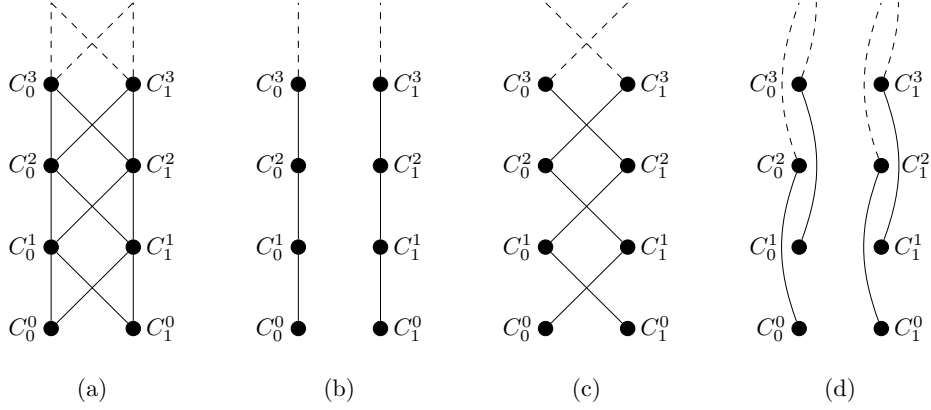

Let $P$ be a $k$\hyp{}poset with a least element $\bot$.
The \emph{alternation depth} of an element $x \in P$, denoted by $d_P(x)$, is the length of the longest alternating $k$\hyp{}chain included in the interval from $\bot$ to $x$.
Note that if $a_0 < \dots < a_{d(x)}$ is such a longest alternating $k$\hyp{}chain included in $[\bot,x]$, then $c(\bot) = c(a_0)$ and $c(x) = c(a_{d(x)})$; for, otherwise we could extend the chain into a longer alternating $k$\hyp{}chain by adjoining $\bot$ or $x$.
Therefore, whenever we consider a longest alternating $k$\hyp{}chain included in $[\bot,x]$, we may assume, if necessary, that it contains both $\bot$ and $x$.

The following slightly generalizes an observation about $2$\hyp{}lattices made by Kosub and Wagner \cite{KosWag-LNCS}, \cite[Proposition~38]{KosWag-IC}.

\begin{proposition}
\label{prop:2-poset}
Every $2$\hyp{}poset with a least element \textup{(}or with a greatest element\textup{)} is homomorphically equivalent to its longest alternating $2$\hyp{}chain.
\end{proposition}

\begin{proof}
Let $(P,{\leq},c)$ be a $2$\hyp{}poset with a least element $\bot$.
(The proof for a $2$\hyp{}poset with a greatest element is similar.)
It is clear that every $2$\hyp{}subposet of $P$ (in particular, any longest alternating $2$\hyp{}chain) maps homomorphically to $P$, so we only need to show that there exists a homomorphism from $P$ to a longest alternating chain in $P$.

Let $A = \{a_0, a_1, \dots, a_s\}$, $\bot = a_0 < a_1 < \dots < a_s$ be a longest alternating $2$\hyp{}chain in $P$.
Now, define the map $h \colon P \to A$ as $h(x) = a_{d(x)}$ for all $x \in P$.
We are going to show that $h$ is a homomorphism.
Firstly, if $x \leq y$ in $P$, then necessarily $d(x) \leq d(y)$.
(Every alternating $2$\hyp{}chain included in $[\bot,x]$ is also included in $[\bot,y]$.)
Therefore $h(x) = a_{d(x)} \leq a_{d(y)} = h(y)$.

Secondly,
observe that,
for all $x \in P$,
$c(x) = c(\bot)$ holds if and only if $d(x) \equiv 0 \pmod{2}$.
The latter in turn holds if and only if $c(a_{d(x)}) = c(\bot)$ (because $A$ is an alternating $2$\hyp{}chain).
Therefore $c(h(x)) = c(a_{d(x)}) = c(x)$, so $c = c \circ h$.
\end{proof}

We can now apply Propositions~\ref{prop:mon}, \ref{prop:bottophom}, and \ref{prop:2-poset} to describe the $C$\hyp{}minor posets of Boolean functions, for $\clMc \subseteq C \subseteq \clM$.
Their downsets give us the $(C, \clIc)$\hyp{}clonoids.

For a Boolean function $f \colon \{0,1\}^n \to \{0,1\}$, define the \emph{alternation depth} of $\vect{a} \in \{0,1\}^n$, and denote it by $d_f(\vect{a})$, as $d_{P(f,{\leq})}(\vect{a})$, i.e., the length of the longest alternating chain in the interval $[\vect{0}, \vect{a}]$ in $P(f,{\leq})$.
Furthermore, define the \emph{alternation number} of $f \colon \{0,1\}^n \to \{0,1\}$, denoted $\Alt(f)$, as the length of the longest alternating chain in $P(f,{\leq})$, i.e., $\Alt(f) = d_f(\vect{1})$.
Denote by $\clAlt{k}$ the set of all Boolean functions with alternation number $k$ and by $\clAlt{\leq k}$ the set of all Boolean functions with alternation number at most $k$.
We now obtain a description of the $C$\hyp{}minor quasi\hyp{}order for each clone $C$ in the interval $[\clMc, \clM]$.

\begin{proposition}
\label{prop:M-minors}
Let $f, g \in \clAll$, and let $k := \Alt(f)$, $\ell := \Alt(g)$, $a := f(\vect{0})$, $b := g(\vect{0})$.
\begin{enumerate}[label=\upshape{(\roman*)}]
\item
$f \leq_{\clM} g$ if and only if $k < \ell$ or $(k,a) = (\ell,b)$.
\item
$f \leq_{\clMo} g$ if and only if $k \leq \ell$ and $a = b$.
\item
$f \leq_{\clMi} g$ if and only if $k \leq \ell$ and $a + k \equiv b + \ell \pmod{2}$.
\item
$f \leq_{\clMc} g$ if and only if $k \leq \ell$, $a = b$, and $k \equiv \ell \pmod{2}$.
\end{enumerate}
\end{proposition}

\begin{proof}
By Proposition~\ref{prop:2-poset}, for every $f \in \clIntVal{\clAlt{k}}{a}{}$ ($a \in \{0,1\}$), the $2$\hyp{}poset $P(f,{\leq})$ is homomorphically equivalent to $C^k_a$, and it is clear that they are even $\bot\top$\hyp{}homomorphically equivalent.
The statements then follow immediately from Propositions~\ref{prop:mon} and \ref{prop:bottophom}.
\end{proof}

As an immediate consequence of Proposition~\ref{prop:M-minors}, we get that for each clone $C$ in the interval $[\clMc, \clM]$, the $C$\hyp{}equivalence classes are precisely the sets of the form $\clIntVal{\clAlt{k}}{a}{}$, for $k \in \IN_{+}$ and $a \in \{0,1\}$.
Note that $\clIntVal{\clAlt{k}}{a}{} = \clIntVal{\clAlt{k}}{a}{b}$ for the unique $b \in \{0,1\}$ satisfying $a + k \equiv b \pmod{2}$.
Moreover, the mapping $\clIntVal{\clAlt{k}}{a}{} \mapsto C^k_a$ is an isomorphism between the $C$\hyp{}minor poset $(\clAll / \mathord{\equiv_C}, \mathord{\leq_C})$ and the corresponding homomorphism order (unrestricted, $\bot$\hyp{}preserving, $\top$\hyp{}preserving, $\bot\top$\hyp{}preserving) of alternating $2$\hyp{}chains (see Figure~\ref{fig:hom2chain}).

\begin{lemma}
\label{lem:AkM}
Let $k \in \IN$ and $a, b \in \{0,1\}$ such that $a + k \equiv b \pmod{2}$.
Then the following statements hold.
\begin{enumerate}[label=\upshape{(\roman*)}]
\item $\clIntVal{\clAlt{k}}{a}{b} \, \clM = \clIntVal{\clAlt{k}}{a}{b} \cup \clAlt{\leq k-1}$ for $k \geq 1$; $\clIntVal{\clAlt{0}}{a}{b} \, \clM = \clIntVal{\clAlt{0}}{a}{b}$.
\item $\clIntVal{\clAlt{k}}{a}{b} \, \clMo = \clIntVal{\clAlt{\leq k}}{a}{}$.
\item $\clIntVal{\clAlt{k}}{a}{b} \, \clMi = \clIntVal{\clAlt{\leq k}}{}{b}$.
\item $\clIntVal{\clAlt{k}}{a}{b} \, \clMc = \clIntVal{\clAlt{\leq k}}{a}{b}$.
\end{enumerate}
\end{lemma}

\begin{proof}
Follows immediately from Proposition~\ref{prop:M-minors}.
\end{proof}

We are now ready to describe the $(C,\clIc)$\hyp{}clonoids for $C \in [\clMc,\clM]$.

\begin{theorem}[See Figure~\ref{fig:mon-clonoids}.]
\label{thm:mon-clonoids}
\leavevmode
\begin{enumerate}[label=\upshape{(\roman*)}]
\item\label{thm:mon-clonoids:M}
The $(\clM,\clIc)$\hyp{}clonoids are the sets $\clEmpty$, $\clAll$, $\clIntVal{\clAlt{0}}{0}{0} = \clVako$, $\clIntVal{\clAlt{0}}{1}{1} = \clVaki$, $\clAlt{\leq k}$, $\clIntVal{\clAlt{k+1}}{0}{} \cup \clAlt{\leq k}$, and $\clIntVal{\clAlt{k+1}}{1}{} \cup \clAlt{\leq k}$ for $k \in \IN$.

\item\label{thm:mon-clonoids:Mo}
The $(\clMo,\clIc)$\hyp{}clonoids are the sets of the form $A \cup B$, where
\begin{align*}
\displaybump
&
A \in \{ \clEmpty, \clOX \} \cup \{ \, \clIntVal{\clAlt{\leq k}}{0}{} \mid k \in \IN \, \},
&&
B \in \{ \clEmpty, \clIX \} \cup \{ \, \clIntVal{\clAlt{\leq k}}{1}{} \mid k \in \IN \, \}.
\end{align*}

\item\label{thm:mon-clonoids:Mi}
The $(\clMi,\clIc)$\hyp{}clonoids are the sets of the form $A \cup B$, where
\begin{align*}
\displaybump
&
A \in \{ \clEmpty, \clXO \} \cup \{ \, \clIntVal{\clAlt{\leq k}}{}{0} \mid k \in \IN \, \},
&&
B \in \{ \clEmpty, \clXI \} \cup \{ \, \clIntVal{\clAlt{\leq k}}{}{1} \mid k \in \IN \, \}.
\end{align*}

\item\label{thm:mon-clonoids:Mc}
The $(\clMc,\clIc)$\hyp{}clonoids are the sets of the form $A \cup B \cup C \cup D$, where
\begin{align*}
\displaybump
&
A \in \{ \clEmpty, \clOO \} \cup \{ \, \clIntVal{\clAlt{\leq k}}{0}{0} \mid k \in \IN \, \},
&&
B \in \{ \clEmpty, \clOI \} \cup \{ \, \clIntVal{\clAlt{\leq k}}{0}{1} \mid k \in \IN \, \},
\\ &
C \in \{ \clEmpty, \clII \} \cup \{ \, \clIntVal{\clAlt{\leq k}}{1}{1} \mid k \in \IN \, \},
&&
D \in \{ \clEmpty, \clIO \} \cup \{ \, \clIntVal{\clAlt{\leq k}}{1}{0} \mid k \in \IN \, \}.
\end{align*}
\end{enumerate}
\end{theorem}

\begin{proof}
For each clone $C$ in the interval $[\clMc, \clM]$, the $(C,\clIc)$\hyp{}clonoids are the downsets of the $C$\hyp{}minor quasi\hyp{}order.
The latter can be easily determined from the Hasse diagrams of Figure~\ref{fig:hom2chain}, using the isomorphism $\clIntVal{\clAlt{k}}{a}{} \mapsto C^k_a$ between the $C$\hyp{}minor poset and the appropriate homomorphism order of alternating $2$\hyp{}chains.
\end{proof}

\begin{figure}
\subfloat[$(\protect\clM,\protect\clIc)$-clonoids]{%
\begin{minipage}[b]{0.46\textwidth}
\begin{center}
\scalebox{0.7}{%
\begin{tikzpicture}[baseline, scale=0.6]
   \node [label = {[label distance=1.4]270:$\clEmpty$}] (E) at (0,0) {};
   \node [label = left:{$\clIntVal{\clAlt{0}}{0}{0} = \clVako$}] (00) at (-1,1) {};
   \node [label = right:{$\clIntVal{\clAlt{0}}{1}{1} = \clVaki$}] (01) at (1,1) {};
   \node [label = {[label distance=1.6]0:$\clAlt{0} = \clVak$}] (0) at (0,2) {};
   \node [label = left:{$\clIntVal{\clAlt{1}}{0}{1} \cup \clAlt{0}$}] (10) at (-1,3) {};
   \node [label = right:{$\clIntVal{\clAlt{1}}{1}{0} \cup \clAlt{0}$}] (11) at (1,3) {};
   \node [label = {[label distance=1.6]0:$\clAlt{\leq 1}$}] (1) at (0,4) {};
   \node [label = left:{$\clIntVal{\clAlt{2}}{0}{0} \cup \clAlt{\leq 1}$}] (20) at (-1,5) {};
   \node [label = right:{$\clIntVal{\clAlt{2}}{1}{1} \cup \clAlt{\leq 1}$}] (21) at (1,5) {};
   \node [label = {[label distance=1.6]0:$\clAlt{\leq 2}$}] (2) at (0,6) {};
   \node [label = {[label distance=1.4]90:$\clAll$}] (All) at (0,9) {};
   \draw (E) -- (00) -- (0) -- (10) -- (1) -- (20) -- (2);
   \draw (E) -- (01) -- (0) -- (11) -- (1) -- (21) -- (2);
   \draw[dashed] (2) to[out=135,in=225] (All);
   \draw[dashed] (2) to[out=45,in=315] (All);
\end{tikzpicture}
}
\end{center}
\end{minipage}
}
\hfill
\subfloat[$(\protect\clMc,\protect\clIc)$-clonoids]{%
\begin{minipage}[b]{0.46\textwidth}
\begin{center}
\scalebox{0.7}{%
\begin{tikzpicture}[baseline, scale=0.6]
   \node [label = {[label distance=1.4]270:$\clEmpty$}] (00) at (0,0) {};
   \node [label = left:{$\clIntVal{\clAlt{\leq 0}}{0}{0}$}] (01) at (0,1) {};
   \node [label = left:{$\clIntVal{\clAlt{\leq 2}}{0}{0}$}] (02) at (0,2) {};
   \node [label = left:{$\clIntVal{\clAlt{\leq 4}}{0}{0}$}] (03) at (0,3) {};
   \node [label = above:{$\clOO$}] (04) at (0,5) {};
   \draw (00) -- (01) -- (02) -- (03);
   \draw[dashed] (03) -- (04);
   \node [label = {[label distance=1.4]270:$\clEmpty$}] (10) at (3,0) {};
   \node [label = left:{$\clIntVal{\clAlt{\leq 1}}{0}{1}$}] (11) at (3,1) {};
   \node [label = left:{$\clIntVal{\clAlt{\leq 3}}{0}{1}$}] (12) at (3,2) {};
   \node [label = left:{$\clIntVal{\clAlt{\leq 5}}{0}{1}$}] (13) at (3,3) {};
   \node [label = above:{$\clOI$}] (14) at (3,5) {};
   \draw (10) -- (11) -- (12) -- (13);
   \draw[dashed] (13) -- (14);
   \node [label = {[label distance=1.4]270:$\clEmpty$}] (20) at (6,0) {};
   \node [label = left:{$\clIntVal{\clAlt{\leq 0}}{1}{1}$}] (21) at (6,1) {};
   \node [label = left:{$\clIntVal{\clAlt{\leq 2}}{1}{1}$}] (22) at (6,2) {};
   \node [label = left:{$\clIntVal{\clAlt{\leq 4}}{1}{1}$}] (23) at (6,3) {};
   \node [label = above:{$\clII$}] (24) at (6,5) {};
   \draw (20) -- (21) -- (22) -- (23);
   \draw[dashed] (23) -- (24);
   \node [label = {[label distance=1.4]270:$\clEmpty$}] (30) at (9,0) {};
   \node [label = left:{$\clIntVal{\clAlt{\leq 1}}{1}{0}$}] (31) at (9,1) {};
   \node [label = left:{$\clIntVal{\clAlt{\leq 3}}{1}{0}$}] (32) at (9,2) {};
   \node [label = left:{$\clIntVal{\clAlt{\leq 5}}{1}{0}$}] (33) at (9,3) {};
   \node [label = above:{$\clIO$}] (34) at (9,5) {};
   \draw (30) -- (31) -- (32) -- (33);
   \draw[dashed] (33) -- (34);
   \node [fill = none, draw = none] at (1,2) {$\times$};
   \node [fill = none, draw = none] at (4,2) {$\times$};
   \node [fill = none, draw = none] at (7,2) {$\times$};
\end{tikzpicture}
}
\end{center}
\end{minipage}
}

\subfloat[$(\protect\clMo,\protect\clIc)$-clonoids]{%
\begin{minipage}[b]{0.46\textwidth}
\begin{center}
\scalebox{0.7}{%
\begin{tikzpicture}[baseline, scale=0.6]
   \node [label = {[label distance=1.4]270:$\clEmpty$}] (00) at (0,0) {};
   \node [label = left:{$\clIntVal{\clAlt{0}}{0}{} = \clVako$}] (10) at (-1,1) {};
   \node [label = {[label distance=1.2]0:$\clIntVal{\clAlt{0}}{1}{} = \clVaki$}] (01) at (1,1) {};
   \node [label = {[label distance=1.2]180:$\clIntVal{\clAlt{\leq 1}}{0}{}$}] (20) at (-2,2) {};
   \node (11) at (0,2) {};
   \node [label = {[label distance=1.2]0:$\clIntVal{\clAlt{\leq 1}}{1}{}$}] (02) at (2,2) {};
   \node [label = {[label distance=1.2]180:$\clIntVal{\clAlt{\leq 2}}{0}{}$}] (30) at (-3,3) {};
   \node (21) at (-1,3) {};
   \node (12) at (1,3) {};
   \node [label = {[label distance=1.2]0:$\clIntVal{\clAlt{\leq 2}}{1}{}$}] (03) at (3,3) {};
   \node [label = {[label distance=1.2]180:$\clIntVal{\clAlt{\leq 3}}{0}{}$}] (40) at (-4,4) {};
   \node (31) at (-2,4) {};
   \node (22) at (0,4) {};
   \node (13) at (2,4) {};
   \node [label = {[label distance=1.2]0:$\clIntVal{\clAlt{\leq 3}}{1}{}$}] (04) at (4,4) {};
   \node [label = left:{$\clOX$}] (50) at (-5.5,5.5) {};
   \node (41) at (-3,5) {};
   \node (32) at (-1,5) {};
   \node (23) at (1,5) {};
   \node (14) at (3,5) {};
   \node [label = {[label distance=1.2]0:$\clIX$}] (05) at (5.5,5.5) {};
   \node (51) at (-4.5,6.5) {};
   \node (42) at (-2,6) {};
   \node (33) at (-0,6) {};
   \node (24) at (2,6) {};
   \node (15) at (4.5,6.5) {};
   \node (52) at (-3.5,7.5) {};
   \node (25) at (3.5,7.5) {};
   \coordinate (43) at (-1,7);
   \coordinate (34) at (1,7);
   \node [label = {[label distance=1.4]90:$\clAll$}] (All) at (0,11) {};
   \draw (00) -- (10) -- (20) -- (30) -- (40);
   \draw[dashed] (40) -- (50);
   \draw (01) -- (11) -- (21) -- (31) -- (41);
   \draw[dashed] (41) -- (51);
   \draw (02) -- (12) -- (22) -- (32) -- (42);
   \draw[dashed] (42) -- (52);
   \draw (03) -- (13) -- (23) -- (33);
   \draw[dashed] (33) -- (43);
   \draw (04) -- (14) -- (24);
   \draw[dashed] (24) -- (34);
   \draw (00) -- (01) -- (02) -- (03) -- (04);
   \draw[dashed] (04) -- (05);
   \draw (10) -- (11) -- (12) -- (13) -- (14);
   \draw[dashed] (14) -- (15);
   \draw (20) -- (21) -- (22) -- (23) -- (24);
   \draw[dashed] (24) -- (25);
   \draw (30) -- (31) -- (32) -- (33);
   \draw[dashed] (33) -- (34);
   \draw (40) -- (41) -- (42);
   \draw[dashed] (42) -- (43);
   \draw (50) -- (51) -- (52);
   \draw[dashed] (52) -- (All);
   \draw (05) -- (15) -- (25);
   \draw[dashed] (25) -- (All);
\end{tikzpicture}
}
\end{center}
\end{minipage}
}
\hfill
\subfloat[$(\protect\clMi,\protect\clIc)$-clonoids]{%
\begin{minipage}[b]{0.46\textwidth}
\begin{center}
\scalebox{0.7}{%
\begin{tikzpicture}[baseline, scale=0.6]
   \node [label = {[label distance=1.4]270:$\clEmpty$}] (00) at (0,0) {};
   \node [label = left:{$\clIntVal{\clAlt{0}}{}{0} = \clVako$}] (10) at (-1,1) {};
   \node [label = {[label distance=1.2]0:$\clIntVal{\clAlt{0}}{}{1} = \clVaki$}] (01) at (1,1) {};
   \node [label = {[label distance=1.2]180:$\clIntVal{\clAlt{\leq 1}}{}{0}$}] (20) at (-2,2) {};
   \node (11) at (0,2) {};
   \node [label = {[label distance=1.2]0:$\clIntVal{\clAlt{\leq 1}}{}{1}$}] (02) at (2,2) {};
   \node [label = {[label distance=1.2]180:$\clIntVal{\clAlt{\leq 2}}{}{0}$}] (30) at (-3,3) {};
   \node (21) at (-1,3) {};
   \node (12) at (1,3) {};
   \node [label = {[label distance=1.2]0:$\clIntVal{\clAlt{\leq 2}}{}{1}$}] (03) at (3,3) {};
   \node [label = {[label distance=1.2]180:$\clIntVal{\clAlt{\leq 3}}{}{0}$}] (40) at (-4,4) {};
   \node (31) at (-2,4) {};
   \node (22) at (0,4) {};
   \node (13) at (2,4) {};
   \node [label = {[label distance=1.2]0:$\clIntVal{\clAlt{\leq 3}}{}{1}$}] (04) at (4,4) {};
   \node [label = left:{$\clXO$}] (50) at (-5.5,5.5) {};
   \node (41) at (-3,5) {};
   \node (32) at (-1,5) {};
   \node (23) at (1,5) {};
   \node (14) at (3,5) {};
   \node [label = {[label distance=1.2]0:$\clXI$}] (05) at (5.5,5.5) {};
   \node (51) at (-4.5,6.5) {};
   \node (42) at (-2,6) {};
   \node (33) at (-0,6) {};
   \node (24) at (2,6) {};
   \node (15) at (4.5,6.5) {};
   \node (52) at (-3.5,7.5) {};
   \node (25) at (3.5,7.5) {};
   \coordinate (43) at (-1,7);
   \coordinate (34) at (1,7);
   \node [label = {[label distance=1.4]90:$\clAll$}] (All) at (0,11) {};
   \draw (00) -- (10) -- (20) -- (30) -- (40);
   \draw[dashed] (40) -- (50);
   \draw (01) -- (11) -- (21) -- (31) -- (41);
   \draw[dashed] (41) -- (51);
   \draw (02) -- (12) -- (22) -- (32) -- (42);
   \draw[dashed] (42) -- (52);
   \draw (03) -- (13) -- (23) -- (33);
   \draw[dashed] (33) -- (43);
   \draw (04) -- (14) -- (24);
   \draw[dashed] (24) -- (34);
   \draw (00) -- (01) -- (02) -- (03) -- (04);
   \draw[dashed] (04) -- (05);
   \draw (10) -- (11) -- (12) -- (13) -- (14);
   \draw[dashed] (14) -- (15);
   \draw (20) -- (21) -- (22) -- (23) -- (24);
   \draw[dashed] (24) -- (25);
   \draw (30) -- (31) -- (32) -- (33);
   \draw[dashed] (33) -- (34);
   \draw (40) -- (41) -- (42);
   \draw[dashed] (42) -- (43);
   \draw (50) -- (51) -- (52);
   \draw[dashed] (52) -- (All);
   \draw (05) -- (15) -- (25);
   \draw[dashed] (25) -- (All);
\end{tikzpicture}
}
\end{center}
\end{minipage}
}
\caption{$(C,\protect\clIc)$-clonoids for $\protect\clMc \subseteq C \subseteq \protect\clM$.}
\label{fig:mon-clonoids}
\end{figure}

With the help of Theorem~\ref{thm:mon-clonoids}, it is now possible to determine the $(C_1,C_2)$\hyp{}clonoids for clones $C_1$ and $C_2$ such that $\clMc \subseteq C_1$ and $C_2$ is arbitrary.
By Lemma~\ref{lem:clonoid-inclusion}, such $(C_1,C_2)$\hyp{}clonoids are $(\clMc,\clIc)$\hyp{}clonoids; it is just a matter of identifying them among the $(\clMc,\clIc)$\hyp{}clonoids that were determined above.
We start with the cases where the target clone is essentially at most unary.

\begin{proposition}
\label{prop:mon-Istar}
\leavevmode
\begin{enumerate}[label=\upshape{(\roman*)}]
\item
For clones $C_1 \in \{\clM, \clMo, \clMi, \clMc\}$,
\begin{itemize}
\item the $(C_1, \clIo)$\hyp{}clonoids are $\clEmpty$ and those $(C_1, \clIc)$\hyp{}clonoids $K$ with $\clVako = \clIntVal{\clAlt{0}}{0}{0} \subseteq K$,
\item the $(C_1, \clIi)$\hyp{}clonoids are $\clEmpty$ and those $(C_1, \clIc)$\hyp{}clonoids $K$ with $\clVaki = \clIntVal{\clAlt{0}}{1}{1} \subseteq K$,
\item the $(C_1, \clI)$\hyp{}clonoids are $\clEmpty$ and those $(C_1, \clIc)$\hyp{}clonoids $K$ with $\clVak = \clAlt{0} \subseteq K$.
\end{itemize}

\item
For clones $C_1 \in \{\clM, \clMo, \clMi\}$ and $C_2 \in \{\clIstar, \clOmegaOne\}$,
the $(C_1,C_2)$\hyp{}clonoids are the sets $\clEmpty$, $\clAll$, $\clAlt{\leq k}$, for $k \in \IN$.

\item
The $(\clMc,\clIstar)$\hyp{}clonoids are the sets of the form $A \cup B$, where
\begin{align*}
\displaybump
&
A \in \{ \clEmpty, \clEq \} \cup \{ \, \clIntVal{\clAlt{\leq k}}{0}{0} \cup \clIntVal{\clAlt{\leq k}}{1}{1} \mid k \in \IN \, \},
\\ &
B \in \{ \clEmpty, \clNeq \} \cup \{ \, \clIntVal{\clAlt{\leq k}}{0}{1} \cup \clIntVal{\clAlt{\leq k}}{1}{0} \mid k \in \IN \, \}.
\end{align*}

\item
The $(\clMc,\clOmegaOne)$\hyp{}clonoids are $\clEmpty$ and the sets of the form $A \cup B$, where
\begin{align*}
\displaybump
&
A \in \{ \clEq \} \cup \{ \, \clIntVal{\clAlt{\leq k}}{0}{0} \cup \clIntVal{\clAlt{\leq k}}{1}{1} \mid k \in \IN \, \},
\\ &
B \in \{ \clEmpty, \clNeq \} \cup \{ \, \clIntVal{\clAlt{\leq k}}{0}{1} \cup \clIntVal{\clAlt{\leq k}}{1}{0} \mid k \in \IN \, \}.
\end{align*}

\item
For clones $C_1 \in \{\clM, \clMo, \clMi, \clMc\}$ and $C_2 \in \{\clIc, \clIo, \clIi, \clI, \clIstar, \clOmegaOne\}$, the
clonoid lattice $\closys{(C_1,C_2)}$ is countably infinite.
\end{enumerate}
\end{proposition}

\begin{proof}
This follows immediately from Theorem~\ref{thm:mon-clonoids} and Lemmata~\ref{lem:C1C2Vak} and \ref{lem:C1Istar}.
Clearly, the clonoid lattices $\closys{(C_1,C_2)}$ considered here are countably infinite.
\end{proof}

The cases where the target clone includes $\clLc$ or $\clSM$ have been described in earlier papers of the author's.
We quote here only the results for the cases where the target clone is $\clLc$ or $\clSM$.
The results for their superclones can be found in the earlier papers, and we simply refer the reader to those papers.
The Hasse diagrams of $\closys{(\clMc,\clLc)}$ and $\closys{(\clMc,\clSM)}$ are shown in Figures~\ref{fig:McLc} and \ref{fig:McSM}.

\begin{proposition}[{\cite[Theorem~7.1, Table~3]{CouLeh-Lcstability}}]
\label{prop:McLc}
There are precisely 15 $(\clMc,\clLc)$\hyp{}clonoids, and they are the following:
$\clAll$, $\clOX$, $\clIX$, $\clXO$, $\clXI$, $\clEq$, $\clNeq$, $\clOO$, $\clOI$, $\clIO$, $\clII$, $\clVak$, $\clVako$, $\clVaki$, $\clEmpty$.
\end{proposition}

\begin{figure}
\begin{center}
\scalebox{0.7}{%
\begin{tikzpicture}[baseline, scale=0.6]
   \node [label = {[label distance=1.4]270:$\clEmpty$}] (E) at (0,0) {};
   \node [label = left:{$\clVako$}] (C0) at (-1,2) {};
   \node [label = right:{$\clVaki$}] (C1) at (1,2) {};
   \node [label = {[label distance=2.2]270:$\clVak$}] (C) at (0,4) {};
   \node [label = left:{$\clOI$}] (01) at (-5,4) {};
   \node [label = {[shift={(-0.45,-0.55)}]$\clOO$}] (00) at (-2,4) {};
   \node [label = {[shift={(0.45,-0.55)}]$\clII$}] (11) at (2,4) {};
   \node [label = right:{$\clIO$}] (10) at (5,4) {};
   \node [label = left:{$\clOX$}] (0X) at (-5,6) {};
   \node [label = left:{$\clXO$}] (X0) at (-3,6) {};
   \node [label = {[shift={(-0.4,-0.3)}]$\clEq$}] (eq) at (-1,6) {};
   \node [label = {[shift={(0.4,-0.3)}]$\clNeq$}] (neq) at (1,6) {};
   \node [label = right:{$\clXI$}] (X1) at (3,6) {};
   \node [label = right:{$\clIX$}] (1X) at (5,6) {};
   \node [label = {[label distance=1.4]90:$\clAll$}] (all) at (0,9) {};
   \draw (E) -- (C0) -- (C);
   \draw (E) -- (C1) -- (C);
   \draw (E) -- (01);
   \draw (E) -- (10);
   \draw (C0) -- (00);
   \draw (C1) -- (11);
   \draw (01) -- (0X);
   \draw (01) -- (X1);
   \draw (01) -- (neq);
   \draw (00) -- (0X);
   \draw (00) -- (X0);
   \draw (00) -- (eq);
   \draw (11) -- (1X);
   \draw (11) -- (X1);
   \draw (11) -- (eq);
   \draw (10) -- (1X);
   \draw (10) -- (X0);
   \draw (10) -- (neq);
   \draw (C) -- (eq);
   \draw (0X) -- (all); \draw (X0) -- (all); \draw (eq) -- (all); \draw (neq) -- (all); \draw (1X) -- (all); \draw (X1) -- (all);
\end{tikzpicture}
}
\end{center}
\caption{$(\protect\clMc,\protect\clLc)$\hyp{}clonoids.}
\label{fig:McLc}
\end{figure}

\begin{proposition}[{\cite[Theorem~5.1, Table~1]{Lehtonen-SM}}]
\label{prop:McSM}
There are precisely 39 $(\clMc,\clSM)$\hyp{}clonoids, and they are the following:
$\clAll$,
$\clEiio$, $\clEioi$, $\clEiii$, $\clEioo$,
$\clOXCI$, $\clXOCI$, $\clIXCO$, $\clXICO$,
$\clOX$, $\clXO$, $\clIX$, $\clXI$, $\clEq$, $\clNeq$,
$\clOIC$, $\clIOC$,
$\clOICO$, $\clIOCO$,
$\clOICI$, $\clIOCI$,
$\clOOCI$, $\clIICO$,
$\clOO$, $\clII$, $\clOI$, $\clIO$,
$\clM$, $\clMo$, $\clMi$, $\clMc$,
$\clMneg$, $\clMineg$, $\clMoneg$, $\clMcneg$,
$\clVak$, $\clVako$, $\clVaki$,
$\clEmpty$.
\end{proposition}

\begin{figure}
\begin{center}
\scalebox{0.7}{%
\begin{tikzpicture}[baseline, scale=0.6]
   \node [label = {[shift={(0,-0.6)}]$\clEmpty$}] (E) at (0,0) {};
   \node [label = {[shift={(-0.35,-0.5)}]$\clVako$}] (C0) at (-2/2,2) {};
   \node [label = {[shift={(0.35,-0.5)}]$\clVaki$}] (C1) at (2/2,2) {};
   \node [label = {[shift={(0,-0.7)}]$\clVak$}] (C) at (0,4) {};
   \node [label = {[shift={(0.45,-0.55)}]$\clOO$}] (00) at (-4/2,4) {};
   \node [label = {[shift={(-0.45,-0.55)}]$\clII$}] (11) at (4/2,4) {};
   \node [label = {[shift={(-0.95,-0.25)},rotate=-21.8]$\clOOCI$}] (00C1) at (-2/2,6) {};
   \node [label = {[shift={(0.95,-0.25)},rotate=21.8]$\clIICO$}] (11C0) at (2/2,6) {};
   \node [label = above:{$\clEq$}] (Eq) at (0,8) {};

   \node [label = {[shift={(-0.3,-0.7)}]$\clMc$}] (Mc) at (-5,2) {};
   \node [label = {[shift={(-0.3,-0.7)}]$\clMo$}] (M0) at ($(Mc)+(-2/2,2)$) {};
   \node [label = {[shift={(0.3,-0.8)}]$\clMi$}] (M1) at ($(Mc)+(2/2,2)$) {};
   \node [label = {[shift={(-0.15,0)}]$\clM$}] (M) at ($(Mc)+(0,4)$) {};
   \node [label = {[shift={(-0.3,-0.7)}]$\clOI$}] (01) at ($(Mc)+(-5,2)$) {};
   \node [label = left:{$\clOICO$}] (01C0) at ($(01)+(-2/2,2)$) {};
   \node [label = {[shift={(0.6,-0.85)},rotate=-21.8]$\clOICI$}] (01C1) at ($(01)+(2/2,2)$) {};
   \node [label = {[shift={(0.35,-1.15)},rotate=-21.8]$\clOIC$}] (01C) at ($(01)+(0,4)$) {};
   \node [label = left:{$\clOX$}] (0X) at ($(01C0)+(-2/2,2)$) {};
   \node [label = left:{$\clOXCI$}] (0XC1) at ($(0X)+(2/2,2)$) {};
   \node [label = {[shift={(-0.4,-0.25)}]$\clXI$}] (X1) at ($(01C1)+(2/2,2)$) {};
   \node [label = {[shift={(0.8,-0.65)}]$\clXICO$}] (X1C0) at ($(X1)+(-2/2,2)$) {};
   \node [label = {[shift={(-0.4,-0.2)}]$\clEiio$}] (leq) at ($(X1C0)+(-2/2,2)$) {};

   \node [label = {[shift={(0.45,-0.75)}]$\clMcneg$}] (Mcneg) at (5,2) {};
   \node [label = {[shift={(0.45,-0.75)}]$\clMoneg$}] (M0neg) at ($(Mcneg)+(2/2,2)$) {};
   \node [label = {[shift={(-0.15,-0.8)},rotate=21.8]$\clMineg$}] (M1neg) at ($(Mcneg)+(-2/2,2)$) {};
   \node [label = {[shift={(0.15,-0.05)}]$\clMneg$}] (Mneg) at ($(Mcneg)+(0,4)$) {};
   \node [label = {[shift={(0.45,-0.75)}]$\clIO$}] (10) at ($(Mcneg)+(5,2)$) {};
   \node [label = right:{\,$\clIOCI$}] (10C1) at ($(10)+(2/2,2)$) {};
   \node [label = {[shift={(-0.6,-0.85)},rotate=21.8]$\clIOCO$}] (10C0) at ($(10)+(-2/2,2)$) {};
   \node [label = {[shift={(-0.35,-1.15)},rotate=21.8]$\clIOC$}] (10C) at ($(10)+(0,4)$) {};
   \node [label = right:{\,$\clIX$}] (1X) at ($(10C1)+(2/2,2)$) {};
   \node [label = right:{\,$\clIXCO$}] (1XC0) at ($(1X)+(-2/2,2)$) {};
   \node [label = {[shift={(0.45,-0.3)}]$\clXO$}] (X0) at ($(10C0)+(-2/2,2)$) {};
   \node [label = {[shift={(-0.8,-0.65)}]$\clXOCI$}] (X0C1) at ($(X0)+(2/2,2)$) {};
   \node [label = {[shift={(0.4,-0.2)}]$\clEioi$}] (geq) at ($(X0C1)+(2/2,2)$) {};

   \node [label = below:{$\clNeq$}] (Neq) at (0,11) {};
   \node [label = {[shift={(-0.5,-0.3)}]$\clEiii$}] (Neq00) at ($(Neq)+2.5*(-2/2,2)$) {};
   \node [label = {[shift={(0.5,-0.3)}]$\clEioo$}] (Neq11) at ($(Neq)+2.5*(2/2,2)$) {};

   \node [label = {[shift={(0,0.1)}]$\clAll$}] (all) at ($(Neq)+2.5*(0,4)$) {};

   \draw (E) -- (C0) -- (C) -- (00C1) -- (Eq);
   \draw (E) -- (C1) -- (C) -- (11C0) -- (Eq);
   \draw (C0) -- (00) -- (00C1);
   \draw (C1) -- (11) -- (11C0);
   \draw (E) -- (Mc) -- (M0) -- (M) -- (01C);
   \draw (Mc) -- (M1) -- (M);
   \draw (Mc) -- (01) -- (01C1) -- (01C);
   \draw (M1) -- (01C1);   \draw (M0) -- (01C0) -- (01C);   \draw (01) -- (01C0);   \draw (C0) -- (M0);   \draw (C1) -- (M1);   \draw (C) -- (M);

   \draw (E) -- (Mcneg) -- (M0neg) -- (Mneg) -- (10C);
   \draw (Mcneg) -- (M1neg) -- (Mneg);
   \draw (Mcneg) -- (10) -- (10C0) -- (10C);
   \draw (M1neg) -- (10C0);
   \draw (M0neg) -- (10C1) -- (10C);
   \draw (10) -- (10C1);   \draw (C1) -- (M0neg);   \draw (C0) -- (M1neg);   \draw (C) -- (Mneg);
   \draw (01) -- (Neq);   \draw (10) -- (Neq);   \draw (00) -- (0X);   \draw (11) -- (1X);   \draw (01C0) -- (0X);   \draw (10C1) -- (1X);
   \draw (00) -- (0X);   \draw (00) -- (X0);   \draw (11) -- (1X);   \draw (11) -- (X1);   \draw (01C0) -- (0X);   \draw (01C1) -- (X1);
   \draw (10C1) -- (1X);   \draw (10C0) -- (X0);   \draw (0X) -- (0XC1);   \draw (1X) -- (1XC0);   \draw (01C) -- (0XC1);   \draw (10C) -- (1XC0);
   \draw (00C1) -- (0XC1);   \draw (11C0) -- (1XC0);   \draw (00C1) -- (X0C1);   \draw (11C0) -- (X1C0);
   \draw (01C) -- (X1C0);   \draw (10C) -- (X0C1);   \draw (X1) -- (X1C0);   \draw (X0) -- (X0C1);
   \draw (Eq) -- (leq);   \draw (Eq) -- (geq);   \draw (0XC1) -- (leq);   \draw (1XC0) -- (geq);   \draw (X1C0) -- (leq);   \draw (X0C1) -- (geq);
   \draw (Neq) -- (Neq00) -- (all);   \draw (Neq) -- (Neq11) -- (all);
   \draw (0X) -- (Neq00);   \draw (1X) -- (Neq11);   \draw (X1) -- (Neq11);   \draw (X0) -- (Neq00);   \draw (leq) -- (all);   \draw (geq) -- (all);
\end{tikzpicture}
}
\end{center}
\caption{$(\protect\clMc,\protect\clSM)$\hyp{}clonoids.}
\label{fig:McSM}
\end{figure}

It remains to consider the cases where the target clone includes $\clVc$ or $\clLambdac$.

\begin{definition}
For $\vect{c} = c_0 c_1 c_2 \dots c_m \in \{0,1\}^{\{0, \dots, m\}}$, let $\lambda_\vect{c} \colon \{0,1\}^m \to \{0,1\}$ be the Boolean function defined by the rule $\lambda_\vect{c}(\vect{a}) = c_{w(\vect{a})}$, where $w(\vect{a})$ is the Hamming weight of $\vect{a}$.
It is clear that $\Alt(\lambda_{\vect{c}})$ equals the number of indices $i \in \{0, \dots, m-1\}$ such that $c_i \neq c_{i+1}$.
\end{definition}

The classes $\clIntVal{\clAlt{\leq n}}{a}{b}$ of Boolean functions with a bounded alternation number are not stable under left composition with $\clVc$, with some exceptions for small bounds $n$.

\begin{lemma}
\label{lem:Alt-apu}
For $n \geq 2$ and $a, b \in \{0,1\}$ with $a + n \equiv b \pmod{2}$, unless $n = 2$ and $a = b = 1$, there exists an $m > n$ such that $\clVc \, \clIntVal{\clAlt{n}}{a}{b} \cap \clVc \, \clIntVal{\clAlt{m}}{a}{b} \neq \emptyset$; in particular, $\clVc \, \clIntVal{\clAlt{n}}{a}{b} \nsubseteq \clAlt{\leq n}$.
\end{lemma}

\begin{proof}
Let $k \in \IN_{+}$.
Let $\vect{a} = (0100)^k 0$, $\vect{b} = (0001)^k 0$.
Then $\lambda_\vect{a}, \lambda_\vect{b} \in \clIntVal{\clAlt{2k}}{0}{0}$, but $\mathord{\vee}(\lambda_\vect{a},\lambda_\vect{b}) = \lambda_\vect{u}$, where $\vect{u} = (0101)^k 0$, and $\lambda_\vect{u} \in \clIntVal{\clAlt{4k}}{0}{0}$, so $\lambda_\vect{u} \notin \clAlt{\leq 2k}$.
Moreover,
$\lambda_{\vect{a} 1}, \lambda_{\vect{b} 1} \in \clIntVal{\clAlt{2k+1}}{0}{1}$, but $\mathord{\vee}( \lambda_{\vect{a} 1}, \lambda_{\vect{b} 1} ) = \lambda_{\vect{u} 1} \in \clIntVal{\clAlt{4k + 1}}{0}{1}$ and $\lambda_{\vect{u} 1} \notin \clAlt{\leq 2k + 1}$;
$\lambda_{1 \vect{a}}, \lambda_{1 \vect{b}} \in \clIntVal{\clAlt{2k+1}}{1}{0}$, but $\mathord{\vee}( \lambda_{1 \vect{a}}, \lambda_{1 \vect{b}} ) = \lambda_{1 \vect{u}} \in \clIntVal{\clAlt{4k+1}}{1}{0}$ and $\lambda_{1 \vect{u}} \notin \clAlt{\leq 2k + 2}$;
and
$\lambda_{1 \vect{a} 1}, \lambda_{1 \vect{b} 1} \in \clIntVal{\clAlt{2k+2}}{1}{1}$, but $\mathord{\vee}( \lambda_{1 \vect{a} 1}, \lambda_{1 \vect{b} 1} ) = \lambda_{1 \vect{u} 1} \in \clIntVal{\clAlt{4k+2}}{1}{1}$ and $\lambda_{1 \vect{u} 1} \notin \clAlt{\leq 2k + 2}$.
\end{proof}

\begin{proposition}
\label{prop:McVc}
There are precisely 56 $(\clMc,\clVc)$\hyp{}clonoids, and they are the following:
\begin{align*}
& \clAll,
\quad \clEioo \cup \clVako,
\quad \clEioo,
\quad \clEioi,
\quad \clEiio,
\quad \clEq,
\\
& \clOXCI,
\quad \clXOCI,
\quad \clIXCO,
\quad \clXICO,
\quad \clOIC,
\quad \clOICO,
\\
& \clOICI,
\quad \clIOC,
\quad \clIOCO,
\quad \clIOCI,
\quad \clOOCI,
\quad \clIICO,
\\
& \clOX,
\quad \clXO,
\quad \clIX,
\quad \clXI,
\quad \clOI,
\quad \clIO,
\quad \clOO,
\quad \clII,
\\
& \clM \cup \clMneg \cup \clII,
\quad \clM \cup \clMneg \cup \clIntVal{\clAlt{2}}{1}{1},
\quad \clMc \cup \clMcneg \cup \clII,
\quad \clMc \cup \clMcneg \cup \clIntVal{\clAlt{\leq 2}}{1}{1},
\\
& \clM \cup \clIntVal{\clAlt{2}}{1}{1},
\quad \clMo \cup \clIX,
\quad \clMo \cup \clII,
\quad \clMi \cup \clIntVal{\clAlt{2}}{1}{1},
\\
& \clMneg \cup \clIntVal{\clAlt{2}}{1}{1},
\quad \clMineg \cup \clXI,
\quad \clMineg \cup \clII,
\quad \clMoneg \cup \clIntVal{\clAlt{2}}{1}{1},
\\
& \clMc \cup \clIX,
\quad \clMc \cup \clII,
\quad \clMcneg \cup \clXI,
\quad \clMcneg \cup \clII,
\\
& \clM,
\quad \clMo,
\quad \clMi,
\quad \clMc,
\quad \clMneg,
\quad \clMineg,
\quad \clMoneg,
\quad \clMcneg,
\\
& \clIntVal{\clAlt{2}}{1}{1} \cup \clVak,
\quad \clIntVal{\clAlt{\leq 2}}{1}{1},
\quad \clVak,
\quad \clVako,
\quad \clVaki,
\quad \clEmpty.
\end{align*}
\end{proposition}

The Hasse diagram of $\closys{(\clMc,\clVc)}$ is shown in Figure~\ref{fig:McVc}.

\begin{figure}
\begin{center}
\makebox[0pt]{%
\scalebox{0.7}{%
\begin{tikzpicture}[baseline, scale=0.6]
   \tikzmath{\Ax = -1; \Ay = 2; \Bx = 1; \By = 2; \Cx = -7; \Cy = 2; \Dx = 7; \Dy = 2;}
   \node [vanha, label = {[shift={(0,-0.6)}]$\clEmpty$}] (0000) at (0,0) {};
   \node [vanha, label = {[shift={(-0.25,-0.6)}]$\clVako$}] (1000) at ($(0000)+(\Ax,\Ay)$) {};
   \node [vanha, label = {[shift={(-0.25,-0.65)}]$\clOO$}] (2000) at ($(1000)+(\Ax,\Ay)$) {};
   \node [vanha, label = {[shift={(0.25,-0.6)}]$\clVaki$}] (0100) at ($(0000)+(\Bx,\By)$) {};
   \node [vanha, label = {[shift={(0,-0.7)}]$\clVak$}] (1100) at ($(0100)+(\Ax,\Ay)$) {};
   \node [vanha, label = {[shift={(-0.9,-0.4)},rotate=-15.95]$\clOOCI$}] (2100) at ($(1100)+(\Ax,\Ay)$) {};
   \node [uusi, label = {[shift={(0.35,-0.7)}]$\clIntVal{\clAlt{\leq 2}}{1}{1}$}] (0200) at ($(0100)+(\Bx,\By)$) {};
   \node [uusi, label = {[shift={(-0.15,-1.05)}]$\clIntVal{\clAlt{2}}{1}{1} \cup \clVak$}] (1200) at ($(0200)+(\Ax,\Ay)$) {};
   \node [vanha, label = {[shift={(-0.4,-0.55)}]$\clII$}] (0300) at ($(0200)+(\Bx,\By)$) {};
   \node [vanha, label = {[shift={(0.8,-0.3)},rotate=15.95]$\clIICO$}] (1300) at ($(0300)+(\Ax,\Ay)$) {};
   \node [vanha, label = {[shift={(0.45,-0.55)}]$\clEq$}] (2300) at ($(1300)+(\Ax,\Ay)$) {};

   \node [vanha, label = {[shift={(0.3,-0.8)}]$\clMcneg$}] (0001) at ($(0000)+(\Dx,\Dy)$) {};
   \node [vanha, label = {[shift={(0.3,-0.8)}]$\clMoneg$}] (0101) at ($(0001)+(\Bx,\By)$) {};
   \node [uusi, label = {[shift={(1.2,-0.55)},rotate=15.95]$\clMoneg \cup \clIntVal{\clAlt{2}}{1}{1}$}] (0201) at ($(0101)+(\Bx,\By)$) {};
   \node [uusi, label = {[shift={(0.9,-0.4)},rotate=15.95]$\clMcneg \cup \clII$}] (0301) at ($(0201)+(\Bx,\By)$) {};
   \node [vanha, label = {[shift={(0.3,-0.8)}]$\clIO$}] (0002) at ($(0001)+(\Dx,\Dy)$) {};
   \node [vanha, label = right:{\,$\clIOCI$}] (0102) at ($(0002)+(\Bx,\By)$) {};
   \node [vanha, label = right:{\,$\clIX$}] (0302) at ($(0102)+(\Bx+\Bx,\By+\By)$) {};
   \node [vanha, label = {[shift={(-0.3,-0.7)}]$\clMc$}] (0010) at ($(0000)+(\Cx,\Cy)$) {};
   \node [vanha, label = {[shift={(-0.3,-0.7)}]$\clOI$}] (0020) at ($(0010)+(\Cx,\Cy)$) {};

   \node [vanha, label = {[shift={(-0.4,-0.25)}]$\clMineg$}] (1001) at ($(1000)+(\Dx,\Dy)$) {};
   \node [vanha, label = {[shift={(-0.25,-0.1)}]$\clMneg$}] (1101) at ($(1001)+(\Bx,\By)$) {};
   \node [uusi, label = {[shift={(-0.8,0.1)},rotate=-40.6]$\clMneg \cup \clIntVal{\clAlt{2}}{1}{1}$}] (1201) at ($(1101)+(\Bx,\By)$) {};
   \node [uusi, label = {[shift={(0.9,-0.4)},rotate=15.95]$\clMineg \cup \clII$}] (1301) at ($(1201)+(\Bx,\By)$) {};
   \node [vanha, label = {[shift={(-0.6,-0.75)},rotate=15.95]$\clIOCO$}] (1002) at ($(1001)+(\Dx,\Dy)$) {};
   \node [vanha, label = {[shift={(0.9,-0.45)},rotate=15.95]$\clIOC$}] (1102) at ($(1002)+(\Bx,\By)$) {};
   \node [vanha, label = right:{\,$\clIXCO$}] (1302) at ($(1102)+(\Bx+\Bx,\By+\By)$) {};
   \node [vanha, label = {[shift={(-0.2,-0.7)}]$\clMo$}] (1010) at ($(1000)+(\Cx,\Cy)$) {};
   \node [vanha, label = left:{$\clOICO$}] (1020) at ($(1010)+(\Cx,\Cy)$) {};

   \node [vanha, label = {[shift={(0.45,-0.35)}]$\clXO$}] (2002) at ($(2000)+2*(\Dx,\Dy)$) {};
   \node [vanha, label = {[shift={(-0.6,-0.75)},rotate=15.95]$\clXOCI$}] (2102) at ($(2002)+(\Bx,\By)$) {};
   \node [vanha, label = {[shift={(0.4,-0.2)}]$\clEioi$}] (2302) at ($(2102)+2*(\Bx,\By)$) {};

   \node [vanha, label = {[shift={(0.5,-0.35)}]$\clMi$}] (0110) at ($(0010)+(\Bx,\By)$) {};
   \node [vanha, label = {[shift={(0.65,-0.75)},rotate=-15.95]$\clOICI$}] (0120) at ($(0020)+(\Bx,\By)$) {};
   \node [uusi, label = {[shift={(0.85,-0.95)},rotate=-15.95]$\clMi \cup \clIntVal{\clAlt{2}}{1}{1}$}] (0210) at ($(0110)+(\Bx,\By)$) {};
   \node [uusi, label = {[shift={(1.05,-0.85)}]$\clMc \cup \clII$}] (0310) at ($(0210)+(\Bx,\By)$) {};

   \node [vanha, label = left:{$\clXI$}] (0320) at ($(0310)+(\Cx,\Cy)$) {};
   \node [vanha, label = {[shift={(-0.1,0)}]$\clM$}] (1110) at ($(1010)+(\Bx,\By)$) {};
   \node [vanha, label = {[shift={(0.9,-0.45)},rotate=15.95]$\clOIC$}] (1120) at ($(1020)+(\Bx,\By)$) {};
   \node [uusi, label = {[shift={(-0.9,-0.55)},rotate=-15.95]$\clM \cup \clIntVal{\clAlt{2}}{1}{1}$}] (1210) at ($(1110)+(\Bx,\By)$) {};
   \node [uusi, label = {[shift={(-0.85,-0.35)},rotate=-15.95]$\clMo \cup \clII$}] (1310) at ($(1210)+(\Bx,\By)$) {};
   \node [vanha, label = {[shift={(1,-0.75)}]$\clXICO$}] (1320) at ($(1120)+2*(\Bx,\By)$) {};
   \node [vanha, label = left:{$\clOX$}] (2020) at ($(1020)+(\Ax,\Ay)$) {};
   \node [vanha, label = left:{$\clOXCI$}] (2120) at ($(2020)+(\Bx,\By)$) {};
   \node [vanha, label = {[shift={(-0.3,-0.1)}]$\clEiio$}] (2320) at ($(2120)+2*(\Bx,\By)$) {};

   \node [uusi, label = {[shift={(-0.35,-1.8)},rotate=40.6]$\clMc \cup \clMcneg \cup \clIntVal{\clAlt{\leq 2}}{1}{1}$}] (0211) at ($(2300)+1*(\Bx,\By)$) {};
   \node [uusi, label = {[shift={(0.45,-1.85)},rotate=-40.6]$\clMc \cup \clMcneg \cup \clII$}] (0311) at ($(0211)+(\Bx,\By)$) {};
   \node [uusi, label = {[shift={(0.15,-1.5)},rotate=-33.69]$\clMc \cup \clIX$}] (0312) at ($(0311)+3*(\Bx,\By)-2*(\Ax,\Ay)$) {};
   \node [uusi, label = {[shift={(-0.9,-0.75)}]$\clMcneg \cup \clXI$}] (0321) at ($(0311)+3*(\Ax,\Ay)-2*(\Bx,\By)$) {};
   \node [vanha, label = {[shift={(-0.65,-0.5)}]$\clEioo$}] (0322) at ($(0302)!14/9!(0312)$) {};
   \node [uusi, label = {[shift={(-0.3,-1.4)},rotate=40.6]$\clM \cup \clMneg \cup \clIntVal{\clAlt{\leq 2}}{1}{1}$}] (1211) at ($(0211)+(\Ax,\Ay)$) {};
   \node [uusi, label = {[shift={(1.25,-1.15)}]$\clM \cup \clMneg \cup \clII$}] (1311) at ($(1211)+(\Bx,\By)$) {};
   \node [uusi, label = {[shift={(-0.15,-0.55)},rotate=-33.69]$\clMo \cup \clIX$}] (1312) at ($(0312)+(\Ax,\Ay)$) {};
   \node [uusi, label = {[shift={(-0.9,-0.75)}]$\clMineg \cup \clXI$}] (1321) at ($(0321)+(\Ax,\Ay)$) {};
   \node [uusi, label = {[shift={(-0.95,-0.8)}]$\clEioo \cup \clVako$}] (1322) at ($(0322)+(\Ax,\Ay)$) {};
   \node [vanha, label = {[shift={(0,0.1)}]$\clAll$}] (2322) at ($(1322)+(\Ax,\Ay)$) {};

    \draw (0000) -- (1000) -- (2000);
    \draw (0100) -- (1100) -- (2100);
    \draw (0200) -- (1200);
    \draw (0300) -- (1300) -- (2300);
    \draw (0000) -- (0100) -- (0200) -- (0300);
    \draw (1000) -- (1100) -- (1200) -- (1300);
    \draw (2000) -- (2100) -- (2300);

    \draw (0000) -- (0010) -- (0020);

    \draw (0000) -- (0001) -- (0002);

    \draw (0001) -- (0101) -- (0201) -- (0301);
    \draw (0002) -- (0102) -- (0302);
    \draw (0100) -- (0101) -- (0102);
    \draw (0200) -- (0201);
    \draw (0300) -- (0301) -- (0302);

    \draw (1001) -- (1101) -- (1201) -- (1301);
    \draw (1002) -- (1102) -- (1302);
    \draw (1000) -- (1001) -- (1002);
    \draw (1100) -- (1101) -- (1102);
    \draw (1200) -- (1201);
    \draw (2000) -- (2002);
    \draw (2100) -- (2102);
    \draw (2300) -- (2302);
    \draw (1300) -- (1301) -- (1302);
    \draw (2002) -- (2102) -- (2302);

    \draw (0001) -- (1001);
    \draw (0101) -- (1101);
    \draw (0201) -- (1201);
    \draw (0301) -- (1301);
    \draw (0002) -- (1002) -- (2002);
    \draw (0102) -- (1102) -- (2102);
    \draw (0302) -- (1302) -- (2302);

    \draw (1000) -- (1010) -- (1020);
    \draw (2000) -- (2020);
    \draw (0100) -- (0110) -- (0120);
    \draw (1100) -- (1110) -- (1120);
    \draw (2100) -- (2120);
    \draw (0200) -- (0210);
    \draw (1200) -- (1210);
    \draw (0300) -- (0310) -- (0320);
    \draw (1300) -- (1310) -- (1320);
    \draw (2300) -- (2320);

    \draw (0010) -- (0110) -- (0210) -- (0310);
    \draw (1010) -- (1110) -- (1210) -- (1310);
    \draw (0020) -- (0120) -- (0320);
    \draw (1020) -- (1120) -- (1320);
    \draw (2020) -- (2120) -- (2320);

    \draw (0010) -- (1010);
    \draw (0110) -- (1110);
    \draw (0210) -- (1210);
    \draw (0310) -- (1310);

    \draw (0020) -- (1020) -- (2020);
    \draw (0120) -- (1120) -- (2120);
    \draw (0320) -- (1320) -- (2320);

    \draw (0210) -- (0211);   \draw (0201) -- (0211);
    \draw (0310) -- (0311);   \draw (0301) -- (0311);
    \draw (0211) -- (0311);
    \draw (0311) -- (0312);
    \draw (0302) -- (0312);
    \draw (0311) -- (0321);
    \draw (0320) -- (0321);
    \draw (0312) -- (0322);
    \draw (0321) -- (0322);

    \draw (0211) -- (1211);
    \draw (1201) -- (1211);
    \draw (1210) -- (1211);

    \draw (0311) -- (1311);
    \draw (1211) -- (1311);
    \draw (1301) -- (1311);
    \draw (1310) -- (1311);

    \draw (0312) -- (1312);
    \draw (1302) -- (1312);
    \draw (1311) -- (1312);
    
    \draw (0321) -- (1321);
    \draw (1311) -- (1321);
    \draw (1320) -- (1321);

    \draw (0322) -- (1322);
    \draw (1312) -- (1322);
    \draw (1321) -- (1322);

    \draw (1322) -- (2322);
    \draw (2302) -- (2322);
    \draw (2320) -- (2322);
\end{tikzpicture}
	}}
\end{center}
\caption{$(\protect\clMc,\protect\clVc)$\hyp{}clonoids.}
\label{fig:McVc}
\end{figure}

\begin{proof}
By Theorem~\ref{thm:mon-clonoids}\ref{thm:mon-clonoids:Mc} and Lemma~\ref{lem:Alt-apu}, every $(\clMc,\clVc)$\hyp{}clonoid must be of the form $A \cup B \cup C \cup D$, where
\begin{align*}
&
A \in \{ \clEmpty, \clIntVal{\clAlt{0}}{0}{0}, \clOO \},
&&
B \in \{ \clEmpty, \clIntVal{\clAlt{0}}{1}{1}, \clIntVal{\clAlt{\leq 2}}{1}{1}, \clII \},
\\ &
C \in \{ \clEmpty, \clIntVal{\clAlt{1}}{0}{1}, \clOI \},
&&
D \in \{ \clEmpty, \clIntVal{\clAlt{1}}{1}{0}, \clIO \}.
\end{align*}
Not every set of this form is a $(\clMc,\clVc)$\hyp{}clonoid, though.
We can exclude several with the following observations.

\begin{claim}
\label{clm:MV}
Assume $K \subseteq \clAll$ is stable under left composition with $\clVc$.
\begin{enumerate}[label=\upshape{(\roman*)}]
\item\label{MV-1}
If $\clIntVal{\clAlt{1}}{0}{1} \cup \clIntVal{\clAlt{1}}{1}{0} \subseteq K$, then $\clIntVal{\clAlt{\leq 2}}{1}{1} \subseteq K$.

\item\label{MV-2}
If $\clIntVal{\clAll}{a}{b} \cup \clIntVal{\clAlt{\ell}}{c}{d} \subseteq K$ for some $a, b, c, d \in \{0,1\}$ and $\ell \in \{1,2\}$ with $(c,d) \neq (0,0)$ and $c + \ell \equiv d \pmod{2}$,
then $\clIntVal{\clAll}{r}{s} \subseteq K$, where $r = a \vee c$ and $s = b \vee d$.
\end{enumerate}
\end{claim}

\begin{pfclaim}
\ref{MV-1}
Let $f \in \clIntVal{\clAlt{\leq 2}}{1}{1}$.
If $f \in \clIntVal{\clAlt{0}}{1}{1} = \clVaki$, let $g$ and $h$ be the functions (of the same arity as $f$) given by the rules
$g(\vect{a}) = 0$ if and only if $\vect{a} = \vect{0}$, and $h(\vect{a}) = 0$ if and only if $\vect{a} = \vect{1}$.
If $f \in \clIntVal{\clAlt{2}}{1}{1}$, define $g$ and $h$ by the rules
\begin{align*}
&
g(\vect{a}) =
\begin{cases}
1, & \text{if $d_f(\vect{a}) = 2$,} \\
0, & \text{otherwise,}
\end{cases}
&&
h(\vect{a}) =
\begin{cases}
1, & \text{if $d_f(\vect{a}) = 0$,} \\
0, & \text{otherwise.}
\end{cases}
\end{align*}
In either case, $g \in \clIntVal{\clAlt{1}}{0}{1}$ and $h \in \clIntVal{\clAlt{1}}{1}{0}$, and $f = g \vee h \in \clVc (\clIntVal{\clAlt{1}}{0}{1} \cup \clIntVal{\clAlt{1}}{1}{0}) \subseteq \clVc K \subseteq K$.

\ref{MV-2}
Let $f \in \clIntVal{\clAll}{r}{s}$, and define functions $g$ and $h$ by the rules
\begin{align*}
&
g(\vect{a}) =
\begin{cases}
f(\vect{a}), & \text{if $\vect{a} \notin \{\vect{0}, \vect{1}\}$,} \\
a, & \text{if $\vect{a} = \vect{0}$,} \\
b, & \text{if $\vect{a} = \vect{1}$,}
\end{cases}
&&
h(\vect{a}) =
\begin{cases}
0, & \text{if $\vect{a} \notin \{\vect{0}, \vect{1}\}$,} \\
c, & \text{if $\vect{a} = \vect{0}$,} \\
d, & \text{if $\vect{a} = \vect{1}$.}
\end{cases}
\end{align*}
We have $g \in \clIntVal{\clAll}{a}{b}$ and $h \in \clIntVal{\clAlt{\ell}}{c}{d}$, and
$f = g \vee h \in \clVc (\clIntVal{\clAll}{a}{b} \cup \clIntVal{\clAlt{\ell}}{c}{d}) \subseteq \clVc K \subseteq K$.
\end{pfclaim}

We can exclude many quadruples $(A,B,C,D)$ with the help of Claim~\ref{clm:MV}.
Namely,
by Claim~\ref{clm:MV}\ref{MV-1}, we have that if $C \neq \clEmpty$ and $D \neq \clEmpty$, then $B \in \{ \clIntVal{\clAlt{\leq 2}}{1}{1}, \clII \}$.
By Claim~\ref{clm:MV}\ref{MV-2}, we have that
if $A = \clOO$ then $B \in \{ \clEmpty, \clIntVal{\clAlt{0}}{1}{1}, \clII \}$, $C \in \{ \clEmpty, \clOI \}$, and $D \in \{ \clEmpty, \clIO \}$;
moreover,
if $C = \clOI$ or $D = \clIO$, then $B \in \{ \clEmpty, \clIntVal{\clAlt{0}}{1}{1}, \clII \}$;
even futher, if ($C = \clOI$ and $D \neq \clEmpty$) or ($C \neq \clEmpty$ and $D = \clIO$), then $B = \clII$.

This leaves us with the 56 possible combinations for the quadruples $(A,B,C,D)$ and the corresponding classes $A \cup B \cup C \cup D$ that are presented in Table~\ref{table:McVc}.

\begin{table}
\begingroup
\small
\[
\begin{array}[t]{ccccc}
A & B & C & D & A \cup B \cup C \cup D \\
\hline
\clEmpty & \clEmpty & \clEmpty & \clEmpty & \clEmpty \\
\clEmpty & \clEmpty & \clEmpty & \clIntVal{\clAlt{1}}{1}{0} & \clMcneg \\
\clEmpty & \clEmpty & \clEmpty & \clIO & \clIO \\
\clEmpty & \clEmpty & \clIntVal{\clAlt{1}}{0}{1} & \clEmpty & \clMc \\
\clEmpty & \clEmpty & \clOI & \clEmpty & \clOI \\
\clEmpty & \clIntVal{\clAlt{0}}{1}{1} & \clEmpty & \clEmpty & \clVaki \\
\clEmpty & \clIntVal{\clAlt{0}}{1}{1} & \clEmpty & \clIntVal{\clAlt{1}}{1}{0} & \clMoneg \\
\clEmpty & \clIntVal{\clAlt{0}}{1}{1} & \clEmpty & \clIO & \clIOCI \\
\clEmpty & \clIntVal{\clAlt{0}}{1}{1} & \clIntVal{\clAlt{1}}{0}{1} & \clEmpty & \clMi \\
\clEmpty & \clIntVal{\clAlt{0}}{1}{1} & \clOI & \clEmpty & \clOICI \\
\clEmpty & \clIntVal{\clAlt{\leq 2}}{1}{1} & \clEmpty & \clEmpty & \clIntVal{\clAlt{\leq 2}}{1}{1} \\
\clEmpty & \clIntVal{\clAlt{\leq 2}}{1}{1} & \clEmpty & \clIntVal{\clAlt{1}}{1}{0} & \clMoneg \cup \clIntVal{\clAlt{2}}{1}{1} \\
\clEmpty & \clIntVal{\clAlt{\leq 2}}{1}{1} & \clIntVal{\clAlt{1}}{0}{1} & \clEmpty & \clMi \cup \clIntVal{\clAlt{2}}{1}{1} \\
\clEmpty & \clIntVal{\clAlt{\leq 2}}{1}{1} & \clIntVal{\clAlt{1}}{0}{1} & \clIntVal{\clAlt{1}}{1}{0} & \clMc \cup \clMcneg \cup \clIntVal{\clAlt{\leq 2}}{1}{1} \\
\clEmpty & \clII & \clEmpty & \clEmpty & \clII \\
\clEmpty & \clII & \clEmpty & \clIntVal{\clAlt{1}}{1}{0} & \clII \cup \clMcneg \\
\clEmpty & \clII & \clEmpty & \clIO & \clIX \\
\clEmpty & \clII & \clIntVal{\clAlt{1}}{0}{1} & \clEmpty & \clII \cup \clMc \\
\clEmpty & \clII & \clIntVal{\clAlt{1}}{0}{1} & \clIntVal{\clAlt{1}}{1}{0} & \clII \cup \clMc \cup \clMcneg \\
\clEmpty & \clII & \clIntVal{\clAlt{1}}{0}{1} & \clIO & \clIX \cup \clMc \\
\clEmpty & \clII & \clOI & \clEmpty & \clXI \\
\clEmpty & \clII & \clOI & \clIntVal{\clAlt{1}}{1}{0} & \clXI \cup \clMcneg \\
\clEmpty & \clII & \clOI & \clIO & \clEioo \\
\clIntVal{\clAlt{0}}{0}{0} & \clEmpty & \clEmpty & \clEmpty & \clVako \\
\clIntVal{\clAlt{0}}{0}{0} & \clEmpty & \clEmpty & \clIntVal{\clAlt{1}}{1}{0} & \clMineg \\
\clIntVal{\clAlt{0}}{0}{0} & \clEmpty & \clEmpty & \clIO & \clIOCO \\
\clIntVal{\clAlt{0}}{0}{0} & \clEmpty & \clIntVal{\clAlt{1}}{0}{1} & \clEmpty & \clMo \\
\clIntVal{\clAlt{0}}{0}{0} & \clEmpty & \clOI & \clEmpty & \clOICO \\
\end{array}
\,\,
\begin{array}[t]{ccccc}
A & B & C & D & A \cup B \cup C \cup D \\
\hline
\clIntVal{\clAlt{0}}{0}{0} & \clIntVal{\clAlt{0}}{1}{1} & \clEmpty & \clEmpty & \clVak \\
\clIntVal{\clAlt{0}}{0}{0} & \clIntVal{\clAlt{0}}{1}{1} & \clEmpty & \clIntVal{\clAlt{1}}{1}{0} & \clMneg \\
\clIntVal{\clAlt{0}}{0}{0} & \clIntVal{\clAlt{0}}{1}{1} & \clEmpty & \clIO & \clIOC \\
\clIntVal{\clAlt{0}}{0}{0} & \clIntVal{\clAlt{0}}{1}{1} & \clIntVal{\clAlt{1}}{0}{1} & \clEmpty & \clM \\
\clIntVal{\clAlt{0}}{0}{0} & \clIntVal{\clAlt{0}}{1}{1} & \clOI & \clEmpty & \clOIC \\
\clIntVal{\clAlt{0}}{0}{0} & \clIntVal{\clAlt{\leq 2}}{1}{1} & \clEmpty & \clEmpty & \clIntVal{\clAlt{\leq 2}}{1}{1} \cup \clVako \\
\clIntVal{\clAlt{0}}{0}{0} & \clIntVal{\clAlt{\leq 2}}{1}{1} & \clEmpty & \clIntVal{\clAlt{1}}{1}{0} & \clMneg \cup \clIntVal{\clAlt{2}}{1}{1} \\
\clIntVal{\clAlt{0}}{0}{0} & \clIntVal{\clAlt{\leq 2}}{1}{1} & \clIntVal{\clAlt{1}}{0}{1} & \clEmpty & \clM \cup \clIntVal{\clAlt{2}}{1}{1} \\
\clIntVal{\clAlt{0}}{0}{0} & \clIntVal{\clAlt{\leq 2}}{1}{1} & \clIntVal{\clAlt{1}}{0}{1} & \clIntVal{\clAlt{1}}{1}{0} & \clM \cup \clMneg \cup \clIntVal{\clAlt{\leq 2}}{1}{1} \\
\clIntVal{\clAlt{0}}{0}{0} & \clII & \clEmpty & \clEmpty & \clIICO \\
\clIntVal{\clAlt{0}}{0}{0} & \clII & \clEmpty & \clIntVal{\clAlt{1}}{1}{0} & \clII \cup \clMineg \\
\clIntVal{\clAlt{0}}{0}{0} & \clII & \clEmpty & \clIO & \clIXCO \\
\clIntVal{\clAlt{0}}{0}{0} & \clII & \clIntVal{\clAlt{1}}{0}{1} & \clEmpty & \clII \cup \clMo \\
\clIntVal{\clAlt{0}}{0}{0} & \clII & \clIntVal{\clAlt{1}}{0}{1} & \clIntVal{\clAlt{1}}{1}{0} & \clM \cup \clMneg \cup \clII \\
\clIntVal{\clAlt{0}}{0}{0} & \clII & \clIntVal{\clAlt{1}}{0}{1} & \clIO & \clIX \cup \clMo \\
\clIntVal{\clAlt{0}}{0}{0} & \clII & \clOI & \clEmpty & \clXICO \\
\clIntVal{\clAlt{0}}{0}{0} & \clII & \clOI & \clIntVal{\clAlt{1}}{1}{0} & \clXI \cup \clMineg \\
\clIntVal{\clAlt{0}}{0}{0} & \clII & \clOI & \clIO & \clEioo \cup \clVako \\
\clOO & \clEmpty & \clEmpty & \clEmpty & \clOO \\
\clOO & \clEmpty & \clEmpty & \clIO & \clXO \\
\clOO & \clEmpty & \clOI & \clEmpty & \clOX \\
\clOO & \clIntVal{\clAlt{0}}{1}{1} & \clEmpty & \clEmpty & \clOOCI \\
\clOO & \clIntVal{\clAlt{0}}{1}{1} & \clEmpty & \clIO & \clXOCI \\
\clOO & \clIntVal{\clAlt{0}}{1}{1} & \clOI & \clEmpty & \clOXCI \\
\clOO & \clII & \clEmpty & \clEmpty & \clEq \\
\clOO & \clII & \clEmpty & \clIO & \clEioi \\
\clOO & \clII & \clOI & \clEmpty & \clEiio \\
\clOO & \clII & \clOI & \clIO & \clAll \\
\end{array}
\]
\endgroup
\caption{The possible combinations of $A$, $B$, $C$, $D$ in the proof of Proposition~\ref{prop:McVc} and the 56 $(\protect\clMc,\protect\clVc)$\hyp{}clonoids.}
\label{table:McVc}
\end{table}

It remains to verify that these classes are indeed $(\clMc,\clVc)$\hyp{}clonoids.
For this, we only need to prove stability under left composition with $\clVc$.
Because intersections of $(\clMc,\clVc)$\hyp{}clonoids are $(\clMc,\clVc)$\hyp{}clonoids,
it suffices to verify this only for the 12 meet\hyp{}irreducible classes
$\clAll$,
$\clEiio$,
$\clEioi$,
$\clEioo \cup \clVako$,
$\clEioo$,
$\clOXCI$,
$\clXOCI$,
$\clOX$,
$\clXO$,
$\clXI \cup \clMineg$,
$\clIX \cup \clMo$,
$\clMo \cup \clMineg \cup \clIntVal{\clAlt{\leq 2}}{1}{1}$.
We give below the proof for $\clMo \cup \clMineg \cup \clIntVal{\clAlt{\leq 2}}{1}{1}$.
For the remaining classes, the proof is straightforward verification, and we leave it for the reader.

\begin{claim}
\label{clm:MMA}
$\clMo \cup \clMineg \cup \clIntVal{\clAlt{\leq 2}}{1}{1}$ is stable under left composition with $\clVc$.
\end{claim}

\begin{pfclaim}[Proof of Claim~\ref{clm:MMA}]
By Lemma~\ref{lem:CL-Lc-3.3}, it is enough to show that
$\mathord{\vee}(f,g) \in \clMo \cup \clMineg \cup \clIntVal{\clAlt{\leq 2}}{1}{1}$ whenever $f, g \in \clMo \cup \clMineg \cup \clIntVal{\clAlt{\leq 2}}{1}{1}$.
So, let $f, g \in \clMo \cup \clMineg \cup \clIntVal{\clAlt{\leq 2}}{1}{1}$.
If $f = 0$ or $g = 0$, then $f \vee g$ equals $g$ or $f$ and hence clearly belongs to $\clMo \cup \clMineg \cup \clIntVal{\clAlt{\leq 2}}{1}{1}$.
We may now assume that neither $f$ nor $g$ is a constant $0$ function, that is, $f, g \in \clMc \cup \clMcneg \cup \clIntVal{\clAlt{\leq 2}}{1}{1}$.
Because $(f(\vect{0}),f(\vect{1})) \neq (0,0)$ and $(g(\vect{0}),g(\vect{1})) \neq (0,0)$, it follows that $((f \vee g)(\vect{0}),(f \vee g)(\vect{1})) \neq (0,0)$.

It suffices to show that $\Alt(f \vee g) \leq 2$.
Suppose, to the contrary, that $\Alt(f \vee g) \geq 3$.
Then there exist tuples $\vect{a}, \vect{b}, \vect{c}$ such that $\vect{0} < \vect{a} < \vect{b} < \vect{c}$ and
\[
(f \vee g)(\vect{0}) \neq 
(f \vee g)(\vect{a}) \neq
(f \vee g)(\vect{b}) \neq
(f \vee g)(\vect{c}).
\]
We consider two cases according to the value of $f \vee g$ at $\vect{0}$.

If $(f \vee g)(\vect{0}) = 0$ (and hence $(f \vee g)(\vect{a}) = (f \vee g)(\vect{c}) = 1$ and $(f \vee g)(\vect{b}) = 0$),
then $f(\vect{0}) = g(\vect{0}) = f(\vect{b}) = g(\vect{b}) = 0$
and $f(\vect{a}) = 1$ or $g(\vect{b}) = 1$; without loss of generality, assume $f(\vect{a}) = 1$.
Because $(f(\vect{0}),f(\vect{1})) \neq (0,0)$, it follows that $\Alt(f) \geq 3$, a contradiction.

If $(f \vee g)(\vect{0}) = 1$ (and hence $(f \vee g)(\vect{a}) = (f \vee g)(\vect{c}) = 0$ and $(f \vee g)(\vect{b}) = 1$),
then $f(\vect{a}) = g(\vect{a}) = f(\vect{c}) = g(\vect{c}) = 0$
and $f(\vect{b}) = 1$ or $g(\vect{b}) = 1$; without loss of generality, assume $f(\vect{b}) = 1$.
Because $(f(\vect{0}),f(\vect{1})) \neq (0,0)$, it follows that $\Alt(f) \geq 3$, a contradiction.
\end{pfclaim}
\vspace{-1.5\baselineskip}
\end{proof}

\begin{proposition}
\label{prop:McLambdac}
There are precisely 56 $(\clMc,\clLambdac)$\hyp{}clonoids, and they are precisely the duals of the $(\clMc,\clVc)$\hyp{}clonoids described in Proposition~\ref{prop:McVc}.
\end{proposition}

\begin{proof}
This follows by duality from Proposition~\ref{prop:McVc}.
\end{proof}

In order to complete the picture, we are still going describe all $(C_1,C_2)$\hyp{}clonoids for clones $C_1$ and $C_2$, where $\clMc \subseteq C_1$ and $\clVc \subseteq C_2$.
We achieve this by determining, for each $(\clMc,\clVc)$\hyp{}clonoid $K$, the unique clones $C_1^K$ and $C_2^K$ with the property that for all clones $C_1$ and $C_2$, $K C_1 \subseteq K$ if and only if $C_1 \subseteq C_1^K$, and $C_2 K \subseteq K$ if and only if $C_2 \subseteq C_2^K$ (see Lemma~\ref{lem:CAKCBK}).

\begin{theorem}
\label{thm:McVc-stab-gen}
For each $(\clMc,\clVc)$\hyp{}clonoid $K$, as determined in Proposition~\ref{prop:McVc}, the clones $C_1^K$ and $C_2^K$ prescribed in Table~\ref{table:McVc-stability} have the property that for every clone $C$, it holds that $KC \subseteq K$ if and only if $C \subseteq C_1^K$, and $CK \subseteq K$ if and only if $C \subseteq C_2^K$.
\end{theorem}

\begin{table}
\newcommand{\sSM}{\cite{Lehtonen-SM}}
\newcommand{\sTH}{\ref{thm:McVc-stab-gen}}
\begingroup
\small
\makebox[0pt]{%
\begin{tabular}{llllllll}
\toprule
$K$ & $KC \subseteq K$ & $CK \subseteq K$ & source & $K$ & $KC \subseteq K$ & $CK \subseteq K$ & source \\
& iff $C \subseteq \ldots$ & iff $C \subseteq \ldots$ & & & iff $C \subseteq \ldots$ & iff $C \subseteq \ldots$ & \\
\midrule
$\clAll$ & $\clAll$ & $\clAll$ & \sSM &
      $\clMc \cup \clMcneg \cup \clII$ & $\clMc$ & $\clVi$ & \sTH \\
$\clEioo \cup \clVako$ & $\clOI$ & $\clV$ & \sTH &
      $\clMc \cup \clMcneg \cup \clIntVal{\clAlt{\leq 2}}{1}{1}$ & $\clMc$ & $\clVi$ & \sTH \\
$\clEioo$ & $\clOI$ & $\clWk{2}$ & \sSM &
      $\clM \cup \clIntVal{\clAlt{2}}{1}{1}$ & $\clMi$ & $\clV$ & \sTH \\
$\clEiio$ & $\clOI$ & $\clM$ & \sSM &
      $\clMo \cup \clIX$ & $\clMo$ & $\clV$ & \sTH \\
$\clEioi$ & $\clOI$ & $\clM$ & \sSM &
      $\clMo \cup \clII$ & $\clMc$ & $\clV$ & \sTH \\
$\clEq$ & $\clOI$ & $\clAll$ & \sSM &
      $\clMi \cup \clIntVal{\clAlt{2}}{1}{1}$ & $\clMi$ & $\clVi$ & \sTH \\
$\clOXCI$ & $\clOX$ & $\clM$ & \sSM &
      $\clMneg \cup \clIntVal{\clAlt{2}}{1}{1}$ & $\clMo$ & $\clVi$ & \sTH \\
$\clXOCI$ & $\clXI$ & $\clM$ & \sSM &
      $\clMineg \cup \clXI$ & $\clMi$ & $\clV$ & \sTH \\
$\clIXCO$ & $\clOX$ & $\clM$ & \sSM &
      $\clMineg \cup \clII$ & $\clMc$ & $\clV$ & \sTH \\
$\clXICO$ & $\clXI$ & $\clM$ & \sSM &
      $\clMoneg \cup \clIntVal{\clAlt{2}}{1}{1}$ & $\clMo$ & $\clVi$ & \sTH \\
$\clOIC$ & $\clOI$ & $\clM$ & \sSM &
      $\clMc \cup \clIX$ & $\clMc$ & $\clVi$ & \sTH \\
$\clOICO$ & $\clOI$ & $\clMo$ & \sSM &
      $\clMc \cup \clII$ & $\clMc$ & $\clVi$ & \sTH \\
$\clOICI$ & $\clOI$ & $\clMi$ & \sSM &
      $\clMcneg \cup \clXI$ & $\clMc$ & $\clVi$ & \sTH \\
$\clIOC$ & $\clOI$ & $\clM$ & \sSM &
      $\clMcneg \cup \clII$ & $\clMc$ & $\clVi$ & \sTH \\
$\clIOCO$ & $\clOI$ & $\clMo$ & \sSM &
      $\clM$ & $\clM$ & $\clM$ & \sSM \\
$\clIOCI$ & $\clOI$ & $\clMi$ & \sSM &
      $\clMo$ & $\clMo$ & $\clMo$ & \sSM \\
$\clOOCI$ & $\clOI$ & $\clM$ & \sSM &
      $\clMi$ & $\clMi$ & $\clMi$ & \sSM \\
$\clIICO$ & $\clOI$ & $\clM$ & \sSM &
      $\clMc$ & $\clMc$ & $\clMc$ & \sSM \\
$\clOX$ & $\clOX$ & $\clOX$ & \sSM &
      $\clMneg$ & $\clM$ & $\clM$ & \sSM \\
$\clXO$ & $\clXI$ & $\clOX$ & \sSM &
      $\clMineg$ & $\clMi$ & $\clMo$ & \sSM \\
$\clIX$ & $\clOX$ & $\clXI$ & \sSM &
      $\clMoneg$ & $\clMo$ & $\clMi$ & \sSM \\
$\clXI$ & $\clXI$ & $\clXI$ & \sSM &
      $\clMcneg$ & $\clMc$ & $\clMc$ & \sSM \\
$\clOI$ & $\clOI$ & $\clOI$ & \sSM &
      $\clIntVal{\clAlt{2}}{1}{1} \cup \clVak$ & $\clMc$ & $\clV$ & \sTH \\
$\clIO$ & $\clOI$ & $\clOI$ & \sSM &
      $\clIntVal{\clAlt{\leq 2}}{1}{1}$ & $\clMc$ & $\clVi$ & \sTH \\
$\clOO$ & $\clOI$ & $\clOX$ & \sSM &
      $\clVak$ & $\clAll$ & $\clAll$ & \sSM \\
$\clII$ & $\clOI$ & $\clXI$ & \sSM &
      $\clVako$ & $\clAll$ & $\clOX$ & \sSM \\
$\clM \cup \clMneg \cup \clII$ & $\clMc$ & $\clV$ & \sTH &
      $\clVaki$ & $\clAll$ & $\clXI$ & \sSM \\
$\clM \cup \clMneg \cup \clIntVal{\clAlt{2}}{1}{1}$ & $\clM$ & $\clV$ & \sTH &
      $\clEmpty$ & $\clAll$ & $\clAll$ & \sSM \\
\bottomrule
\end{tabular}
}
\endgroup
\caption{$(\protect\clMc,\protect\clVc)$\hyp{}clonoids and their stability under right and left composition with clones of Boolean functions.}
\label{table:McVc-stability}
\end{table}

For 37 out of the 56 $(\clMc,\clVc)$\hyp{}clonoids, the stability has already been determined in \cite[Theorem~5.1, Table~1]{Lehtonen-SM}, as indicated in Table~\ref{table:McVc-stability}.
We focus on the remaining 19 clonoids,
which are all contained in the interval $[\clIntVal{\clAlt{\leq 2}}{1}{1}, \clEioo \cup \clVako]$ in $\closys{(\clMc,\clVc)}$.

\begin{lemma}
\label{lem:stab-apu}
\leavevmode
\begin{enumerate}[label=\upshape{(\roman*)}]
\item\label{lem:stab-apu:Lneg}
$\{\neg\} \clIntVal{\clAlt{2}}{1}{1} \nsubseteq \clEioo \cup \clVako$.
\item\label{lem:stab-apu:LLambda}
$\{\mathord{\wedge}\} \clIntVal{\clAlt{2}}{1}{1} \nsubseteq \clEiii \cup \clIntVal{\clAlt{\leq 2}}{1}{1}$
and
$\{ \veewedge \} \clIntVal{\clAlt{2}}{1}{1} \nsubseteq \clEiii \cup \clIntVal{\clAlt{\leq 2}}{1}{1}$.
\item\label{lem:stab-apu:LMMneg}
$\{\mathord{\wedge}\} ( \clMc \cup \clMcneg ) \nsubseteq \clEioo \cup \clVako$.
\item\label{lem:stab-apu:LMA2}
$\{\mathord{\wedge}\} ( \clMc \cup \clIntVal{\clAlt{2}}{1}{1} ) \nsubseteq \clEioi \cup \clM$
and
$\{ \veewedge \} ( \clMc \cup \clIntVal{\clAlt{2}}{1}{1} ) \nsubseteq \clEioi \cup \clM$.
\item\label{lem:stab-apu:LMnegA2}
$\{\mathord{\wedge}\} \clMcneg \cup \clIntVal{\clAlt{2}}{1}{1} \nsubseteq \clEiio \cup \clMneg$
and
$\{ \veewedge \} \clMcneg \cup \clIntVal{\clAlt{2}}{1}{1} \nsubseteq \clEiio \cup \clMneg$.
\item\label{lem:stab-apu:L0}
If $0 \in K$ and $\{\mathord{\wedge}\} K \nsubseteq K$, then $\{ \veewedge \} K \nsubseteq K$.
\end{enumerate}
\end{lemma}

\begin{proof}
\ref{lem:stab-apu:Lneg}
We have $\mathord{\leftrightarrow} \in \clIntVal{\clAlt{2}}{1}{1}$, but $\mathord{\neg}(\mathord{\leftrightarrow}) = \mathord{+} \notin \clEioo \cup \clVako$.

\ref{lem:stab-apu:LLambda}
We have $\lambda_{11101}, \lambda_{10111}, \lambda_{10001} \in \clIntVal{\clAlt{2}}{1}{1}$, but
$\mathord{\wedge}(\lambda_{11101},\lambda_{10111}) = \lambda_{10101} \in \clIntVal{\clAlt{4}}{1}{1}$
and
$\lambda_{10001} \vee (\lambda_{11101} \wedge \lambda_{10111}) = \lambda_{10101} \in \clIntVal{\clAlt{4}}{1}{1}$,
and these are not in $\clEiii \cup \clIntVal{\clAlt{\leq 2}}{1}{1}$.

\ref{lem:stab-apu:LMMneg}
We have $\lambda_{011} \in \clMc$ and $\lambda_{110} \in \clMcneg$, but
$\mathord{\wedge}(\lambda_{011}, \lambda_{110}) = \lambda_{010} \in \clIntVal{\clAlt{2}}{0}{0}$, which is disjoint from $\clEioo \cup \clVako$.

\ref{lem:stab-apu:LMA2}
We have $\lambda_{0111}, \lambda_{0001} \in \clMc$ and $\lambda_{1101} \in \clIntVal{\clAlt{2}}{1}{1}$, but
$\mathord{\wedge}(\lambda_{0111}, \lambda_{1101}) = \lambda_{0101} \in \clIntVal{\clAlt{3}}{0}{1}$
and
$\lambda_{0001} \vee (\lambda_{0111} \wedge \lambda_{1101}) = \lambda_{0101} \in \clIntVal{\clAlt{3}}{0}{1}$,
and these are not in $\clEioi \cup \clM$.

\ref{lem:stab-apu:LMnegA2}
The proof is similar to \ref{lem:stab-apu:LMA2}.

\ref{lem:stab-apu:L0}
Because $\{\mathord{\wedge}\} K \nsubseteq K$, there exist $f, g \in K$ such that $f \wedge g \notin K$.
Consequently, $0 \vee (f \wedge g) = f \wedge g \notin K$, which shows $\{ \veewedge \} K \nsubseteq K$.
\end{proof}

\begin{lemma}
\label{lem:stab-apu-R}
\leavevmode
\begin{enumerate}[label=\upshape{(\roman*)}]
\item\label{lem:stab-apu:Rneg}
$\clIntVal{\clAlt{2}}{1}{1} \, \clonegen{\mathord{\neg}} \nsubseteq \clEioo \cup \clVako$.
\item\label{lem:stab-apu:Rplus1}
$\clIntVal{\clAlt{2}}{1}{1} \, \clonegen{\mathord{+_3}} \nsubseteq \clEiii \cup \clIntVal{\clAlt{\leq 2}}{1}{1}$.
\item\label{lem:stab-apu:Rplus2}
$\{\id\} \, \clonegen{\mathord{+_3}} \nsubseteq \clEioi \cup \clMc$.
\item\label{lem:stab-apu:Rplus3}
$\{\neg\} \, \clonegen{\mathord{+_3}} \nsubseteq \clEiio \cup \clMcneg$.
\item\label{lem:stab-apu:R0a}
$\{\id\} \, \clonegen{0} \nsubseteq \clEioo$.
\item\label{lem:stab-apu:R0b}
$\clIntVal{\clAlt{2}}{1}{1} \, \clonegen{0} \nsubseteq \clEiio$.
\item\label{lem:stab-apu:R0c}
$\clIntVal{\clAlt{4}}{1}{1} \, \clonegen{0} \nsubseteq \clEiio \cup \clMcneg$.
\item\label{lem:stab-apu:R0d}
$\clIntVal{\clAlt{3}}{0}{1} \, \clonegen{0} \nsubseteq \clEioo \cup \clVako$.
\item\label{lem:stab-apu:R1a}
$\{\neg\} \, \clonegen{1} \nsubseteq \clEiii$.
\item\label{lem:stab-apu:R1b}
$\clIntVal{\clAlt{2}}{1}{1} \, \clonegen{1} \nsubseteq \clEioi$.
\item\label{lem:stab-apu:R1c}
$\clIntVal{\clAlt{4}}{1}{1} \, \clonegen{1} \nsubseteq \clEioi \cup \clMc$.
\item\label{lem:stab-apu:R1d}
$\clIntVal{\clAlt{3}}{1}{0} \, \clonegen{1} \nsubseteq \clEioo \cup \clVako$.
\end{enumerate}
\end{lemma}

\begin{proof}
We apply Lemma~\ref{lem:CL-Lc-3.2}.

\ref{lem:stab-apu:Rneg}
We have $\lambda_{101} = x_1 + x_2 + 1 \in \clIntVal{\clAlt{2}}{1}{1}$, but
$\lambda_{101} \ast \mathord{\neg} = \mathord{+} \in \clOO \setminus \clVako$,
and this is not in $\clEioo \cup \clVako$.

\ref{lem:stab-apu:Rplus1}
We have $\lambda_{101} \in \clIntVal{\clAlt{2}}{1}{1}$, but
$\lambda_{101} \ast \mathord{+_3} = \lambda_{10101} \in \clIntVal{\clAlt{4}}{1}{1}$,
and this is not in $\clEiii \cup \clIntVal{\clAlt{\leq 2}}{1}{1}$.

\ref{lem:stab-apu:Rplus2}
We have $\id \ast \mathord{+_3} = \mathord{+_3} = \lambda_{0101}$,
which is not in $\clEioi \cup \clMc$.

\ref{lem:stab-apu:Rplus3}
The proof is similar to \ref{lem:stab-apu:Rplus2}.

\ref{lem:stab-apu:R0a}
We have $\id \ast 0 = 0 \notin \clEioo$.

\ref{lem:stab-apu:R0b}
We have $\lambda_{101} \in \clIntVal{\clAlt{2}}{1}{1}$, but
$\lambda_{101} \ast 0 = \lambda_{10} \notin \clEiio$.

\ref{lem:stab-apu:R0c}
We have $\lambda_{10101} \in \clIntVal{\clAlt{4}}{1}{1}$, but
$\lambda_{10101} \ast 0 = \lambda_{1010} \notin \clEiio \cup \clMcneg$.

\ref{lem:stab-apu:R0d}
We have $\mathord{+_3} = \lambda_{0101} \in \clIntVal{\clAlt{3}}{0}{1}$, but
$\lambda_{0101} \ast 0 = \lambda_{010} \notin \clEioo \cup \clVako$.

\ref{lem:stab-apu:R1a}
The proof is similar to \ref{lem:stab-apu:R0a}.

\ref{lem:stab-apu:R1b}
The proof is similar to \ref{lem:stab-apu:R0b}.

\ref{lem:stab-apu:R1c}
The proof is similar to \ref{lem:stab-apu:R0c}.

\ref{lem:stab-apu:R1d}
The proof is similar to \ref{lem:stab-apu:R0d}.
\end{proof}

\begin{proof}[Proof of Theorem~\ref{thm:McVc-stab-gen}]
For 37 out of the 56 $(\clMc,\clVc)$\hyp{}clonoids $K$, the clones $C_1^K$ and $C_2^K$ have already been determined in \cite[Theorem~5.1, Table~1]{Lehtonen-SM}.
These are indicated in Table~\ref{table:McVc-stability} with ``\cite{Lehtonen-SM}'' in the source column and in Figure~\ref{fig:McVc} as the vertices drawn as hollow circles.
From now on, we assume that $K$ is one of the remaining $(\clMc,\clVc)$\hyp{}clonoids:
$\clEioo \cup \clVako$,
$\clM \cup \clMneg \cup \clII$,
$\clM \cup \clMneg \cup \clIntVal{\clAlt{2}}{1}{1}$,
$\clMc \cup \clMcneg \cup \clII$,
$\clMc \cup \clMcneg \cup \clIntVal{\clAlt{\leq 2}}{1}{1}$,
$\clM \cup \clIntVal{\clAlt{2}}{1}{1}$,
$\clMo \cup \clIX$,
$\clMo \cup \clII$,
$\clMi \cup \clIntVal{\clAlt{2}}{1}{1}$,
$\clMneg \cup \clIntVal{\clAlt{2}}{1}{1}$,
$\clMineg \cup \clXI$,
$\clMineg \cup \clII$,
$\clMoneg \cup \clIntVal{\clAlt{2}}{1}{1}$,
$\clMc \cup \clIX$,
$\clMc \cup \clII$,
$\clMcneg \cup \clXI$,
$\clMcneg \cup \clII$,
$\clIntVal{\clAlt{2}}{1}{1} \cup \clVak$,
$\clIntVal{\clAlt{\leq 2}}{1}{1}$.

We focus first on stability on left composition with a clone.
We obviously have $\clVc \, K \subseteq K$ because $K$ is an $(\clMc, \clVc)$\hyp{}clonoid.
Moreover, $\clVaki \subseteq K$, so $\{1\} K \subseteq \clVaki \subseteq K$
On the other hand, we have $\{0\} K \subseteq K$ if and only if $\clVako \subseteq K$.
This condition holds precisely when $K$ is one of
$\clEioo \cup \clVako$,
$\clM \cup \clMneg \cup \clII$,
$\clM \cup \clMneg \cup \clIntVal{\clAlt{2}}{1}{1}$,
$\clM \cup \clIntVal{\clAlt{2}}{1}{1}$,
$\clMo \cup \clIX$,
$\clMo \cup \clII$,
$\clMneg \cup \clIntVal{\clAlt{2}}{1}{1}$,
$\clMineg \cup \clXI$,
$\clMineg \cup \clII$,
$\clIntVal{\clAlt{2}}{1}{1} \cup \clVak$,
which we will refer to as the ``clonoids with zero''.

Because $\clIntVal{\clAlt{\leq 2}}{1}{1} \subseteq K \subseteq \clEioo \cup \clVako$,
it follows from Lemma~\ref{lem:stab-apu}\ref{lem:stab-apu:Lneg} that $\{\neg\} K \nsubseteq K$.

Because one of the conditions
\begin{align*}
\clIntVal{\clAlt{2}}{1}{1} \subseteq K & \subseteq \clEiii \cup \clIntVal{\clAlt{\leq 2}}{1}{1}, &
\clMc \cup \clMcneg \subseteq K & \subseteq \clEioo \cup \clVako, \\
\clMc \cup \clIntVal{\clAlt{2}}{1}{1} \subseteq K & \subseteq \clEioi \cup \clM, &
\clMcneg \cup \clIntVal{\clAlt{2}}{1}{1} \subseteq K & \subseteq \clEiio \cup \clMneg
\end{align*}
holds,
it follows from Lemma~\ref{lem:stab-apu}\ref{lem:stab-apu:LLambda}--\ref{lem:stab-apu:LMnegA2}
that $\{\mathord{\wedge}\} K \nsubseteq K$.

Moreover, because one of the conditions
\begin{align*}
\clIntVal{\clAlt{2}}{1}{1} \subseteq K & \subseteq \clEiii \cup \clIntVal{\clAlt{\leq 2}}{1}{1}, &
0 \in K &, \\
\clMc \cup \clIntVal{\clAlt{2}}{1}{1} \subseteq K & \subseteq \clEioi \cup \clM, &
\clMcneg \cup \clIntVal{\clAlt{2}}{1}{1} \subseteq K & \subseteq \clEiio \cup \clMneg
\end{align*}
holds,
it follows from Lemma~\ref{lem:stab-apu}\ref{lem:stab-apu:LLambda}, \ref{lem:stab-apu:LMA2}--\ref{lem:stab-apu:L0}
that $\{ \veewedge \} K \nsubseteq K$.

We conclude that $C K \subseteq K$ if and only if ($K$ is not a clonoid with zero and $C \subseteq \clVi$) or ($K$ is a clonoid with zero and $C \subseteq \clV$).

We now focus on stability under right composition with clones.
We obviously have $K \, \clMc \subseteq K$ because $K$ is a $(\clMc,\clVc)$\hyp{}clonoid.

Because $\clIntVal{\clAlt{\leq 2}}{1}{1} \subseteq K \subseteq \clEioo \cup \clVako$,
it follows from Lemma~\ref{lem:stab-apu-R}\ref{lem:stab-apu:Rneg} that $K \, \clonegen{\neg} \nsubseteq K$.

It is easy to verify that if $K = \clEioo \cup \clVako$, then $K \, \clonegen{\mathord{+_3}} \subseteq K$.
(For, note that $\clonegen{\mathord{+_3}} \subseteq \clOI$, and therefore, for any $a, b \in \{0,1\}$ and for $f \in \clIntVal{\clAll}{a}{b}$ and $g_1, \dots, g_n \in \clOI$, we have $f(g_1,\dots,g_n)(\vect{0}) = f(\vect{0}) = a$ and $f(g_1,\dots,g_n)(\vect{1}) = f(\vect{1}) = b$; thus $f(g_1, \dots, g_n) \in \clIntVal{\clAll}{a}{b}$. Moreover, if $f = 0$, then $f(g_1, \dots, g_n) = 0$.)

If $K \neq \clEioo \cup \clVako$, then one of the conditions
\[
\clIntVal{\clAlt{2}}{1}{1} \subseteq K \subseteq \clEiii \cup \clIntVal{\clAlt{\leq 2}}{1}{1},
\qquad
\{\id\} \subseteq K \subseteq \clEioi \cup \clMc,
\qquad
\{\neg\} \subseteq K \subseteq \clEiio \cup \clMcneg
\]
holds, and it follows from Lemma~\ref{lem:stab-apu-R}\ref{lem:stab-apu:Rplus1}--\ref{lem:stab-apu:Rplus3} that $K \, \clonegen{\mathord{+_3}} \nsubseteq K$.

Let us consider stability under right composition with $\clonegen{0}$.
Note that $\clonegen{0} = \clIo \subseteq \clMo$.
Therefore,
$\clMo \, \clonegen{0} \subseteq \clMo$,
$\clMoneg \, \clonegen{0} = \overline{\clMo} \, \clonegen{0} = \overline{\clMo \clonegen{0}} \subseteq \overline{\clMo} = \clMoneg$.
$\clVak \, \clonegen{0} = \clVak$,
and, moreover, by Lemma~\ref{lem:AkM},
$\clIntVal{\clAlt{2}}{1}{1} \, \clonegen{0} \subseteq \clIntVal{\clAlt{\leq 2}}{1}{} = \clIntVal{\clAlt{2}}{1}{1} \cup \clMoneg$.
Using Lemma~\ref{lem:CompUnion}, we get that
if $K$ is one of
$\clM \cup \clMneg \cup \clIntVal{\clAlt{2}}{1}{1}$,
$\clMo \cup \clIX$,
$\clMneg \cup \clIntVal{\clAlt{2}}{1}{1}$,
$\clMoneg \cup \clIntVal{\clAlt{2}}{1}{1}$,
then $K \, \clonegen{0} \subseteq K$.

If $K$ is not one of
$\clM \cup \clMneg \cup \clIntVal{\clAlt{2}}{1}{1}$,
$\clMo \cup \clIX$,
$\clMneg \cup \clIntVal{\clAlt{2}}{1}{1}$,
$\clMoneg \cup \clIntVal{\clAlt{2}}{1}{1}$,
then one of the conditions
\begin{align*}
\{\id\} \subseteq K & \subseteq \clEioo, &
\clIntVal{\clAlt{2}}{1}{1} \subseteq K & \subseteq \clEiio, \\
\clIntVal{\clAlt{4}}{1}{1} \subseteq K & \subseteq \clEiio \cup \clMcneg, &
\clIntVal{\clAlt{3}}{0}{1} \subseteq K & \subseteq \clEioo \cup \clVako
\end{align*}
holds, and it follows from Lemma~\ref{lem:stab-apu-R}\ref{lem:stab-apu:R0a}--\ref{lem:stab-apu:R0d} that $K \, \clonegen{0} \nsubseteq K$.

Let us consider stability under right composition with $\clonegen{1}$.
Note that $\clonegen{1} = \clIi \subseteq \clMi$.
Therefore,
$\clMi \, \clonegen{1} \subseteq \clMi$,
$\clMineg \, \clonegen{1} = \overline{\clMi} \, \clonegen{1} = \overline{\clMi \clonegen{1}} \subseteq \overline{\clMi} = \clMineg$.
$\clVak \, \clonegen{1} = \clVak$,
and, moreover, by Lemma~\ref{lem:AkM},
$\clIntVal{\clAlt{2}}{1}{1} \, \clonegen{1} \subseteq \clIntVal{\clAlt{\leq 2}}{}{1} = \clIntVal{\clAlt{2}}{1}{1} \cup \clMi$.
Using Lemma~\ref{lem:CompUnion}, we get that
if $K$ is one of
$\clM \cup \clMneg \cup \clIntVal{\clAlt{2}}{1}{1}$,
$\clMineg \cup \clXI$,
$\clM \cup \clIntVal{\clAlt{2}}{1}{1}$,
$\clMi \cup \clIntVal{\clAlt{2}}{1}{1}$,
then $K \, \clonegen{1} \subseteq K$.

If $K$ is not one of
$\clM \cup \clMneg \cup \clIntVal{\clAlt{2}}{1}{1}$,
$\clMineg \cup \clXI$,
$\clM \cup \clIntVal{\clAlt{2}}{1}{1}$,
$\clMi \cup \clIntVal{\clAlt{2}}{1}{1}$,
then one of the conditions
\begin{align*}
\{\neg\} \subseteq K & \subseteq \clEiii, &
\clIntVal{\clAlt{2}}{1}{1} \subseteq K & \subseteq \clEioi, \\
\clIntVal{\clAlt{4}}{1}{1} \subseteq K & \subseteq \clEioi \cup \clMc, &
\clIntVal{\clAlt{3}}{1}{0} \subseteq K & \subseteq \clEioo \cup \clVako
\end{align*}
holds, and it follows from Lemma~\ref{lem:stab-apu-R}\ref{lem:stab-apu:R1a}--\ref{lem:stab-apu:R1d} that $K \, \clonegen{1} \nsubseteq K$.

Putting these facts together, we obtain that
\begin{itemize}
\item if $K = \clEioo \cup \clVako$, then $K C \subseteq K$ if and only if $C \subseteq \clonegen{\clMc \cup \{\mathord{+_3}\}} = \clOI$,
\item if $K = \clM \cup \clMneg \cup \clIntVal{\clAlt{2}}{1}{1}$, then $K C \subseteq K$ if and only if $C \subseteq \clonegen{\clMc \cup \{0, 1\}} = \clM$,
\item if $K \in \{\clMo \cup \clIX, \clMneg \cup \clIntVal{\clAlt{2}}{1}{1}, \clMoneg \cup \clIntVal{\clAlt{2}}{1}{1}\}$, then $K C \subseteq K$ if and only if $C \subseteq \clonegen{\clMc \cup \{0\}} = \clMo$,
\item if $K \in \{\clMineg \cup \clXI, \clM \cup \clIntVal{\clAlt{2}}{1}{1}, \clMi \cup \clIntVal{\clAlt{2}}{1}{1}\}$, then $K C \subseteq K$ if and only if $C \subseteq \clonegen{\clMc \cup \{1\}} = \clMi$,
\item otherwise, $K C \subseteq K$ if and only if $C \subseteq \clMc$.
\end{itemize}
This is presented explicitly in Table~\ref{table:McVc-stability}.
\end{proof}


\section{Clonoids with a discriminator source clone}
\label{sec:discriminator}

It was shown in \cite[Corollary~4.2]{LehSze-discriminator} that if $C$ is a \emph{discriminator clone} on a finite set $A$ i.e., a clone containing the \emph{discriminator function}
\[
t(x,y,z) =
\begin{cases}
z, & \text{if $x = y$,} \\
x, & \text{if $x \neq y$,}
\end{cases}
\]
then there are only a finite number of $C$\hyp{}equivalence classes in $\mathcal{O}_A$.
In the case of Boolean functions, the discriminator clones are precisely the superclones of the clone $\clSc$ of idempotent self\hyp{}dual functions.
The $C$\hyp{}equivalence classes and the $C$\hyp{}minor poset were explicitly described for each discriminator clone $C$ on $\{0,1\}$ in \cite[Section~5]{LehSze-discriminator}.
We are going to rephrase this result below in Theorem~\ref{thm:disc-minors}.
For this, the following notation will be used.
For $f \colon \{0,1\}^n \to \{0,1\}$, let $\range^{[2]}(f) := \{ \, \{f(\vect{a}), f(\overline{\vect{a}}) \} \mid \vect{a} \in \{0,1\}^n \, \}$, and
for a nonempty subset $R$ of $\{\{0\}, \{1\}, \{0,1\}\}$, let
\[
F^R := \{ \, f \in \clAll \mid \range^{[2]}(f) = R \, \}.
\]
In order to simplify this notation, we are going to omit set brackets when we list elements of $R$; for example, we are going to write $F^{0,1,01}_{01}$ for $F^{\{\{0\},\{1\},\{0,1\}\}}_{01}$.

\begin{theorem}[{\cite[Section~5]{LehSze-discriminator}}]
\label{thm:disc-minors}
\leavevmode
\begin{enumerate}[label=\upshape{(\roman*)}]
\item
There are exactly 16 $\equiv_{\clSc}$\hyp{}equivalence classes; they are the sets of the form $F^R_{ab}$, where $\emptyset \neq R \subseteq \{\{0\}, \{1\}, \{0,1\}\}$ and $(a,b) \in \{0,1\}^2$ with $\{a,b\} \in R$.
The $\clSc$\hyp{}minor poset is shown in Figure~\ref{fig:Sc-minor}.

\item
There are exactly 7 $\equiv_{\clS}$\hyp{}equivalence classes; they are the sets of the form $F^R$, where $\emptyset \neq R \subseteq \{\{0\}, \{1\}, \{0,1\}\}$.
The $\clS$\hyp{}minor poset is shown in Figure~\ref{fig:S-minor}.

\item
There are exactly 6 $\equiv_{\clOI}$\hyp{}equivalence classes; they are $\clVako$, $\clVaki$, $\clOO \setminus \clVako$, $\clII \setminus \clVaki$, $\clOI$, $\clIO$.
The $\clOI$\hyp{}minor poset is shown in Figure~\ref{fig:Tc-minor}.

\item
There are exactly 4 $\equiv_{\clOX}$\hyp{}equivalence classes; they are $\clVako$, $\clVaki$, $\clOX \setminus \clVako$, $\clIX \setminus \clVaki$.
The $\clOX$\hyp{}minor poset is shown in Figure~\ref{fig:T0-minor}.

\item
There are exactly 4 $\equiv_{\clXI}$\hyp{}equivalence classes; they are $\clVako$, $\clVaki$, $\clXO \setminus \clVako$, $\clXI \setminus \clVaki$.
The $\clXI$\hyp{}minor poset is shown in Figure~\ref{fig:T1-minor}.

\item
There are exactly 3 $\equiv_{\clAll}$\hyp{}equivalence classes; they are $\clVako$, $\clVaki$, $\clAll \setminus \clVak$.
The $\clAll$\hyp{}minor poset is shown in Figure~\ref{fig:Omega-minor}.
\end{enumerate}
\end{theorem}

\begin{theorem}
\label{thm:disc-clonoids}
\leavevmode
\begin{enumerate}[label=\upshape{(\roman*)}]
\item\label{thm:disc-clonoids:ScIc}
There are exactly 1296 $(\clSc,\clIc)$\hyp{}clonoids; they are the sets of the form $A \cup B \cup C \cup D$, where
\begin{align*}
\displaybump
A & \in \{ \clEmpty, \clVako, \clReflOO, \clSminOO, \clReflOO \cup \clSminOO, \clOO \}, &
B & \in \{ \clEmpty, \clSc, \clSmajOI, \clSminOI, \clSmajOI \cup \clSminOI, \clOI \}, \\
C & \in \{ \clEmpty, \clIntVal{\clS}{1}{0}, \clSmajIO, \clSminIO, \clSmajIO \cup \clSminIO, \clIO \}, &
D & \in \{ \clEmpty, \clVaki, \clSmajII, \clReflII, \clSmajII \cup \clReflII, \clII \}.
\end{align*}
The lattice of $(\clSc,\clIc)$\hyp{}clonoids is isomorphic to the direct product of four 6\hyp{}element lattices as shown in Figure~\ref{fig:ScIc-clonoids}.

\item
There are exactly 19 $(\clS,\clIc)$\hyp{}clonoids; they are
\begin{align*}
\displaybump
& \clAll, &
& \clRefl \cup \clSmin \cup \clSmaj, &
& \clRefl \cup \clSmin, &
& \clRefl \cup \clSmaj, &
& \clSmin \cup \clSmaj, \\
& \clSmaj \cup \clVako, &
& \clSmin \cup \clVaki, &
& \clRefl \cup \clS, &
& \clSmin, &
& \clSmaj, &
& \clRefl, \\
& \clS \cup \clVak, &
& \clS \cup \clVako, &
& \clS \cup \clVaki, &
& \clS, &
& \clVak, &
& \clVako, &
& \clVaki, &
& \clEmpty.
\end{align*}
The lattice of $(\clS,\clIc)$\hyp{}clonoids is shown in Figure~\ref{fig:SIc-clonoids}.

\item
There are exactly 36 $(\clOI,\clIc)$\hyp{}clonoids; they are the sets of the form $A \cup B \cup C \cup D$, where
\begin{align*}
\displaybump
A & \in \{ \clEmpty, \clVako, \clOO \}, &
B & \in \{ \clEmpty, \clOI \}, &
C & \in \{ \clEmpty, \clIO \}, &
D & \in \{ \clEmpty, \clVaki, \clII \}.
\end{align*}
The lattice of $(\clOI,\clIc)$\hyp{}clonoids is shown in Figure~\ref{fig:TcIc-clonoids}.

\item
There are exactly 9 $(\clOX,\clIc)$\hyp{}clonoids; they are the sets of the form $A \cup B$, where
$A \in \{\clEmpty, \clVako, \clOX\}$ and $B \in \{\clEmpty, \clVaki, \clIX\}$.
The lattice of $(\clOX,\clIc)$\hyp{}clonoids is shown in Figure~\ref{fig:T0Ic-clonoids}.

\item
There are exactly 9 $(\clXI,\clIc)$\hyp{}clonoids; they are the sets of the form $A \cup B$, where
$A \in \{\clEmpty, \clVako, \clXO\}$ and $B \in \{\clEmpty, \clVaki, \clXI\}$.
The lattice of $(\clXI,\clIc)$\hyp{}clonoids is shown in Figure~\ref{fig:T1Ic-clonoids}.

\item
There are exactly 5 $(\clAll,\clIc)$\hyp{}clonoids; they are
$\clAll$, $\clVak$, $\clVako$, $\clVaki$, $\clEmpty$.
The lattice of $(\clAll,\clIc)$\hyp{}clonoids is shown in Figure~\ref{fig:OmegaIc-clonoids}.
\end{enumerate}
\end{theorem}

\begin{figure}
\begin{center}
\scalebox{0.9}{%
\begin{tikzpicture}[baseline, scale=0.8]
   \node [label = below:$F^0_{00}$] (F0-00) at (0,0) {};
   \node [label = above left:$F^{0,1}_{00}$] (F01-00) at (-1,1) {};
   \node [label = above right:$F^{0,01}_{00}$] (F001-00) at (1,1) {};
   \node [label = {[yshift=-5pt]$F^{0,1,01}_{00}$}] (F0101-00) at (0,2) {};
   \draw (F0-00) -- (F01-00) -- (F0101-00);
   \draw (F0-00) -- (F001-00) -- (F0101-00);
   \node [label = below:$F^{01}_{01}$] (F01-01) at (4,0) {};
   \node [label = below left:$F^{1,01}_{01}$] (F101-01) at (3,1) {};
   \node [label = below right:$F^{0,01}_{01}$] (F001-01) at (5,1) {};
   \node [label = {[yshift=-5pt]$F^{0,1,01}_{01}$}] (F0101-01) at (4,2) {};
   \draw (F01-01) -- (F101-01) -- (F0101-01);
   \draw (F01-01) -- (F001-01) -- (F0101-01);
   \node [label = below:$F^{01}_{10}$] (F01-10) at (8,0) {};
   \node [label = above left:$F^{1,01}_{10}$] (F101-10) at (7,1) {};
   \node [label = above right:$F^{0,01}_{10}$] (F001-10) at (9,1) {};
   \node [label = {[yshift=-5pt]$F^{0,1,01}_{10}$}] (F0101-10) at (8,2) {};
   \draw (F01-10) -- (F101-10) -- (F0101-10);
   \draw (F01-10) -- (F001-10) -- (F0101-10);
   \node [label = below:$F^1_{11}$] (F1-11) at (12,0) {};
   \node [label = below left:$F^{1,01}_{11}$] (F101-11) at (11,1) {};
   \node [label = below right:$F^{0,1}_{11}$] (F01-11) at (13,1) {};
   \node [label = {[yshift=-5pt]$F^{0,1,01}_{11}$}] (F0101-11) at (12,2) {};
   \draw (F1-11) -- (F101-11) -- (F0101-11);
   \draw (F1-11) -- (F01-11) -- (F0101-11);
\end{tikzpicture}
}
\end{center}
\caption{$\protect\clSc$\hyp{}minor poset.}
\label{fig:Sc-minor}
\end{figure}

\begin{figure}
\begin{center}
\scalebox{0.83}{%
\begin{tikzpicture}[baseline, scale=0.8]
   \node [label = below:$\clEmpty$] (00-empty) at (0,-1) {};
   \node [label = below left:$\clVako$] (00-C0) at (0,0) {};
   \node [label = left:$\clReflOO$] (00-R00) at (-1,1) {};
   \node [label = right:$\clSminOO$] (00-Smin00) at (1,1) {};
   \node [label = {[xshift=-25pt,yshift=-15pt]$\clReflOO \cup \clSminOO$}] (00-RS) at (0,2) {};
   \node [label = above:$\clOO$] (00-00) at (0,3) {};
   \draw (00-empty) -- (00-C0) -- (00-R00) -- (00-RS) -- (00-00);
   \draw (00-C0) -- (00-Smin00) -- (00-RS);
   \node [label = below:$\clEmpty$] (01-empty) at (5,-1) {};
   \node [label = below left:$\clSc$] (01-S01) at (5,0) {};
   \node [label = left:$\clSmajOI$] (01-Smaj01) at (4,1) {};
   \node [label = right:$\clSminOI$] (01-Smin01) at (6,1) {};
   \node [label = {[xshift=-25pt,yshift=-15pt]$\clSmajOI \cup \clSminOI$}] (01-SmajSmin) at (5,2) {};
   \node [label = above:$\clOI$] (01-01) at (5,3) {};
   \draw (01-empty) -- (01-S01) -- (01-Smaj01) -- (01-SmajSmin) -- (01-01);
   \draw (01-S01) -- (01-Smin01) -- (01-SmajSmin);
   \node [label = below:$\clEmpty$] (10-empty) at (10,-1) {};
   \node [label = below left:$\clIntVal{\clS}{1}{0}$] (10-S10) at (10,0) {};
   \node [label = left:$\clSmajIO$] (10-Smaj10) at (9,1) {};
   \node [label = right:$\clSminIO$] (10-Smin10) at (11,1) {};
   \node [label = {[xshift=-25pt,yshift=-15pt]$\clSmajIO \cup \clSminIO$}] (10-SmajSmin) at (10,2) {};
   \node [label = above:$\clIO$] (10-10) at (10,3) {};
   \draw (10-empty) -- (10-S10) -- (10-Smaj10) -- (10-SmajSmin) -- (10-10);
   \draw (10-S10) -- (10-Smin10) -- (10-SmajSmin);
   \node [label = below:$\clEmpty$] (11-empty) at (15,-1) {};
   \node [label = below left:$\clVaki$] (11-C1) at (15,0) {};
   \node [label = left:$\clSmajII$] (11-Smaj11) at (14,1) {};
   \node [label = right:$\clReflII$] (11-R11) at (16,1) {};
   \node [label = {[xshift=-25pt,yshift=-15pt]$\clReflII \cup \clSmajII$}] (11-RS) at (15,2) {};
   \node [label = above:$\clII$] (11-11) at (15,3) {};
   \draw (11-empty) -- (11-C1) -- (11-R11) -- (11-RS) -- (11-11);
   \draw (11-C1) -- (11-Smaj11) -- (11-RS);
   \node [fill = none, draw = none] at (2.5,1) {$\times$};
   \node [fill = none, draw = none] at (7.5,1) {$\times$};
   \node [fill = none, draw = none] at (12.5,1) {$\times$};
\end{tikzpicture}
}
\end{center}
\caption{$(\protect\clSc,\protect\clIc)$\hyp{}clonoids. The lattice $\protect\closys{(\protect\clSc,\protect\clIc)}$ is isomorphic to the direct product of four 6-element lattices. For the element $(A,B,C,D)$ of the direct product, the corresponding clonoid is $A \cup B \cup C \cup D$.}
\label{fig:ScIc-clonoids}
\end{figure}

\begin{figure}
\noindent
\begin{minipage}[c]{0.5\textwidth}
\begin{center}
\scalebox{0.9}{%
\begin{tikzpicture}[baseline, scale=0.8]
   \node [label = below:$F^0$] (F0) at (0,0) {};
   \node [label = below:$F^{01}$] (F01) at (2,0) {};
   \node [label = below:$F^{1}$] (F1) at (4,0) {};
   \node [label = left:$F^{0,01}$] (F0-01) at (0,2) {};
   \node [label = left:$F^{0,1}$] (F0-1) at (2,2) {};
   \node [label = right:$F^{1,01}$] (F1-01) at (4,2) {};
   \node [label = {[yshift=-5pt]$F^{0,1,01}$}] (F0-1-01) at (2,4) {};
   \draw (F0) -- (F0-01) -- (F0-1-01);
   \draw (F0) -- (F0-1) -- (F0-1-01);
   \draw (F01) -- (F0-01);
   \draw (F01) -- (F1-01) -- (F0-1-01);
   \draw (F1) -- (F0-1);
   \draw (F1) -- (F1-01);
\end{tikzpicture}
}
\end{center}
\end{minipage}%
\begin{minipage}[c]{0.5\textwidth}
\begin{center}
\scalebox{0.9}{%
\tikzstyle{every node}=[circle, draw, fill=black, inner sep=0pt, minimum size=6pt, scale=1, font=\scriptsize]
\begin{tikzpicture}[baseline, scale=0.6]
   \node [label = below:$\clEmpty$] (empty) at (0,0) {};
   \node [label = left:$\clVako$] (C0) at (-2,2) {};
   \node [label = left:$\clS$] (S) at (0,2) {};
   \node [label = right:$\clVaki$] (C1) at (2,2) {};
   \node [label = left:$\clS \cup \clVako$] (SC0) at (-2,4) {};
   \node [label = left:$\clVak$] (C) at (0,4) {};
   \node [label = right:$\clS \cup \clVaki$] (SC1) at (2,4) {};
   \node [label = left:$\clSmin$] (Smin) at (-4,6) {};
   \node [label = left:$\clRefl$] (R) at (-2,6) {};
   \node [label = right:$\clS \cup \clVak$] (SC) at (0,6) {};
   \node [label = right:$\clSmaj$] (Smaj) at (4,6) {};
   \node [label = left:$\clSmin \cup \clVaki$] (SminC1) at (-2,8) {};
   \node [label = right:$\clRefl \cup \clS$] (RS) at (0,8) {};
   \node [label = right:$\clSmaj \cup \clVako$] (SmajC0) at (2,8) {};
   \node [label = left:$\clRefl \cup \clSmin$] (RSmin) at (-2,10) {};
   \node [label = {[xshift=14pt,yshift=-12pt]\begin{tabular}{@{}l@{}}$\clSmin$ \\ $\mathop{\cup} \clSmaj$ \end{tabular}}] (SminSmaj) at (0,10) {};
   \node [label = right:$\clRefl \cup \clSmaj$] (RSmaj) at (2,10) {};
   \node [label = left:$\clRefl \cup \clSmin \cup \clSmaj$] (RSminSmaj) at (0,12) {};
   \node [label = above:$\clAll$] (all) at (0,14) {};
   \draw (empty) -- (C0);
   \draw (empty) -- (S);
   \draw (empty) -- (C1);
   \draw (C0) -- (SC0);
   \draw (C0) -- (C);
   \draw (C1) -- (SC1);
   \draw (C1) -- (C);
   \draw (S) -- (SC0);
   \draw (S) -- (SC1);
   \draw (SC0) -- (Smin);
   \draw (SC0) -- (SC);
   \draw (SC1) -- (Smaj);
   \draw (SC1) -- (SC);
   \draw (C) -- (R);
   \draw (C) -- (SC);
   \draw (Smin) -- (SminC1);
   \draw (Smaj) -- (SmajC0);
   \draw (R) -- (RS);
   \draw (SC) -- (SminC1);
   \draw (SC) -- (SmajC0);
   \draw (SC) -- (RS);
   \draw (SminC1) -- (RSmin);
   \draw (SminC1) -- (SminSmaj);
   \draw (SmajC0) -- (RSmaj);
   \draw (SmajC0) -- (SminSmaj);
   \draw (RS) -- (RSmin);
   \draw (RS) -- (RSmaj);
   \draw (RSmin) -- (RSminSmaj);
   \draw (RSmaj) -- (RSminSmaj);
   \draw (SminSmaj) -- (RSminSmaj);
   \draw (RSminSmaj) -- (all);
\end{tikzpicture}
}
\end{center}
\end{minipage}
\begin{minipage}[t]{0.5\textwidth}
\captionsetup{width=.9\textwidth}
\captionof{figure}{$\protect\clS$\hyp{}minor poset.}
\label{fig:S-minor}
\end{minipage}%
\begin{minipage}[t]{0.5\textwidth}
\captionsetup{width=.9\textwidth}
\captionof{figure}{$(\protect\clS,\protect\clIc)$\hyp{}clonoids.}
\label{fig:SIc-clonoids}
\end{minipage}
\end{figure}

\begin{figure}
\begin{center}
\scalebox{0.9}{%
\begin{tikzpicture}[baseline, scale=0.8]
   \node [label = below:$\clVako$] (C0) at (0,0) {};
   \node [label = {[yshift=-10pt]$\clOO \setminus \clVako$}] (00) at (0,2) {};
   \node [label = above:$\clOI$] (01) at (2,2) {};
   \node [label = above:$\clIO$] (10) at (4,2) {};
   \node [label = {[yshift=-10pt]$\clII \setminus \clVaki$}] (11) at (6,2) {};
   \node [label = below:$\clVaki$] (C1) at (6,0) {};
   \draw (C0) -- (00);
   \draw (C1) -- (11);
\end{tikzpicture}
}
\end{center}
\caption{$\protect\clOI$\hyp{}minor poset.}
\label{fig:Tc-minor}

\bigskip\bigskip
\begin{center}
\scalebox{0.9}{%
\begin{tikzpicture}[baseline, scale=0.8]
   \tikzmath{\x = 3; \y = 2; \p = 0.75;} 
   \node [label = {[shift={(0,-0.7)}]$\clEmpty$}] (0000) at (0,0) {};
   \node [label = {[shift={(-0.1,0)}]$\clOI$}] (0100) at ($(0000)+(-\p,1)$) {};
   \node [label = {[shift={(0.1,0)}]$\clIO$}] (0010) at ($(0000)+(\p,1)$) {};
   \node (0110) at ($(0000)+(0,2)$) {};
   \node [label = below left:$\clVako$] (1000) at ($(0000)+(-\x,\y)$) {};
   \node (1100) at ($(1000)+(-\p,1)$) {};
   \node (1010) at ($(1000)+(\p,1)$) {};
   \node (1110) at ($(1000)+(0,2)$) {};
   \node [label = below left:$\clOO$] (2000) at ($(1000)+(-\x,\y)$) {};
   \node (2100) at ($(2000)+(-\p,1)$) {};
   \node (2010) at ($(2000)+(\p,1)$) {};
   \node (2110) at ($(2000)+(0,2)$) {};
   \node [label = below right:$\clVaki$] (0001) at ($(0000)+(\x,\y)$) {};
   \node (0101) at ($(0001)+(-\p,1)$) {};
   \node (0011) at ($(0001)+(\p,1)$) {};
   \node (0111) at ($(0001)+(0,2)$) {};
   \node (1001) at ($(0001)+(-\x,\y)$) {};
   \node (1101) at ($(1001)+(-\p,1)$) {};
   \node (1011) at ($(1001)+(\p,1)$) {};
   \node (1111) at ($(1001)+(0,2)$) {};
   \node (2001) at ($(1001)+(-\x,\y)$) {};
   \node (2101) at ($(2001)+(-\p,1)$) {};
   \node (2011) at ($(2001)+(\p,1)$) {};
   \node (2111) at ($(2001)+(0,2)$) {};
   \node [label = below right:$\clII$] (0002) at ($(0001)+(\x,\y)$) {};
   \node (0102) at ($(0002)+(-\p,1)$) {};
   \node (0012) at ($(0002)+(\p,1)$) {};
   \node (0112) at ($(0002)+(0,2)$) {};
   \node (1002) at ($(0002)+(-\x,\y)$) {};
   \node (1102) at ($(1002)+(-\p,1)$) {};
   \node (1012) at ($(1002)+(\p,1)$) {};
   \node (1112) at ($(1002)+(0,2)$) {};
   \node (2002) at ($(1002)+(-\x,\y)$) {};
   \node (2102) at ($(2002)+(-\p,1)$) {};
   \node (2012) at ($(2002)+(\p,1)$) {};
   \node (2112) at ($(2002)+(0,2)$) {};

   \draw (0000) -- (0100) -- (0110) -- (0010) -- (0000);
   \draw (1000) -- (1100) -- (1110) -- (1010) -- (1000);
   \draw (2000) -- (2100) -- (2110) -- (2010) -- (2000);
   \draw (0001) -- (0101) -- (0111) -- (0011) -- (0001);
   \draw (1001) -- (1101) -- (1111) -- (1011) -- (1001);
   \draw (2001) -- (2101) -- (2111) -- (2011) -- (2001);
   \draw (0002) -- (0102) -- (0112) -- (0012) -- (0002);
   \draw (1002) -- (1102) -- (1112) -- (1012) -- (1002);
   \draw (2002) -- (2102) -- (2112) -- (2012) -- (2002);

   \draw (0000) -- (1000) -- (2000);
   \draw (0100) -- (1100) -- (2100);
   \draw (0010) -- (1010) -- (2010);
   \draw (0110) -- (1110) -- (2110);
   \draw (0001) -- (1001) -- (2001);
   \draw (0101) -- (1101) -- (2101);
   \draw (0011) -- (1011) -- (2011);
   \draw (0111) -- (1111) -- (2111);
   \draw (0002) -- (1002) -- (2002);
   \draw (0102) -- (1102) -- (2102);
   \draw (0012) -- (1012) -- (2012);
   \draw (0112) -- (1112) -- (2112);

   \draw (0000) -- (0001) -- (0002);
   \draw (0100) -- (0101) -- (0102);
   \draw (0010) -- (0011) -- (0012);
   \draw (0110) -- (0111) -- (0112);
   \draw (1000) -- (1001) -- (1002);
   \draw (1100) -- (1101) -- (1102);
   \draw (1010) -- (1011) -- (1012);
   \draw (1110) -- (1111) -- (1112);
   \draw (2000) -- (2001) -- (2002);
   \draw (2100) -- (2101) -- (2102);
   \draw (2010) -- (2011) -- (2012);
   \draw (2110) -- (2111) -- (2112);
\end{tikzpicture}
}
\end{center}
\caption{$(\protect\clOI,\protect\clIc)$\hyp{}clonoids.
Only the join\hyp{}irreducible elements are named in the diagram. The remaining clonoids are unions of the join\hyp{}irreducible ones.}
\label{fig:TcIc-clonoids}
\end{figure}

\begin{figure}
\begin{minipage}[c]{0.5\textwidth}
\begin{center}
\scalebox{0.9}{%
\begin{tikzpicture}[baseline, scale=0.8]
   \node [label = below:$\clVako$] (C0) at (0,0) {};
   \node [label = {[yshift=-10pt]$\clOX \setminus \clVako$}] (0X) at (0,2) {};
   \node [label = {[yshift=-10pt]$\clIX \setminus \clVaki$}] (1X) at (2,2) {};
   \node [label = below:$\clVaki$] (C1) at (2,0) {};
   \draw (C0) -- (0X);
   \draw (C1) -- (1X);
\end{tikzpicture}
}
\end{center}
\end{minipage}%
\begin{minipage}[c]{0.5\textwidth}
\begin{center}
\scalebox{0.9}{%
\begin{tikzpicture}[baseline, scale=0.8]
   \node [label = below:$\clEmpty$] (00) at (0,0) {};
   \node [label = below left:$\clVako$] (10) at (-1,1) {};
   \node [label = below right:$\clVaki$] (01) at (1,1) {};
   \node [label = left:$\clOX$] (20) at (-2,2) {};
   \node [label = below:$\clVak$] (11) at (0,2) {};
   \node [label = right:$\clIX$] (02) at (2,2) {};
   \node [label = {[xshift=-22pt,yshift=-12pt]$\clOX \cup \clVaki$}] (21) at (-1,3) {};
   \node [label = {[xshift=22pt,yshift=-12pt]$\clIX \cup \clVako$}] (12) at (1,3) {};
   \node [label = above:$\clAll$] (22) at (0,4) {};
   \draw (00) -- (10) -- (20);
   \draw (01) -- (11) -- (21);
   \draw (02) -- (12) -- (22);
   \draw (00) -- (01) -- (02);
   \draw (10) -- (11) -- (12);
   \draw (20) -- (21) -- (22);
\end{tikzpicture}
}
\end{center}
\end{minipage}
\begin{minipage}[t]{0.5\textwidth}
\captionsetup{width=.9\textwidth}
\captionof{figure}{$\protect\clOX$\hyp{}minor poset.}
\label{fig:T0-minor}
\end{minipage}%
\begin{minipage}[t]{0.5\textwidth}
\captionsetup{width=.9\textwidth}
\captionof{figure}{$(\protect\clOX,\protect\clIc)$\hyp{}clonoids.}
\label{fig:T0Ic-clonoids}
\end{minipage}
\end{figure}

\begin{figure}
\begin{minipage}[c]{0.5\textwidth}
\begin{center}
\scalebox{0.9}{%
\begin{tikzpicture}[baseline, scale=0.8]
   \node [label = below:$\clVako$] (C0) at (0,0) {};
   \node [label = {[yshift=-10pt]$\clXO \setminus \clVako$}] (X0) at (0,2) {};
   \node [label = {[yshift=-10pt]$\clXI \setminus \clVaki$}] (X1) at (2,2) {};
   \node [label = below:$\clVaki$] (C1) at (2,0) {};
   \draw (C0) -- (X0);
   \draw (C1) -- (X1);
\end{tikzpicture}
}
\end{center}
\end{minipage}%
\begin{minipage}[c]{0.5\textwidth}
\begin{center}
\scalebox{0.9}{%
\begin{tikzpicture}[baseline, scale=0.8]
   \node [label = below:$\clEmpty$] (00) at (0,0) {};
   \node [label = below left:$\clVako$] (10) at (-1,1) {};
   \node [label = below right:$\clVaki$] (01) at (1,1) {};
   \node [label = left:$\clXO$] (20) at (-2,2) {};
   \node [label = below:$\clVak$] (11) at (0,2) {};
   \node [label = right:$\clXI$] (02) at (2,2) {};
   \node [label = {[xshift=-22pt,yshift=-12pt]$\clXO \cup \clVaki$}] (21) at (-1,3) {};
   \node [label = {[xshift=22pt,yshift=-12pt]$\clXI \cup \clVako$}] (12) at (1,3) {};
   \node [label = above:$\clAll$] (22) at (0,4) {};
   \draw (00) -- (10) -- (20);
   \draw (01) -- (11) -- (21);
   \draw (02) -- (12) -- (22);
   \draw (00) -- (01) -- (02);
   \draw (10) -- (11) -- (12);
   \draw (20) -- (21) -- (22);
\end{tikzpicture}
}
\end{center}
\end{minipage}
\begin{minipage}[t]{0.5\textwidth}
\captionsetup{width=.9\textwidth}
\captionof{figure}{$\protect\clXI$\hyp{}minor poset.}
\label{fig:T1-minor}
\end{minipage}%
\begin{minipage}[t]{0.5\textwidth}
\captionsetup{width=.9\textwidth}
\captionof{figure}{$(\protect\clXI,\protect\clIc)$\hyp{}clonoids.}
\label{fig:T1Ic-clonoids}
\end{minipage}
\end{figure}

\begin{figure}
\begin{minipage}[c]{0.5\textwidth}
\begin{center}
\scalebox{0.9}{%
\begin{tikzpicture}[baseline, scale=0.8]
   \node [label = below:$\clVako$] (C0) at (0,0) {};
   \node [label = {[yshift=-4pt]$\clAll \setminus \clVak$}] (All) at (1,2) {};
   \node [label = below:$\clVaki$] (C1) at (2,0) {};
   \draw (C0) -- (All);
   \draw (C1) -- (All);
\end{tikzpicture}
}
\end{center}
\end{minipage}%
\begin{minipage}[c]{0.5\textwidth}
\begin{center}
\scalebox{0.9}{%
\begin{tikzpicture}[baseline, scale=0.8]
   \node [label = below:$\clEmpty$] (E) at (0,0) {};
   \node [label = left:$\clVako$] (0) at (-1,1) {};
   \node [label = right:$\clVaki$] (1) at (1,1) {};
   \node [label = below:$\clVak$] (01) at (0,2) {};
   \node [label = above:$\clAll$] (All) at (0,3) {};
   \draw (E) -- (0) -- (01) -- (All);
   \draw (E) -- (1) -- (01);
\end{tikzpicture}
}
\end{center}
\end{minipage}
\begin{minipage}[t]{0.5\textwidth}
\captionsetup{width=.9\textwidth}
\captionof{figure}{$\protect\clAll$\hyp{}minor poset.}
\label{fig:Omega-minor}
\end{minipage}%
\begin{minipage}[t]{0.5\textwidth}
\captionsetup{width=.9\textwidth}
\captionof{figure}{$(\protect\clAll,\protect\clIc)$\hyp{}clonoids.}
\label{fig:OmegaIc-clonoids}
\end{minipage}
\end{figure}

\begin{proof}
This follows from Theorem~\ref{thm:disc-minors} by applying Lemma~\ref{lem:clon-C-min}.
The $(C,\clIc)$\hyp{}clonoids are precisely the downsets of the $C$\hyp{}minor quasi\hyp{}order, and the latter are easy to determine from the Hasse diagram of the $C$\hyp{}minor poset.
\end{proof}

\begin{remark}
Because the $\leq_{\clSc}$\hyp{}minor poset is disconnected with four connected components, its lattice of downsets is isomorphic to the direct product of the lattices of downsets of the four connected components (which are isomorphic to each other).
The lattice has $6^4 = 1296$ elements; it is too large for its Hasse diagram being presented explicitly here.
We simply give it as the direct product of the downset lattices of the four connected components; see Figure~\ref{fig:ScIc-clonoids}.
For the element $(A,B,C,D)$ of the direct product, the corresponding $(\clSc,\clIc)$\hyp{}clonoid is $A \cup B \cup C \cup D$.
\end{remark}

With the help of Theorem~\ref{thm:disc-clonoids}, it is now possible to determine the $(C_1,C_2)$\hyp{}clonoids for clones $C_1$ and $C_2$ such that $\clSc \subseteq C_1$ and $C_2$ is arbitrary.
By Lemma~\ref{lem:clonoid-inclusion}, such $(C_1,C_2)$\hyp{}clonoids are $(\clSc,\clIc)$\hyp{}clonoids; it is just a matter of identifying them among the $(\clSc,\clIc)$\hyp{}clonoids that were determined above.
We start with the cases where the target clone is essentially at most unary.

\begin{proposition}
\label{prop:disc-Istar}
\leavevmode
\begin{enumerate}[label=\upshape{(\roman*)}]
\item
\begin{enumerate}[label=\upshape{(\alph*)}]
\item
There are exactly 1081 $(\clSc,\clIo)$\hyp{}clonoids;
they are $\clEmpty$ and those $(\clSc,\clIc)$\hyp{}clonoids $K$ with $\clVako \subseteq K$.
\item
There are exactly 1081 $(\clSc,\clIi)$\hyp{}clonoids;
they are $\clEmpty$ and those $(\clSc,\clIc)$\hyp{}clonoids $K$ with $\clVaki \subseteq K$.
\item
There are exactly 901 $(\clSc,\clI)$\hyp{}clonoids;
they are $\clEmpty$ and those $(\clSc,\clIc)$\hyp{}clonoids $K$ with $\clVak \subseteq K$.
\item
There are exactly 36 $(\clSc,\clIstar)$\hyp{}clonoids;
they are the sets of the form $A \cup B$, where
\begin{align*}
\displaybump
& A \in \{ \clEmpty, \clVak, \clRefl, \clSminOO \cup \clSmajII, \clRefl \cup \clSminOO \cup \clSmajII, \clEq \},
\\
& B \in \{ \clEmpty, \clS, \clSmajOI \cup \clSminIO, \clSminOI \cup \clSmajIO, \clIntNeq{\clSmaj} \cup \clIntNeq{\clSmin}, \clNeq \}.
\end{align*}
\item
There are exactly 31 $(\clSc,\clOmegaOne)$\hyp{}clonoids;
they are $\clEmpty$ and those $(\clSc,\clIstar)$\hyp{}clonoids $K$ with $\clVak \subseteq K$.
\end{enumerate}

\item
\begin{enumerate}[label=\upshape{(\alph*)}]
\item
There are exactly 15 $(\clS,\clIo)$\hyp{}clonoids;
they are
$\clAll$, 
$\clRefl \cup \clSmin \cup \clSmaj$,
$\clRefl \cup \clSmin$,
$\clRefl \cup \clSmaj$,
$\clSmin \cup \clSmaj$,
$\clSmaj \cup \clVako$,
$\clSmin \cup \clVaki$,
$\clRefl \cup \clS$,
$\clSmin$,
$\clRefl$,
$\clS \cup \clVak$,
$\clS \cup \clVako$,
$\clVak$,
$\clVako$,
$\clEmpty$.

\item
There are exactly 15 $(\clS,\clIi)$\hyp{}clonoids;
they are
$\clAll$, 
$\clRefl \cup \clSmin \cup \clSmaj$,
$\clRefl \cup \clSmin$,
$\clRefl \cup \clSmaj$,
$\clSmin \cup \clSmaj$,
$\clSmaj \cup \clVako$,
$\clSmin \cup \clVaki$,
$\clRefl \cup \clS$,
$\clSmaj$,
$\clRefl$,
$\clS \cup \clVak$,
$\clS \cup \clVaki$,
$\clVak$,
$\clVaki$,
$\clEmpty$.

\item
There are exactly 12 $(\clS,\clI)$\hyp{}clonoids;
they are
$\clAll$, 
$\clRefl \cup \clSmin \cup \clSmaj$,
$\clRefl \cup \clSmin$,
$\clRefl \cup \clSmaj$,
$\clSmin \cup \clSmaj$,
$\clSmaj \cup \clVako$,
$\clSmin \cup \clVaki$,
$\clRefl \cup \clS$,
$\clRefl$,
$\clS \cup \clVak$,
$\clVak$,
$\clEmpty$.

\item
There are exactly 9 $(\clS,\clIstar)$\hyp{}clonoids;
they are
$\clAll$, 
$\clRefl \cup \clSmin \cup \clSmaj$,
$\clSmin \cup \clSmaj$,
$\clRefl \cup \clS$,
$\clRefl$,
$\clS \cup \clVak$,
$\clS$,
$\clVak$,
$\clEmpty$.

\item
There are exactly 8 $(\clS,\clOmegaOne)$\hyp{}clonoids;
they are
$\clAll$, 
$\clRefl \cup \clSmin \cup \clSmaj$,
$\clSmin \cup \clSmaj$,
$\clRefl \cup \clS$,
$\clRefl$,
$\clS \cup \clVak$,
$\clVak$,
$\clEmpty$.
\end{enumerate}

\item
\begin{enumerate}[label=\upshape{(\alph*)}]
\item
There are exactly 25 $(\clOI,\clIo)$\hyp{}clonoids;
they are $\clEmpty$ and the $(\clOI,\clIc)$\hyp{}clonoids $K$ with $\clVako \subseteq K$.

\item
There are exactly 25 $(\clOI,\clIi)$\hyp{}clonoids;
they are $\clEmpty$ and the $(\clOI,\clIc)$\hyp{}clonoids $K$ with $\clVaki \subseteq K$.

\item
There are exactly 17 $(\clOI,\clI)$\hyp{}clonoids;
they are $\clEmpty$ and the $(\clOI,\clIc)$\hyp{}clonoids $K$ with $\clVak \subseteq K$.

\item
There are exactly 6 $(\clOI,\clIstar)$\hyp{}clonoids;
they are
$\clAll$, $\clNeq \cup \clVak$, $\clNeq$, $\clEq$, $\clVak$, $\clEmpty$.

\item
There are exactly 5 $(\clOI,\clOmegaOne)$\hyp{}clonoids;
they are
$\clAll$, $\clNeq \cup \clVak$, $\clEq$, $\clVak$, $\clEmpty$.
\end{enumerate}

\item
\begin{enumerate}[label=\upshape{(\alph*)}]
\item
There are exactly 7 $(\clOX,\clIo)$\hyp{}clonoids;
they are
$\clAll$, $\clOXCI$, $\clIXCO$, $\clOX$, $\clVak$, $\clVako$, $\clEmpty$.

\item
There are exactly 7 $(\clOX,\clIi)$\hyp{}clonoids;
they are
$\clAll$, $\clOXCI$, $\clIXCO$, $\clIX$, $\clVak$, $\clVaki$, $\clEmpty$.

\item
There are exactly 5 $(\clOX,\clI)$\hyp{}clonoids;
they are
$\clAll$, $\clOXCI$, $\clIXCO$, $\clVak$, $\clEmpty$.

\item
There are exactly 3 $(\clOX,\clIstar)$\hyp{}clonoids;
they are
$\clAll$, $\clVak$, $\clEmpty$.

\item
There are exactly 3 $(\clOX,\clOmegaOne)$\hyp{}clonoids;
they are
$\clAll$, $\clVak$, $\clEmpty$.
\end{enumerate}

\item
\begin{enumerate}[label=\upshape{(\alph*)}]
\item
There are exactly 7 $(\clXI,\clIo)$\hyp{}clonoids;
they are
$\clAll$, $\clXOCI$, $\clXICO$, $\clXO$, $\clVak$, $\clVako$, $\clEmpty$.

\item
There are exactly 7 $(\clXI,\clIi)$\hyp{}clonoids;
they are
$\clAll$, $\clXOCI$, $\clXICO$, $\clXI$, $\clVak$, $\clVaki$, $\clEmpty$.

\item
There are exactly 5 $(\clXI,\clI)$\hyp{}clonoids;
they are
$\clAll$, $\clXOCI$, $\clXICO$, $\clVak$, $\clEmpty$.

\item
There are exactly 3 $(\clXI,\clIstar)$\hyp{}clonoids;
they are
$\clAll$, $\clVak$, $\clEmpty$.

\item
There are exactly 3 $(\clXI,\clOmegaOne)$\hyp{}clonoids;
they are
$\clAll$, $\clVak$, $\clEmpty$.
\end{enumerate}

\item
\begin{enumerate}[label=\upshape{(\alph*)}]
\item
There are exactly 4 $(\clAll,\clIo)$\hyp{}clonoids;
they are
$\clAll$, $\clVak$, $\clVako$, $\clEmpty$.

\item
There are exactly 4 $(\clAll,\clIi)$\hyp{}clonoids;
they are
$\clAll$, $\clVak$, $\clVaki$, $\clEmpty$.

\item
There are exactly 3 $(\clAll,\clI)$\hyp{}clonoids;
they are
$\clAll$, $\clVak$, $\clEmpty$.

\item
There are exactly 3 $(\clAll,\clIstar)$\hyp{}clonoids;
they are
$\clAll$, $\clVak$, $\clEmpty$.

\item
There are exactly 3 $(\clAll,\clOmegaOne)$\hyp{}clonoids;
they are
$\clAll$, $\clVak$, $\clEmpty$.
\end{enumerate}
\end{enumerate}
\end{proposition}

\begin{proof}
This follows immediately from Theorem~\ref{thm:disc-clonoids} and Lemmata~\ref{lem:C1C2Vak} and \ref{lem:C1Istar}.
\end{proof}

The cases where the target clone $C_2$ includes $\clLc$ or $\clSM$ have been descibed earlier papers of the author's.
We quote here only the results for the cases where the target clone is $\clLc$ or $\clSM$.
The results for their superclones can be found in the earlier papers, and we simply refer the reader to those papers.

\begin{proposition}[{\cite[Theorem~7.1, Table~3]{CouLeh-Lcstability}}]
There are exactly 26
$(\clSc,\clLc)$\hyp{}clonoids,
and they are the following:
$\clAll$,
$\clOX$, $\clIX$,
$\clXO$, $\clXI$,
$\clEq$,
$\clNeq$,
$\clOO$, $\clOI$, $\clIO$, $\clII$,
$\clS \cup \clRefl$,
$\clSc \cup \clReflOO$,
$\clIntVal{\clS}{1}{0} \cup \clReflII$,
$\clSc \cup \clReflII$,
$\clIntVal{\clS}{1}{0} \cup \clReflOO$,
$\clRefl$,
$\clS$,
$\clReflOO$,
$\clReflII$,
$\clSc$,
$\clIntVal{\clS}{1}{0}$,
$\clVak$,
$\clVako$,
$\clVaki$,
$\clEmpty$.
\end{proposition}

\begin{proposition}[{\cite[Theorem 5.1, Table 1]{Lehtonen-SM}}]
\label{prop:ScSM}
There are precisely 57 $(\clSc,\clSM)$\hyp{}clonoids, and they are the following:
$\clAll$,
$\clEiio$, $\clEioi$, $\clEiii$, $\clEioo$,
$\clOXCI$, $\clXOCI$, $\clIXCO$, $\clXICO$,
$\clOX$, $\clXO$, $\clIX$, $\clXI$, $\clEq$, $\clNeq$,
$\clOIC$, $\clIOC$,
$\clOICO$, $\clIOCO$,
$\clOICI$, $\clIOCI$,
$\clOOCI$, $\clIICO$,
$\clOO$, $\clII$, $\clOI$, $\clIO$,
$\clSmin$, $\clSmaj$,
$\clIntNeq{\clSmin}$, $\clIntNeq{\clSmaj}$,
$\clSminOX$, $\clSmajIX$,
$\clSminXO$, $\clSmajXI$,
$\clSminOICO$, $\clSmajIOCI$,
$\clSminIOCO$, $\clSmajOICI$,
$\clSminOI$, $\clSmajIO$,
$\clSminIO$, $\clSmajOI$,
$\clSminOO$, $\clSmajII$,
$\clS$, $\clSc$, $\clIntVal{\clS}{1}{0}$,
$\clRefl$, $\clReflOOC$, $\clReflIIC$, $\clReflOO$, $\clReflII$,
$\clVak$, $\clVako$, $\clVaki$,
$\clEmpty$.
\end{proposition}

\begin{proposition}
\label{prop:ScVc-clonoids}
There are precisely 123
$(\clSc,\clVc)$\hyp{}clonoids.
They are the classes listed in Table~\ref{table:ScVc}.
\end{proposition}

The Hasse diagram of $\closys{(\clSc,\clVc)}$ is shown in Figure~\ref{fig:ScVc-clonoids}.

\begingroup
\newcommand{\sSM}{\cite{Lehtonen-SM}}
\newcommand{\sHE}{\ref{thm:ScVc-stab-gen}}
\begin{table}
\begingroup
\scriptsize
\begin{tabular}{llllllll}
\toprule
&&&&& $K C \subseteq K$ & $C K \subseteq K$ & \\
$A$ & $B$ & $C$ & $D$ & $K$ & iff $C \subseteq \ldots{}$ & iff $C \subseteq \ldots{}$ & source \\
\midrule
$\clEmpty$ & $\clEmpty$ & $\clEmpty$ & $\clEmpty$ & $\clEmpty$ & $\clAll$ & $\clAll$ & \sSM \\
$\clEmpty$ & $\clEmpty$ & $\clEmpty$ & $\clVaki$ & $\clVaki$ & $\clAll$ & $\clXI$ & \sSM \\
$\clEmpty$ & $\clEmpty$ & $\clEmpty$ & $\clSmajII$ & $\clSmajII$ & $\clSc$ & $\clWk{2}$ & \sSM \\
$\clEmpty$ & $\clEmpty$ & $\clEmpty$ & $\clReflII$ & $\clReflII$ & $\clSc$ & $\clXI$ & \sSM \\
$\clEmpty$ & $\clEmpty$ & $\clEmpty$ & $\clSmajII \cup \clReflII$ & $\clSmajII \cup \clReflII$ & $\clSc$ & $\clVi$ & \sHE \\
$\clEmpty$ & $\clEmpty$ & $\clEmpty$ & $\clII$ & $\clII$ & $\clOI$ & $\clXI$ & \sSM \\

$\clEmpty$ & $\clEmpty$ & $\clSmajIO$ & $\clEmpty$ & $\clSmajIO$ & $\clSc$ & $\clTcWk{2}$ & \sSM \\
$\clEmpty$ & $\clEmpty$ & $\clSmajIO$ & $\clVaki$ & $\clSmajIO \cup \clVaki$ & $\clSc$ & $\clMWk{2}$ & \sSM \\
$\clEmpty$ & $\clEmpty$ & $\clSmajIO$ & $\clSmajII$ & $\clSmajIX$ & $\clSc$ & $\clWk{2}$ & \sSM \\
$\clEmpty$ & $\clEmpty$ & $\clSmajIO$ & $\clSmajII \cup \clReflII$ & $\clSmajIX \cup \clReflII$ & $\clSc$ & $\clVi$ & \sHE \\
$\clEmpty$ & $\clEmpty$ & $\clSmajIO$ & $\clII$ & $\clSmajIO \cup \clII$ & $\clSc$ & $\clWk{3}$ & \sHE \\

$\clEmpty$ & $\clEmpty$ & $\clIO$ & $\clEmpty$ & $\clIO$ & $\clOI$ & $\clOI$ & \sSM \\
$\clEmpty$ & $\clEmpty$ & $\clIO$ & $\clVaki$ & $\clIOCI$ & $\clOI$ & $\clMi$ & \sSM \\
$\clEmpty$ & $\clEmpty$ & $\clIO$ & $\clSmajII$ & $\clIO \cup \clSmajII$ & $\clSc$ & $\clVi$ & \sHE \\
$\clEmpty$ & $\clEmpty$ & $\clIO$ & $\clII$ & $\clIX$ & $\clOX$ & $\clXI$ & \sSM \\

$\clEmpty$ & $\clSmajOI$ & $\clEmpty$ & $\clEmpty$ & $\clSmajOI$ & $\clSc$ & $\clTcWk{2}$ & \sSM \\
$\clEmpty$ & $\clSmajOI$ & $\clEmpty$ & $\clVaki$ & $\clSmajOI \cup \clVaki$ & $\clSc$ & $\clMWk{2}$ & \sSM \\
$\clEmpty$ & $\clSmajOI$ & $\clEmpty$ & $\clSmajII$ & $\clSmajXI$ & $\clSc$ & $\clWk{2}$ & \sSM \\
$\clEmpty$ & $\clSmajOI$ & $\clEmpty$ & $\clSmajII \cup \clReflII$ & $\clSmajXI \cup \clReflII$ & $\clSc$ & $\clVi$ & \sHE \\
$\clEmpty$ & $\clSmajOI$ & $\clEmpty$ & $\clII$ & $\clSmajOI \cup \clII$ & $\clSc$ & $\clWk{3}$ & \sHE \\

$\clEmpty$ & $\clSmajOI$ & $\clSmajIO$ & $\clSmajII$ & $\clSmaj$ & $\clS$ & $\clWk{2}$ & \sSM \\
$\clEmpty$ & $\clSmajOI$ & $\clSmajIO$ & $\clSmajII \cup \clReflII$ & $\clSmaj \cup \clReflII$ & $\clSc$ & $\clVi$ & \sHE \\
$\clEmpty$ & $\clSmajOI$ & $\clSmajIO$ & $\clII$ & $\clSmaj \cup \clII$ & $\clSc$ & $\clWk{3}$ & \sHE \\

$\clEmpty$ & $\clSmajOI$ & $\clIO$ & $\clSmajII$ & $\clIO \cup \clSmajXI$ & $\clSc$ & $\clVi$ & \sHE \\
$\clEmpty$ & $\clSmajOI$ & $\clIO$ & $\clII$ & $\clIX \cup \clSmajOI$ & $\clSc$ & $\clWk{3}$ & \sHE \\

$\clEmpty$ & $\clOI$ & $\clEmpty$ & $\clEmpty$ & $\clOI$ & $\clOI$ & $\clOI$ & \sSM \\
$\clEmpty$ & $\clOI$ & $\clEmpty$ & $\clVaki$ & $\clOICI$ & $\clOI$ & $\clMi$ & \sSM \\
$\clEmpty$ & $\clOI$ & $\clEmpty$ & $\clSmajII$ & $\clOI \cup \clSmajII$ & $\clSc$ & $\clVi$ & \sHE \\
$\clEmpty$ & $\clOI$ & $\clEmpty$ & $\clII$ & $\clXI$ & $\clXI$ & $\clXI$ & \sSM \\

$\clEmpty$ & $\clOI$ & $\clSmajIO$ & $\clSmajII$ & $\clOI \cup \clSmaj$ & $\clSc$ & $\clVi$ & \sHE \\
$\clEmpty$ & $\clOI$ & $\clSmajIO$ & $\clII$ & $\clSmaj \cup \clXI$ & $\clSc$ & $\clWk{3}$ & \sHE \\

$\clEmpty$ & $\clOI$ & $\clIO$ & $\clII$ & $\clEioo$ & $\clOI$ & $\clWk{2}$ & \sSM \\

$\clVako$ & $\clEmpty$ & $\clEmpty$ & $\clEmpty$ & $\clVako$ & $\clAll$ & $\clOX$ & \sSM \\
$\clVako$ & $\clEmpty$ & $\clEmpty$ & $\clVaki$ & $\clVak$ & $\clAll$ & $\clAll$ & \sSM \\
$\clVako$ & $\clEmpty$ & $\clEmpty$ & $\clSmajII$ & $\clSmajII \cup \clVako$ & $\clSc$ & $\clV$ & \sHE \\
$\clVako$ & $\clEmpty$ & $\clEmpty$ & $\clReflII$ & $\clReflII \cup \clVako$ & $\clSc$ & $\clM$ & \sSM \\
$\clVako$ & $\clEmpty$ & $\clEmpty$ & $\clSmajII \cup \clReflII$ & $\clSmajII \cup \clReflII \cup \clVako$ & $\clSc$ & $\clV$ & \sHE \\
$\clVako$ & $\clEmpty$ & $\clEmpty$ & $\clII$ & $\clIICO$ & $\clOI$ & $\clM$ & \sSM \\

$\clVako$ & $\clEmpty$ & $\clSmajIO$ & $\clEmpty$ & $\clSmajIO \cup \clVako$ & $\clSc$ & $\clVo$ & \sHE \\
$\clVako$ & $\clEmpty$ & $\clSmajIO$ & $\clVaki$ & $\clSmajIO \cup \clVak$ & $\clSc$ & $\clV$ & \sHE \\
$\clVako$ & $\clEmpty$ & $\clSmajIO$ & $\clSmajII$ & $\clSmajIX \cup \clVako$ & $\clSc$ & $\clV$ & \sHE \\
$\clVako$ & $\clEmpty$ & $\clSmajIO$ & $\clSmajII \cup \clReflII$ & $\clSmajIX \cup \clReflII \cup \clVako$ & $\clSc$ & $\clV$ & \sHE \\
$\clVako$ & $\clEmpty$ & $\clSmajIO$ & $\clII$ & $\clSmajIO \cup \clIICO$ & $\clSc$ & $\clV$ & \sHE \\

$\clVako$ & $\clEmpty$ & $\clIO$ & $\clEmpty$ & $\clIOCO$ & $\clOI$ & $\clMo$ & \sSM \\
$\clVako$ & $\clEmpty$ & $\clIO$ & $\clVaki$ & $\clIO \cup \clVak$ & $\clOI$ & $\clM$ & \sSM \\
$\clVako$ & $\clEmpty$ & $\clIO$ & $\clSmajII$ & $\clIO \cup \clSmajII \cup \clVako$ & $\clSc$ & $\clV$ & \sHE \\
$\clVako$ & $\clEmpty$ & $\clIO$ & $\clII$ & $\clIXCO$ & $\clOX$ & $\clM$ & \sSM \\

$\clVako$ & $\clSmajOI$ & $\clEmpty$ & $\clEmpty$ & $\clSmajOI \cup \clVako$ & $\clSc$ & $\clVo$ & \sHE \\
$\clVako$ & $\clSmajOI$ & $\clEmpty$ & $\clVaki$ & $\clSmajOI \cup \clVak$ & $\clSc$ & $\clV$ & \sHE \\
$\clVako$ & $\clSmajOI$ & $\clEmpty$ & $\clSmajII$ & $\clSmajXI \cup \clVako$ & $\clSc$ & $\clV$ & \sHE \\
$\clVako$ & $\clSmajOI$ & $\clEmpty$ & $\clSmajII \cup \clReflII$ & $\clSmajXI \cup \clReflII \cup \clVako$ & $\clSc$ & $\clV$ & \sHE \\
$\clVako$ & $\clSmajOI$ & $\clEmpty$ & $\clII$ & $\clSmajOI \cup \clIICO$ & $\clSc$ & $\clV$ & \sHE \\

$\clVako$ & $\clSmajOI$ & $\clSmajIO$ & $\clSmajII$ & $\clSmaj \cup \clVako$ & $\clS$ & $\clV$ & \sHE \\
$\clVako$ & $\clSmajOI$ & $\clSmajIO$ & $\clSmajII \cup \clReflII$ & $\clSmaj \cup \clReflII \cup \clVako$ & $\clSc$ & $\clV$ & \sHE \\
$\clVako$ & $\clSmajOI$ & $\clSmajIO$ & $\clII$ & $\clSmaj \cup \clIICO$ & $\clSc$ & $\clV$ & \sHE \\

$\clVako$ & $\clSmajOI$ & $\clIO$ & $\clSmajII$ & $\clIO \cup \clSmajXI \cup \clVako$ & $\clSc$ & $\clV$ & \sHE \\
$\clVako$ & $\clSmajOI$ & $\clIO$ & $\clII$ & $\clIX \cup \clSmajOI \cup \clVako$ & $\clSc$ & $\clV$ & \sHE \\

$\clVako$ & $\clOI$ & $\clEmpty$ & $\clEmpty$ & $\clOICO$ & $\clOI$ & $\clMo$ & \sSM \\
$\clVako$ & $\clOI$ & $\clEmpty$ & $\clVaki$ & $\clOIC$ & $\clOI$ & $\clM$ & \sSM \\
$\clVako$ & $\clOI$ & $\clEmpty$ & $\clSmajII$ & $\clOI \cup \clSmajII \cup \clVako$ & $\clSc$ & $\clV$ & \sHE \\
$\clVako$ & $\clOI$ & $\clEmpty$ & $\clII$ & $\clXICO$ & $\clXI$ & $\clM$ & \sSM \\

$\clVako$ & $\clOI$ & $\clSmajIO$ & $\clSmajII$ & $\clOI \cup \clSmaj \cup \clVako$ & $\clSc$ & $\clV$ & \sHE \\
\midrule
\multicolumn{8}{r}{\textsl{Continues to the next page.}} \\
\bottomrule
\end{tabular}
\endgroup
\caption{$(\protect\clSc,\protect\clVc)$\hyp{}clonoids.}
\label{table:ScVc}
\end{table}

\begin{table}
\ContinuedFloat
\begingroup
\scriptsize
\begin{tabular}{llllllll}
\toprule
\multicolumn{8}{l}{\textsl{Continued from the previous page.}} \\
\midrule
&&&&& $K C \subseteq K$ & $C K \subseteq K$ & \\
$A$ & $B$ & $C$ & $D$ & $K$ & iff $C \subseteq \ldots{}$ & iff $C \subseteq \ldots{}$ & source \\
\midrule
$\clVako$ & $\clOI$ & $\clSmajIO$ & $\clII$ & $\clXI \cup \clSmaj \cup \clVako$ & $\clSc$ & $\clV$ & \sHE \\

$\clVako$ & $\clOI$ & $\clIO$ & $\clII$ & $\clEioo \cup \clVako$ & $\clOI$ & $\clV$ & \sHE \\

$\clReflOO$ & $\clEmpty$ & $\clEmpty$ & $\clEmpty$ & $\clReflOO$ & $\clSc$ & $\clOX$ & \sSM \\
$\clReflOO$ & $\clEmpty$ & $\clEmpty$ & $\clVaki$ & $\clReflOO \cup \clVaki$ & $\clSc$ & $\clM$ & \sSM \\
$\clReflOO$ & $\clEmpty$ & $\clEmpty$ & $\clSmajII$ & $\clSmajII \cup \clReflOO$ & $\clSc$ & $\clV$ & \sHE \\
$\clReflOO$ & $\clEmpty$ & $\clEmpty$ & $\clReflII$ & $\clRefl$ & $\clS$ & $\clAll$ & \sSM \\
$\clReflOO$ & $\clEmpty$ & $\clEmpty$ & $\clSmajII \cup \clReflII$ & $\clSmajII \cup \clRefl$ & $\clSc$ & $\clV$ & \sHE \\
$\clReflOO$ & $\clEmpty$ & $\clEmpty$ & $\clII$ & $\clII \cup \clReflOO$ & $\clSc$ & $\clV$ & \sHE \\

$\clReflOO$ & $\clEmpty$ & $\clSmajIO$ & $\clEmpty$ & $\clSmajIO \cup \clReflOO$ & $\clSc$ & $\clVo$ & \sHE \\
$\clReflOO$ & $\clEmpty$ & $\clSmajIO$ & $\clVaki$ & $\clSmajIO \cup \clReflOO \cup \clVaki$ & $\clSc$ & $\clV$ & \sHE \\
$\clReflOO$ & $\clEmpty$ & $\clSmajIO$ & $\clSmajII$ & $\clSmajIX \cup \clReflOO$ & $\clSc$ & $\clV$ & \sHE \\
$\clReflOO$ & $\clEmpty$ & $\clSmajIO$ & $\clSmajII \cup \clReflII$ & $\clSmajIX \cup \clRefl$ & $\clSc$ & $\clV$ & \sHE \\
$\clReflOO$ & $\clEmpty$ & $\clSmajIO$ & $\clII$ & $\clSmajIO \cup \clII \cup \clReflOO$ & $\clSc$ & $\clV$ & \sHE \\

$\clReflOO$ & $\clEmpty$ & $\clIO$ & $\clEmpty$ & $\clIO \cup \clReflOO$ & $\clSc$ & $\clVo$ & \sHE \\
$\clReflOO$ & $\clEmpty$ & $\clIO$ & $\clVaki$ & $\clIO \cup \clReflOO \cup \clVaki$ & $\clSc$ & $\clV$ & \sHE \\
$\clReflOO$ & $\clEmpty$ & $\clIO$ & $\clSmajII$ & $\clIO \cup \clSmajII \cup \clReflOO$ & $\clSc$ & $\clV$ & \sHE \\
$\clReflOO$ & $\clEmpty$ & $\clIO$ & $\clII$ & $\clIX \cup \clReflOO$ & $\clSc$ & $\clV$ & \sHE \\

$\clReflOO$ & $\clSmajOI$ & $\clEmpty$ & $\clEmpty$ & $\clSmajOI \cup \clReflOO$ & $\clSc$ & $\clVo$ & \sHE \\
$\clReflOO$ & $\clSmajOI$ & $\clEmpty$ & $\clVaki$ & $\clSmajOI \cup \clReflOO \cup \clVaki$ & $\clSc$ & $\clV$ & \sHE \\
$\clReflOO$ & $\clSmajOI$ & $\clEmpty$ & $\clSmajII$ & $\clSmajXI \cup \clReflOO$ & $\clSc$ & $\clV$ & \sHE \\
$\clReflOO$ & $\clSmajOI$ & $\clEmpty$ & $\clSmajII \cup \clReflII$ & $\clSmajXI \cup \clRefl$ & $\clSc$ & $\clV$ & \sHE \\
$\clReflOO$ & $\clSmajOI$ & $\clEmpty$ & $\clII$ & $\clSmajOI \cup \clII \cup \clReflOO$ & $\clSc$ & $\clV$ & \sHE \\

$\clReflOO$ & $\clSmajOI$ & $\clSmajIO$ & $\clSmajII$ & $\clSmaj \cup \clReflOO$ & $\clSc$ & $\clV$ & \sHE \\
$\clReflOO$ & $\clSmajOI$ & $\clSmajIO$ & $\clSmajII \cup \clReflII$ & $\clSmaj \cup \clRefl$ & $\clS$ & $\clV$ & \sHE \\
$\clReflOO$ & $\clSmajOI$ & $\clSmajIO$ & $\clII$ & $\clSmaj \cup \clII \cup \clReflOO$ & $\clSc$ & $\clV$ & \sHE \\

$\clReflOO$ & $\clSmajOI$ & $\clIO$ & $\clSmajII$ & $\clIO \cup \clSmajXI \cup \clReflOO$ & $\clSc$ & $\clV$ & \sHE \\
$\clReflOO$ & $\clSmajOI$ & $\clIO$ & $\clII$ & $\clIX \cup \clSmajOI \cup \clReflOO$ & $\clSc$ & $\clV$ & \sHE \\

$\clReflOO$ & $\clOI$ & $\clEmpty$ & $\clEmpty$ & $\clOI \cup \clReflOO$ & $\clSc$ & $\clVo$ & \sHE \\
$\clReflOO$ & $\clOI$ & $\clEmpty$ & $\clVaki$ & $\clOI \cup \clReflOO \cup \clVaki$ & $\clSc$ & $\clV$ & \sHE \\
$\clReflOO$ & $\clOI$ & $\clEmpty$ & $\clSmajII$ & $\clOI \cup \clSmajII \cup \clReflOO$ & $\clSc$ & $\clV$ & \sHE \\
$\clReflOO$ & $\clOI$ & $\clEmpty$ & $\clII$ & $\clXI \cup \clReflOO$ & $\clSc$ & $\clV$ & \sHE \\

$\clReflOO$ & $\clOI$ & $\clSmajIO$ & $\clSmajII$ & $\clOI \cup \clSmaj \cup \clReflOO$ & $\clSc$ & $\clV$ & \sHE \\
$\clReflOO$ & $\clOI$ & $\clSmajIO$ & $\clII$ & $\clSmaj \cup \clXI \cup \clReflOO$ & $\clSc$ & $\clV$ & \sHE \\

$\clReflOO$ & $\clOI$ & $\clIO$ & $\clII$ & $\clEioo \cup \clReflOO$ & $\clSc$ & $\clV$ & \sHE \\

$\clOO$ & $\clEmpty$ & $\clEmpty$ & $\clEmpty$ & $\clOO$ & $\clOI$ & $\clOX$ & \sSM \\
$\clOO$ & $\clEmpty$ & $\clEmpty$ & $\clVaki$ & $\clOOCI$ & $\clOI$ & $\clM$ & \sSM \\
$\clOO$ & $\clEmpty$ & $\clEmpty$ & $\clSmajII$ & $\clOO \cup \clSmajII$ & $\clSc$ & $\clV$ & \sHE \\
$\clOO$ & $\clEmpty$ & $\clEmpty$ & $\clII$ & $\clEq$ & $\clOI$ & $\clAll$ & \sSM \\

$\clOO$ & $\clEmpty$ & $\clSmajIO$ & $\clEmpty$ & $\clOO \cup \clSmajIO$ & $\clSc$ & $\clVo$ & \sHE \\
$\clOO$ & $\clEmpty$ & $\clSmajIO$ & $\clVaki$ & $\clOO \cup \clSmajIO \cup \clVaki$ & $\clSc$ & $\clV$ & \sHE \\
$\clOO$ & $\clEmpty$ & $\clSmajIO$ & $\clSmajII$ & $\clOO \cup \clSmajIX$ & $\clSc$ & $\clV$ & \sHE \\
$\clOO$ & $\clEmpty$ & $\clSmajIO$ & $\clII$ & $\clEq \cup \clSmajIO$ & $\clSc$ & $\clV$ & \sHE \\

$\clOO$ & $\clEmpty$ & $\clIO$ & $\clEmpty$ & $\clXO$ & $\clXI$ & $\clOX$ & \sSM \\
$\clOO$ & $\clEmpty$ & $\clIO$ & $\clVaki$ & $\clXO \cup \clVaki$ & $\clXI$ & $\clM$ & \sSM \\
$\clOO$ & $\clEmpty$ & $\clIO$ & $\clSmajII$ & $\clXO \cup \clSmajII$ & $\clSc$ & $\clV$ & \sHE \\
$\clOO$ & $\clEmpty$ & $\clIO$ & $\clII$ & $\clEioi$ & $\clOI$ & $\clM$ & \sSM \\

$\clOO$ & $\clSmajOI$ & $\clEmpty$ & $\clEmpty$ & $\clOO \cup \clSmajOI$ & $\clSc$ & $\clVo$ & \sHE \\
$\clOO$ & $\clSmajOI$ & $\clEmpty$ & $\clVaki$ & $\clOO \cup \clSmajOI \cup \clVaki$ & $\clSc$ & $\clV$ & \sHE \\
$\clOO$ & $\clSmajOI$ & $\clEmpty$ & $\clSmajII$ & $\clOO \cup \clSmajXI$ & $\clSc$ & $\clV$ & \sHE \\
$\clOO$ & $\clSmajOI$ & $\clEmpty$ & $\clII$ & $\clEq \cup \clSmajOI$ & $\clSc$ & $\clV$ & \sHE \\

$\clOO$ & $\clSmajOI$ & $\clSmajIO$ & $\clSmajII$ & $\clOO \cup \clSmaj$ & $\clSc$ & $\clV$ & \sHE \\
$\clOO$ & $\clSmajOI$ & $\clSmajIO$ & $\clII$ & $\clEq \cup \clSmaj$ & $\clSc$ & $\clV$ & \sHE \\

$\clOO$ & $\clSmajOI$ & $\clIO$ & $\clSmajII$ & $\clXO \cup \clSmaj$ & $\clSc$ & $\clV$ & \sHE \\
$\clOO$ & $\clSmajOI$ & $\clIO$ & $\clII$ & $\clEioi \cup \clSmaj$ & $\clSc$ & $\clV$ & \sHE \\

$\clOO$ & $\clOI$ & $\clEmpty$ & $\clEmpty$ & $\clOX$ & $\clOX$ & $\clOX$ & \sSM \\
$\clOO$ & $\clOI$ & $\clEmpty$ & $\clVaki$ & $\clOXCI$ & $\clOX$ & $\clM$ & \sSM \\
$\clOO$ & $\clOI$ & $\clEmpty$ & $\clSmajII$ & $\clOX \cup \clSmajII$ & $\clSc$ & $\clV$ & \sHE \\
$\clOO$ & $\clOI$ & $\clEmpty$ & $\clII$ & $\clEiio$ & $\clOI$ & $\clM$ & \sSM \\

$\clOO$ & $\clOI$ & $\clSmajIO$ & $\clSmajII$ & $\clOX \cup \clSmaj$ & $\clSc$ & $\clV$ & \sHE \\
$\clOO$ & $\clOI$ & $\clSmajIO$ & $\clII$ & $\clEiio \cup \clSmaj$ & $\clSc$ & $\clV$ & \sHE \\

$\clOO$ & $\clOI$ & $\clIO$ & $\clII$ & $\clAll$ & $\clAll$ & $\clAll$ & \sSM \\
\bottomrule
\end{tabular}
\endgroup
\caption{$(\protect\clSc,\protect\clVc)$\hyp{}clonoids.}
\end{table}
\endgroup

\begin{figure}
\begin{center}
\scalebox{0.675}{%
\begin{tikzpicture}[baseline, scale=0.6]
   \tikzmath{\Ax = -4; \Ay = 12; \Bx = -4; \By = 2; \Cx = 4; \Cy = 2; \Dx = 0; \Dy = 1; \Rx = 1; \Ry = 1;}
   \node [vanha, label = {[shift={(0,-0.65)}]$\clEmpty$}] (0000) at (0,0) {};
   \node [vanha, label = {[shift={(0.35,-0.45)}]$\clVaki$}] (0001) at ($(0000)+(\Dx,\Dy)$) {};
   \node [vanha, label = {[shift={(-0.65,-0.4)}]$\clSmajII$}] (0002) at ($(0001)+(\Dx,\Dy)$) {};
   \node [vanha, label = {[shift={(0.4,-0.25)}]$\clReflII$}] (000R) at ($(0001)+(\Rx,\Ry)$) {};
   \node [uusi] (0003) at ($(0002)+(\Dx,\Dy)$) {};
   \node [vanha, label = {[shift={(0.55,-0.45)}]$\clII$}] (0004) at ($(0003)+(\Dx,\Dy)$) {};

   \node [vanha, label = below:$\clSmajIO$] (0010) at ($(0000)+(\Cx,\Cy)$) {};
   \node [vanha] (0011) at ($(0010)+(\Dx,\Dy)$) {};
   \node [vanha] (0012) at ($(0011)+(\Dx,\Dy)$) {};
   \node [uusi] (0013) at ($(0012)+(\Dx,\Dy)$) {};
   \node [uusi] (0014) at ($(0013)+(\Dx,\Dy)$) {};

   \node [vanha, label = below:$\clIO$] (0020) at ($(0010)+(\Cx,\Cy)$) {};
   \node [vanha] (0021) at ($(0020)+(\Dx,\Dy)$) {};
   \node [uusi] (0022) at ($(0021)+(\Dx,\Dy)$) {};
   \node [vanha] (0024) at ($(0022)+(2*\Dx,2*\Dy)$) {};

   \node [vanha, label = below:$\clSmajOI$] (0100) at ($(0000)+(\Bx,\By)$) {};
   \node [vanha] (0101) at ($(0100)+(\Dx,\Dy)$) {};
   \node [vanha] (0102) at ($(0101)+(\Dx,\Dy)$) {};
   \node [uusi] (0103) at ($(0102)+(\Dx,\Dy)$) {};
   \node [uusi] (0104) at ($(0103)+(\Dx,\Dy)$) {};

   \node [vanha, label = below:$\clOI$] (0200) at ($(0100)+(\Bx,\By)$) {};
   \node [vanha] (0201) at ($(0200)+(\Dx,\Dy)$) {};
   \node [uusi] (0202) at ($(0201)+(\Dx,\Dy)$) {};
   \node [vanha] (0204) at ($(0202)+(2*\Dx,2*\Dy)$) {};

   \node [vanha] (0112) at ($(0012)+(\Bx,\By)$) {};
   \node [uusi] (0113) at ($(0112)+(\Dx,\Dy)$) {};
   \node [uusi] (0114) at ($(0113)+(\Dx,\Dy)$) {};

   \node [uusi] (0122) at ($(0022)+(\Bx,\By)$) {};
   \node [uusi] (0124) at ($(0122)+(2*\Dx,2*\Dy)$) {};

   \node [uusi] (0212) at ($(0112)+(\Bx,\By)$) {};
   \node [uusi] (0214) at ($(0212)+(2*\Dx,2*\Dy)$) {};

   \node [vanha, label = below:$\clEioo$] (0224) at ($(0124)+(\Bx,\By)$) {};
   
   \draw (0000) -- (0001) -- (0002) -- (0003) --(0004); \draw (0001) -- (000R) -- (0003);
   \draw (0010) -- (0011) -- (0012) -- (0013) --(0014);
   \draw (0020) -- (0021) -- (0022) -- (0024);
   \draw (0100) -- (0101) -- (0102) -- (0103) --(0104);
   \draw (0200) -- (0201) -- (0202) -- (0204);
   \draw (0112) -- (0113) -- (0114);
   \draw (0122) -- (0124);
   \draw (0212) -- (0214);
   \draw (0000) -- (0010) -- (0020);
   \draw (0001) -- (0011) -- (0021);
   \draw (0002) -- (0012) -- (0022);
   \draw (0003) -- (0013);
   \draw (0004) -- (0014) -- (0024);
   \draw (0102) -- (0112) -- (0122);
   \draw (0103) -- (0113);
   \draw (0104) -- (0114) -- (0124);
   \draw (0202) -- (0212);
   \draw (0204) -- (0214) -- (0224);
   \draw (0000) -- (0100) -- (0200);
   \draw (0001) -- (0101) -- (0201);
   \draw (0002) -- (0102) -- (0202);
   \draw (0003) -- (0103);
   \draw (0004) -- (0104) -- (0204);
   \draw (0012) -- (0112) -- (0212);
   \draw (0013) -- (0113);
   \draw (0014) -- (0114) -- (0214);
   \draw (0022) -- (0122);
   \draw (0024) -- (0124) -- (0224);

   \node [vanha, label = {[shift={(-0.07,-0.7)}]$\clVako$}] (1000) at ($(0000)+(\Ax,\Ay)$) {};
   \node [vanha] (1001) at ($(1000)+(\Dx,\Dy)$) {};
   \node [uusi] (1002) at ($(1001)+(\Dx,\Dy)$) {};
   \node [vanha] (100R) at ($(1001)+(\Rx,\Ry)$) {};
   \node [uusi] (1003) at ($(1002)+(\Dx,\Dy)$) {};
   \node [vanha] (1004) at ($(1003)+(\Dx,\Dy)$) {};

   \node [uusi] (1010) at ($(1000)+(\Cx,\Cy)$) {};
   \node [uusi] (1011) at ($(1010)+(\Dx,\Dy)$) {};
   \node [uusi] (1012) at ($(1011)+(\Dx,\Dy)$) {};
   \node [uusi] (1013) at ($(1012)+(\Dx,\Dy)$) {};
   \node [uusi] (1014) at ($(1013)+(\Dx,\Dy)$) {};

   \node [vanha] (1020) at ($(1010)+(\Cx,\Cy)$) {};
   \node [vanha] (1021) at ($(1020)+(\Dx,\Dy)$) {};
   \node [uusi] (1022) at ($(1021)+(\Dx,\Dy)$) {};
   \node [vanha] (1024) at ($(1022)+(2*\Dx,2*\Dy)$) {};

   \node [uusi] (1100) at ($(1000)+(\Bx,\By)$) {};
   \node [uusi] (1101) at ($(1100)+(\Dx,\Dy)$) {};
   \node [uusi] (1102) at ($(1101)+(\Dx,\Dy)$) {};
   \node [uusi] (1103) at ($(1102)+(\Dx,\Dy)$) {};
   \node [uusi] (1104) at ($(1103)+(\Dx,\Dy)$) {};

   \node [vanha] (1200) at ($(1100)+(\Bx,\By)$) {};
   \node [vanha] (1201) at ($(1200)+(\Dx,\Dy)$) {};
   \node [uusi] (1202) at ($(1201)+(\Dx,\Dy)$) {};
   \node [vanha] (1204) at ($(1202)+(2*\Dx,2*\Dy)$) {};

   \node [uusi] (1112) at ($(1012)+(\Bx,\By)$) {};
   \node [uusi] (1113) at ($(1112)+(\Dx,\Dy)$) {};
   \node [uusi] (1114) at ($(1113)+(\Dx,\Dy)$) {};

   \node [uusi] (1122) at ($(1022)+(\Bx,\By)$) {};
   \node [uusi] (1124) at ($(1122)+(2*\Dx,2*\Dy)$) {};

   \node [uusi] (1212) at ($(1112)+(\Bx,\By)$) {};
   \node [uusi] (1214) at ($(1212)+(2*\Dx,2*\Dy)$) {};

   \node [uusi, label = {[shift={(-0.75,-1.35)},rotate=-71.57]$\clEioo \cup \clVako$}] (1224) at ($(1124)+(\Bx,\By)$) {};

   \draw (1000) -- (1001) -- (1002) -- (1003) --(1004); \draw (1001) -- (100R) -- (1003);
   \draw (1010) -- (1011) -- (1012) -- (1013) --(1014);
   \draw (1020) -- (1021) -- (1022) -- (1024);
   \draw (1100) -- (1101) -- (1102) -- (1103) --(1104);
   \draw (1200) -- (1201) -- (1202) -- (1204);
   \draw (1112) -- (1113) -- (1114);
   \draw (1122) -- (1124);
   \draw (1212) -- (1214);
   \draw (1000) -- (1010) -- (1020);
   \draw (1001) -- (1011) -- (1021);
   \draw (1002) -- (1012) -- (1022);
   \draw (1003) -- (1013);
   \draw (1004) -- (1014) -- (1024);
   \draw (1102) -- (1112) -- (1122);
   \draw (1103) -- (1113);
   \draw (1104) -- (1114) -- (1124);
   \draw (1202) -- (1212);
   \draw (1204) -- (1214) -- (1224);
   \draw (1000) -- (1100) -- (1200);
   \draw (1001) -- (1101) -- (1201);
   \draw (1002) -- (1102) -- (1202);
   \draw (1003) -- (1103);
   \draw (1004) -- (1104) -- (1204);
   \draw (1012) -- (1112) -- (1212);
   \draw (1013) -- (1113);
   \draw (1014) -- (1114) -- (1214);
   \draw (1022) -- (1122);
   \draw (1024) -- (1124) -- (1224);

   \node [vanha, label = {[shift={(-0.35,-0.65)},rotate=-71.57]$\clReflOO$}] (2000) at ($(1000)+(\Ax,\Ay)$) {};
   \node [vanha] (2001) at ($(2000)+(\Dx,\Dy)$) {};
   \node [uusi] (2002) at ($(2001)+(\Dx,\Dy)$) {};
   \node [vanha, label = {[shift={(0.7,0.05)}]$\clRefl$}] (200R) at ($(2001)+(\Rx,\Ry)$) {};
   \node [uusi] (2003) at ($(2002)+(\Dx,\Dy)$) {};
   \node [uusi] (2004) at ($(2003)+(\Dx,\Dy)$) {};

   \node [uusi] (2010) at ($(2000)+(\Cx,\Cy)$) {};
   \node [uusi] (2011) at ($(2010)+(\Dx,\Dy)$) {};
   \node [uusi] (2012) at ($(2011)+(\Dx,\Dy)$) {};
   \node [uusi] (2013) at ($(2012)+(\Dx,\Dy)$) {};
   \node [uusi] (2014) at ($(2013)+(\Dx,\Dy)$) {};

   \node [uusi] (2020) at ($(2010)+(\Cx,\Cy)$) {};
   \node [uusi] (2021) at ($(2020)+(\Dx,\Dy)$) {};
   \node [uusi] (2022) at ($(2021)+(\Dx,\Dy)$) {};
   \node [uusi] (2024) at ($(2022)+(2*\Dx,2*\Dy)$) {};

   \node [uusi] (2100) at ($(2000)+(\Bx,\By)$) {};
   \node [uusi] (2101) at ($(2100)+(\Dx,\Dy)$) {};
   \node [uusi] (2102) at ($(2101)+(\Dx,\Dy)$) {};
   \node [uusi] (2103) at ($(2102)+(\Dx,\Dy)$) {};
   \node [uusi] (2104) at ($(2103)+(\Dx,\Dy)$) {};

   \node [uusi] (2200) at ($(2100)+(\Bx,\By)$) {};
   \node [uusi] (2201) at ($(2200)+(\Dx,\Dy)$) {};
   \node [uusi] (2202) at ($(2201)+(\Dx,\Dy)$) {};
   \node [uusi] (2204) at ($(2202)+(2*\Dx,2*\Dy)$) {};

   \node [uusi] (2112) at ($(2012)+(\Bx,\By)$) {};
   \node [uusi, label = {[shift={(0.5,-0.65)},rotate=-26.57]$\clSmaj \cup \clRefl$}] (2113) at ($(2112)+(\Dx,\Dy)$) {};
   \node [uusi] (2114) at ($(2113)+(\Dx,\Dy)$) {};

   \node [uusi] (2122) at ($(2022)+(\Bx,\By)$) {};
   \node [uusi] (2124) at ($(2122)+(2*\Dx,2*\Dy)$) {};

   \node [uusi] (2212) at ($(2112)+(\Bx,\By)$) {};
   \node [uusi] (2214) at ($(2212)+(2*\Dx,2*\Dy)$) {};

   \node [uusi, label = {[shift={(-0.8,-1.45)},rotate=-71.57]$\clEioo \cup \clReflOO$}] (2224) at ($(2124)+(\Bx,\By)$) {};

   \draw (2000) -- (2001) -- (2002) -- (2003) --(2004); \draw (2001) -- (200R) -- (2003);
   \draw (2010) -- (2011) -- (2012) -- (2013) --(2014);
   \draw (2020) -- (2021) -- (2022) -- (2024);
   \draw (2100) -- (2101) -- (2102) -- (2103) --(2104);
   \draw (2200) -- (2201) -- (2202) -- (2204);
   \draw (2112) -- (2113) -- (2114);
   \draw (2122) -- (2124);
   \draw (2212) -- (2214);
   \draw (2000) -- (2010) -- (2020);
   \draw (2001) -- (2011) -- (2021);
   \draw (2002) -- (2012) -- (2022);
   \draw (2003) -- (2013);
   \draw (2004) -- (2014) -- (2024);
   \draw (2102) -- (2112) -- (2122);
   \draw (2103) -- (2113);
   \draw (2104) -- (2114) -- (2124);
   \draw (2202) -- (2212);
   \draw (2204) -- (2214) -- (2224);
   \draw (2000) -- (2100) -- (2200);
   \draw (2001) -- (2101) -- (2201);
   \draw (2002) -- (2102) -- (2202);
   \draw (2003) -- (2103);
   \draw (2004) -- (2104) -- (2204);
   \draw (2012) -- (2112) -- (2212);
   \draw (2013) -- (2113);
   \draw (2014) -- (2114) -- (2214);
   \draw (2022) -- (2122);
   \draw (2024) -- (2124) -- (2224);

   \node [vanha, label = {[shift={(-0.35,-0.65)},rotate=-71.57]$\clOO$}] (3000) at ($(2000)+(\Ax,\Ay)$) {};
   \node [vanha] (3001) at ($(3000)+(\Dx,\Dy)$) {};
   \node [uusi] (3002) at ($(3001)+(\Dx,\Dy)$) {};
   \node [vanha] (3004) at ($(3002)+(2*\Dx,2*\Dy)$) {};

   \node [uusi] (3010) at ($(3000)+(\Cx,\Cy)$) {};
   \node [uusi] (3011) at ($(3010)+(\Dx,\Dy)$) {};
   \node [uusi] (3012) at ($(3011)+(\Dx,\Dy)$) {};
   \node [uusi] (3014) at ($(3012)+(2*\Dx,2*\Dy)$) {};

   \node [vanha, label = {[shift={(-0.35,-0.4)},rotate=26.57]$\clXO$}] (3020) at ($(3010)+(\Cx,\Cy)$) {};
   \node [vanha, label = {[shift={(-0.35,-0.73)},rotate=26.57]$\clXOCI$}] (3021) at ($(3020)+(\Dx,\Dy)$) {};
   \node [uusi] (3022) at ($(3021)+(\Dx,\Dy)$) {};
   \node [vanha, label = {[shift={(0.4,-0.2)}]$\clEioi$}] (3024) at ($(3022)+(2*\Dx,2*\Dy)$) {};

   \node [uusi] (3100) at ($(3000)+(\Bx,\By)$) {};
   \node [uusi] (3101) at ($(3100)+(\Dx,\Dy)$) {};
   \node [uusi] (3102) at ($(3101)+(\Dx,\Dy)$) {};
   \node [uusi] (3104) at ($(3102)+(2*\Dx,2*\Dy)$) {};

   \node [vanha, label = left:$\clOX$] (3200) at ($(3100)+(\Bx,\By)$) {};
   \node [vanha, label = left:$\clOXCI$] (3201) at ($(3200)+(\Dx,\Dy)$) {};
   \node [uusi] (3202) at ($(3201)+(\Dx,\Dy)$) {};
   \node [vanha, label = {[shift={(-0.4,-0.2)}]$\clEiio$}] (3204) at ($(3202)+(2*\Dx,2*\Dy)$) {};

   \node [uusi] (3112) at ($(3012)+(\Bx,\By)$) {};
   \node [uusi] (3114) at ($(3112)+(2*\Dx,2*\Dy)$) {};

   \node [uusi, label = {[shift={(-0.7,-0.65)}]$\clXO \cup \clSmaj$}] (3122) at ($(3022)+(\Bx,\By)$) {};
   \node [uusi, label = {[shift={(0.8,-0.55)}]$\clEioi \cup \clSmaj$}] (3124) at ($(3122)+(2*\Dx,2*\Dy)$) {};

   \node [uusi, label = {[shift={(-0.7,-0.65)}]$\clOX \cup \clSmaj$}] (3212) at ($(3112)+(\Bx,\By)$) {};
   \node [uusi, label = {[shift={(-0.7,-0.65)}]$\clEiio \cup \clSmaj$}] (3214) at ($(3212)+(2*\Dx,2*\Dy)$) {};

   \node [vanha, label = {[shift={(0,0.1)}]$\clAll$}] (3224) at ($(3124)+(\Bx,\By)$) {};

   \draw (3000) -- (3001) -- (3002) -- (3004); 
   \draw (3010) -- (3011) -- (3012) -- (3014);
   \draw (3020) -- (3021) -- (3022) -- (3024);
   \draw (3100) -- (3101) -- (3102) -- (3104);
   \draw (3200) -- (3201) -- (3202) -- (3204);
   \draw (3112) -- (3114);
   \draw (3122) -- (3124);
   \draw (3212) -- (3214);
   \draw (3000) -- (3010) -- (3020);
   \draw (3001) -- (3011) -- (3021);
   \draw (3002) -- (3012) -- (3022);
   \draw (3004) -- (3014) -- (3024);
   \draw (3102) -- (3112) -- (3122);
   \draw (3104) -- (3114) -- (3124);
   \draw (3202) -- (3212);
   \draw (3204) -- (3214) -- (3224);
   \draw (3000) -- (3100) -- (3200);
   \draw (3001) -- (3101) -- (3201);
   \draw (3002) -- (3102) -- (3202);
   \draw (3004) -- (3104) -- (3204);
   \draw (3012) -- (3112) -- (3212);
   \draw (3014) -- (3114) -- (3214);
   \draw (3022) -- (3122);
   \draw (3024) -- (3124) -- (3224);

   \draw (0000) -- (1000) -- (2000) -- (3000);
   \draw (0001) -- (1001) -- (2001) -- (3001);
   \draw (0002) -- (1002) -- (2002) -- (3002);
   \draw (000R) -- (100R) -- (200R);
   \draw (0003) -- (1003) -- (2003);
   \draw (0004) -- (1004) -- (2004) -- (3004);

   \draw (0010) -- (1010) -- (2010) -- (3010);
   \draw (0011) -- (1011) -- (2011) -- (3011);
   \draw (0012) -- (1012) -- (2012) -- (3012);
   \draw (0013) -- (1013) -- (2013);
   \draw (0014) -- (1014) -- (2014) -- (3014);

   \draw (0020) -- (1020) -- (2020) -- (3020);
   \draw (0021) -- (1021) -- (2021) -- (3021);
   \draw (0022) -- (1022) -- (2022) -- (3022);
   \draw (0024) -- (1024) -- (2024) -- (3024);

   \draw (0100) -- (1100) -- (2100) -- (3100);
   \draw (0101) -- (1101) -- (2101) -- (3101);
   \draw (0102) -- (1102) -- (2102) -- (3102);
   \draw (0103) -- (1103) -- (2103);
   \draw (0104) -- (1104) -- (2104) -- (3104);

   \draw (0200) -- (1200) -- (2200) -- (3200);
   \draw (0201) -- (1201) -- (2201) -- (3201);
   \draw (0202) -- (1202) -- (2202) -- (3202);
   \draw (0204) -- (1204) -- (2204) -- (3204);

   \draw (0112) -- (1112) -- (2112) -- (3112);
   \draw (0113) -- (1113) -- (2113);
   \draw (0114) -- (1114) -- (2114) -- (3114);

   \draw (0212) -- (1212) -- (2212) -- (3212);
   \draw (0214) -- (1214) -- (2214) -- (3214);

   \draw (0122) -- (1122) -- (2122) -- (3122);
   \draw (0124) -- (1124) -- (2124) -- (3124);

   \draw (0224) -- (1224) -- (2224) -- (3224);

\end{tikzpicture}
}
\end{center}
\caption{$(\protect\clSc,\protect\clVc)$-clonoids.}
\label{fig:ScVc-clonoids}
\end{figure}

\begin{proof}
By Lemma~\ref{lem:clonoid-inclusion}, every $(\clSc,\clVc)$\hyp{}clonoid is a $(\clSc,\clIc)$\hyp{}clonoid, so
$(\clSc,\clVc)$\hyp{}clonoids are sets of the form $A \cup B \cup C \cup D$, where $A$, $B$, $C$, and $D$ are as
described in Theorem~\ref{thm:disc-clonoids}\ref{thm:disc-clonoids:ScIc}.
However, not every set of this form is a $(\clSc,\clVc)$\hyp{}clonoid, and certain values for $A$, $B$, $C$, and $D$ can be immediately excluded with the help of the following claims.

\begin{claim}
\label{clm:ScVcK:1}
Let $K \subseteq \clAll$, and assume that $\clVc K \subseteq K$.
Let $a, b \in \{0,1\}$.
\begin{enumerate}[label=\upshape{(\roman*)}]
\item\label{clm:ScVcK:1:1}
If $\clIntVal{\clSmin}{a}{b} \subseteq K$ and $(a,b) \neq (1,1)$, then $\clIntVal{\clAll}{a}{b} \subseteq K$.
\item\label{clm:ScVcK:1:2}
If $\clIntVal{\clS}{a}{\overline{a}} \subseteq K$, then $\clIntVal{\clSmaj}{a}{\overline{a}} \subseteq K$.
\end{enumerate}
\end{claim}

\begin{pfclaim}[Proof of Claim~\ref{clm:ScVcK:1}]
\ref{clm:ScVcK:1:1}
Let $f \in \clIntVal{\clAll}{a}{b}$.
Define functions $g$ and $h$ (of the same arity as $f$) as follows.
For $\vect{a} \in \{0,1\}^n$, let
\begin{align*}
g(\vect{a}) &:=
\begin{cases}
f(\vect{a}), & \text{if $(f(\vect{a}), f(\overline{\vect{a}})) \neq (1,1)$,} \\
0, & \text{if $(f(\vect{a}), f(\overline{\vect{a}})) = (1,1)$ and $a_1 = 0$,} \\
1, & \text{if $(f(\vect{a}), f(\overline{\vect{a}})) = (1,1)$ and $a_1 = 1$,}
\end{cases}
\\ 
h(\vect{a}) &:=
\begin{cases}
f(\vect{a}), & \text{if $(f(\vect{a}), f(\overline{\vect{a}})) \neq (1,1)$,} \\
1, & \text{if $(f(\vect{a}), f(\overline{\vect{a}})) = (1,1)$ and $a_1 = 0$,} \\
0, & \text{if $(f(\vect{a}), f(\overline{\vect{a}})) = (1,1)$ and $a_1 = 1$.}
\end{cases}
\end{align*}
Clearly, $g, h \in \clIntVal{\clSmin}{a}{b}$ and $f = g \vee h$.
Therefore, $f \in \clVc \clIntVal{\clSmin}{a}{b} \subseteq \clVc K \subseteq K$.
We have thus shown that $\clIntVal{\clAll}{a}{b} \subseteq K$.

\ref{clm:ScVcK:1:2}
If $f \in \clIntVal{\clSmaj}{a}{\overline{a}}$, then the functions $g$ and $h$ defined above belong to $\clIntVal{\clS}{a}{\overline{a}}$ and $f = g \vee h$.
We can then conclude that
$f \in \clVc \clIntVal{\clS}{a}{\overline{a}} \subseteq \clVc K \subseteq K$.
This shows that $\clIntVal{\clSmaj}{a}{\overline{a}} \subseteq K$.
\end{pfclaim}

It follows immediately from Claim~\ref{clm:ScVcK:1} that every $(\clSc,\clVc)$\hyp{}clonoid is of the form $A \cup B \cup C \cup D$, where
\begin{gather*}
A \in \{ \clEmpty, \clVako, \clReflOO, \clOO \}, \quad
B \in \{ \clEmpty, \clSmajOI, \clOI \}, \quad
C \in \{ \clEmpty, \clSmajIO, \clIO \}, \\
D \in \{ \clEmpty, \clVaki, \clSmajII, \clReflII, \clSmajII \cup \clReflII, \clII \}.
\end{gather*}

We can exclude further quadruples $(A,B,C,D)$ with the help of the following claim.

\begin{claim}
\label{clm:ScVcK:2}
Let $K \subseteq \clAll$, and assume that $\clVc K \subseteq K$.
Let $a, b \in \{0,1\}$.
\begin{enumerate}[label=\upshape{(\roman*)}]
\item\label{clm:ScVcK:2:1}
If $\clOI \cup \clIO \subseteq K$, then $\clII \subseteq K$.
\item\label{clm:ScVcK:2:2}
If $\clIntVal{\clAll}{a}{b} \cup \clReflII \subseteq K$, then $\clII \subseteq K$.
\item\label{clm:ScVcK:2:3}
If $\clSmajOI \cup \clSmajIO \subseteq K$, then $\clSmajII \subseteq K$.
\item\label{clm:ScVcK:2:4}
If $\clIntVal{\clSmaj}{a}{\overline{a}} \cup \clReflII \subseteq K$, then $\clSmajII \subseteq K$.
\end{enumerate}
\end{claim}

\begin{pfclaim}[Proof of Claim~\ref{clm:ScVcK:2}]
\ref{clm:ScVcK:2:1}
Let $f \in \clII$.
Define functions $g$ and $h$ (of the same arity as $f$) as follows.
For $\vect{a} \in \{0,1\}^n$, let
\begin{align*}
g(\vect{a}) &:=
\begin{cases}
f(\vect{a}), & \text{if $\vect{a} \neq \vect{0}$,} \\
0, & \text{if $\vect{a} = \vect{0}$,}
\end{cases}
&
h(\vect{a}) &:=
\begin{cases}
f(\vect{a}), & \text{if $\vect{a} \neq \vect{1}$,} \\
0, & \text{if $\vect{a} = \vect{1}$,}
\end{cases}
\end{align*}
Clearly, $g \in \clOI$, $h \in \clIO$, and $f = g \vee h$.
Therefore, $f \in \clVc (\clOI \cup \clIO) \subseteq \clVc K \subseteq K$.
We conclude that $\clII \subseteq K$.

\ref{clm:ScVcK:2:2}
Let $f \in \clII$.
Define functions $g$ and $h$ (of the same arity as $f$) as follows.
For $\vect{a} \in \{0,1\}^n$, let
\begin{align*}
g(\vect{a}) &:=
\begin{cases}
f(\vect{a}), & \text{if $\vect{a} \notin \{ \vect{0}, \vect{1} \}$,} \\
a, & \text{if $\vect{a} = \vect{0}$,} \\
b, & \text{if $\vect{a} = \vect{1}$,}
\end{cases}
&
h(\vect{a}) &:=
\begin{cases}
1, & \text{if $\vect{a} \in \{ \vect{0}, \vect{1} \}$,} \\
0, & \text{otherwise.}
\end{cases}
\end{align*}
Clearly, $g \in \clIntVal{\clAll}{a}{b}$, $h \in \clReflII$, and $f = g \vee h$.
Therefore, $f \in \clVc (\clIntVal{\clAll}{a}{b} \cup \clReflII) \subseteq \clVc K \subseteq K$.
We conclude that $\clII \subseteq K$.

\ref{clm:ScVcK:2:3}
Let $f \in \clSmajII$.
Define functions $g$ and $h$ as in the proof of Claim~\ref{clm:ScVcK:1}\ref{clm:ScVcK:1:1}.
We now have $g \in \clSmajOI$, $h \in \clSmajIO$, and $f = g \vee h$.
Therefore, $f \in \clVc (\clSmajOI \cup \clSmajIO) \subseteq \clVc K \subseteq K$.
We conclude that $\clSmajII \subseteq K$.

\ref{clm:ScVcK:2:4}
Let $f \in \clSmajII$.
Define functions $g$ and $h$ as in \ref{clm:ScVcK:2:2}, with $b = \overline{a}$.
We now have $g \in \clIntVal{\clSmaj}{a}{\overline{a}}$, $h \in \clReflII$, and $f = g \vee h$.
Therefore, $f \in \clVc (\clIntVal{\clSmaj}{a}{\overline{a}} \cup \clReflII) \subseteq \clVc K \subseteq K$.
We conclude that $\clSmajII \subseteq K$.
\end{pfclaim}

Claim~\ref{clm:ScVcK:2} leaves us with the 123 quadruples $(A,B,C,D)$, or classes $K = A \cup B \cup C \cup D$, that are presented in Table~\ref{table:ScVc} and in Figure~\ref{fig:ScVc-clonoids}.
It remains to verify that these classes are indeed $(\clSc,\clVc)$\hyp{}clonoids.
For this, we only need to prove stability under left composition with $\clVc$.
Because intersections of $(\clSc,\clVc)$\hyp{}clonoids are again $(\clSc,\clVc)$\hyp{}clonoids, it suffices to verify this only for the 16 meet\hyp{}irreducible classes:
$\clAll$,
$\clEiio \cup \clSmaj$,
$\clEioi \cup \clSmaj$,
$\clEiio$,
$\clEioi$,
$\clOX \cup \clSmaj$,
$\clXO \cup \clSmaj$,
$\clOXCI$,
$\clXOCI$,
$\clOX$,
$\clXO$,
$\clEioo \cup \clReflOO$,
$\clEioo \cup \clVako$
$\clEioo$,
$\clSmaj \cup \clRefl$,
$\clRefl$.
This is straightforward verification, especially with the help of Lemma~\ref{lem:AllSmaj}, and we leave the details to the reader.
\end{proof}

\begin{theorem}
\label{thm:ScVc-stab-gen}
For each $(\clSc,\clVc)$\hyp{}clonoid $K$, as determined in 
Proposition~\ref{prop:ScVc-clonoids} and Table~\ref{table:ScVc},
the clones $C_1^K$ and $C_2^K$ prescribed in Table~\ref{table:ScVc} have the property that for every clone $C$, it holds that $KC \subseteq K$ if and only if $C \subseteq C_1^K$, and $CK \subseteq K$ if and only if $C \subseteq C_2^K$.
\end{theorem}

For 42 out of the 123 $(\clSc,\clVc)$\hyp{}clonoids $K$, the clones $C_1^K$ and $C_2^K$ have already been determined in \cite[Theorem 5.1, Table 1]{Lehtonen-SM}.
These are indicated in Table~\ref{table:ScVc} with ``\cite{Lehtonen-SM}'' in the source column and in Figure~\ref{fig:ScVc-clonoids} as the vertices drawn as hollow circles (\begin{tikzpicture}[scale=0.6] \node [vanha] (0) at (0,0) {}; \end{tikzpicture}).
We focus on the remaining $(\clSc,\clVc)$\hyp{}clonoids, indicated in Table~\ref{table:ScVc} with ``\ref{thm:ScVc-stab-gen}'' in the source column and in Figure~\ref{fig:ScVc-clonoids} as the vertices drawn as filled circles (\begin{tikzpicture}[scale=0.6] \node [uusi] (0) at (0,0) {}; \end{tikzpicture}).

\begin{lemma}
\label{lem:Scstabapu}
\leavevmode
\begin{enumerate}[label=\upshape{(\roman*)}]
\item\label{lem:Scstabapu:Lambda1}
For $a, b \in \{0,1\}$ with $(a,b) \neq (0,0)$,
$\{ \mathord{\wedge} \} \clIntVal{\clSmaj}{a}{b} \nsubseteq \clSmaj \cup \clSmin \cup \clRefl \cup (\clAll \setminus \clIntVal{\clAll}{a}{b})$.
\item\label{lem:Scstabapu:Lambda2}
For $a \in \{0,1\}$, $\{ \mathord{\wedge} \} ( \clIntVal{\clAll}{a}{\overline{a}} \cup \clReflOO ) \nsubseteq \clEioo \cup \clReflOO$.
\item\label{lem:Scstabapu:SM1}
If $\{ \mathord{\wedge} \} K \nsubseteq K$ and $\clVako \subseteq K$, then $\{ \mu \} K \nsubseteq K$ and $\{ \veewedge \} K \nsubseteq K$.
\item\label{lem:Scstabapu:SM2}
$\{ \mu \} (\clSmajII \cup \clReflII) \nsubseteq \clSmaj \cup \clSmin \cup \clRefl \cup \clEiii$.
\item\label{lem:Scstabapu:SM3}
For $a, b, c, d \in \{0,1\}$ with $(a,b) \neq (0,0)$,
$\{ \mu \} ( \clIntVal{\clSmaj}{a}{b} \cup \clIntVal{\clAll}{c}{d} ) \nsubseteq \clSmaj \cup \clSmin \cup \clRefl \cup (\clAll \setminus \clIntVal{\clAll}{a}{b})$.
\item\label{lem:Scstabapu:McWinf1}
For $a, b \in \{0,1\}$ with $(a,b) \neq (0,0)$,
$\{ \veewedge \} ( \clIntVal{\clSmaj}{a}{b} \cup \clReflII ) \nsubseteq \clSmaj \cup \clSmin \cup \clRefl \cup \clEiii$.
\item\label{lem:Scstabapu:McWinf2}
For $a \in \{0,1\}$,
$\{ \veewedge \} ( \clIntVal{\clAll}{a}{\overline{a}} \cup \clSmajII ) \nsubseteq \clSmaj \cup \clSmin \cup \clRefl \cup \clEiii$.
\end{enumerate}
\end{lemma}

\begin{proof}
\ref{lem:Scstabapu:Lambda1}
For the functions $f$ and $g$ specified by the operation table below,
we have $f, g \in \clIntVal{\clSmaj}{a}{b}$, but $f \wedge g \notin \clSmaj \cup \clSmin \cup \clRefl \cup (\clAll \setminus \clIntVal{\clAll}{a}{b})$.
\[
\begin{array}{ccc|ccc}
x_1 & x_2 & x_3 & f & g & f \wedge g \\
\hline
0 & 0 & 0 & a & a & a \\
1 & 1 & 1 & b & b & b \\
0 & 0 & 1 & 1 & 1 & 1 \\
1 & 1 & 0 & 1 & 1 & 1 \\
0 & 1 & 0 & 1 & 0 & 0 \\
1 & 0 & 1 & 0 & 1 & 0 \\
1 & 0 & 0 & 0 & 0 & 0 \\
0 & 1 & 1 & 1 & 1 & 1 \\
\end{array}
\]

\ref{lem:Scstabapu:Lambda2}
For the functions $f$ and $g$ specified by the operation table below,
we have $f \in \clIntVal{\clAll}{a}{\overline{a}}$, $g \in \clReflOO$ but $f \wedge g \notin \clEioo \cup \clReflOO$.
\[
\begin{array}{cc|ccc}
x_1 & x_2 & f & g & f \wedge g \\
\hline
0 & 0 & a & 0 & 0 \\
0 & 1 & 0 & 1 & 0 \\
1 & 0 & 1 & 1 & 1 \\
1 & 1 & \overline{a} & 0 & 0
\end{array}
\]

\ref{lem:Scstabapu:SM1}
Because $\{\mathord{\wedge}\} K \nsubseteq K$, there exist $f, g \in K$ such that $f \wedge g \notin K$.
Then
\begin{align*}
& \mu(0,f,g) = (0 \wedge f) \vee (0 \wedge g) \vee (f \wedge g) = f \wedge g \notin K, \\
& \veewedge(0, f, g) = 0 \vee (f \wedge g) = f \wedge g \notin K,
\end{align*}
so $\{\mu\} K \nsubseteq K$ and $\{ \veewedge \} K \nsubseteq K$.

\ref{lem:Scstabapu:SM2},
\ref{lem:Scstabapu:SM3}
For the functions $f$, $g$, and $h$ specified by the operation table below,
we have $f, g \in \clIntVal{\clSmaj}{a}{b}$ and $h \in \clIntVal{\clAll}{c}{d}$ but $\mu(f,g,h) \notin \clSmaj \cup \clSmin \cup \clRefl \cup (\clAll \setminus \clIntVal{\clAll}{a}{b})$.
Moreover, if $c = d$, then $h \in \clIntVal{\clRefl}{c}{c}$.
\[
\begin{array}{ccc|cccc}
x_1 & x_2 & x_3 & f & g & h & \mu(f,g,h) \\
\hline
0 & 0 & 0 & a & a & c & a \\
1 & 1 & 1 & b & b & d & b \\
0 & 0 & 1 & 1 & 1 & 1 & 1 \\
1 & 1 & 0 & 1 & 1 & 1 & 1 \\
0 & 1 & 0 & 1 & 0 & 0 & 0 \\
1 & 0 & 1 & 0 & 1 & 0 & 0 \\
1 & 0 & 0 & 0 & 0 & 0 & 0 \\
0 & 1 & 1 & 1 & 1 & 0 & 1 \\
\end{array}
\]

\ref{lem:Scstabapu:McWinf1}
For the functions $f$, $g$, and $h$ specified by the operation table below,
we have $f \in \clReflII$ and $g, h \in \clIntVal{\clSmaj}{a}{b}$ but $\veewedge(f,g,h) \notin \clSmaj \cup \clSmin \cup \clRefl \cup \clEiii$.
\[
\begin{array}{ccc|cccc}
x_1 & x_2 & x_3 & f & g & h & \veewedge(f,g,h) \\
\hline
0 & 0 & 0 & 1 & a & a & 1 \\
1 & 1 & 1 & 1 & b & b & 1 \\
0 & 0 & 1 & 1 & 0 & 0 & 1 \\
1 & 1 & 0 & 1 & 1 & 1 & 1 \\
0 & 1 & 0 & 0 & 0 & 1 & 0 \\
1 & 0 & 1 & 0 & 1 & 0 & 0 \\
1 & 0 & 0 & 0 & 0 & 0 & 0 \\
0 & 1 & 1 & 0 & 1 & 1 & 1
\end{array}
\]

\ref{lem:Scstabapu:McWinf2}
For the functions $f$, $g$, and $h$ specified by the operation table below,
we have $f \in \clIntVal{\clAll}{a}{\overline{a}}$ and $g, h \in \clSmajII$ but $\veewedge(f,g,h) \notin \clSmaj \cup \clSmin \cup \clRefl \cup \clEiii$.
\[
\begin{array}{ccc|cccc}
x_1 & x_2 & x_3 & f & g & h & \veewedge(f,g,h) \\
\hline
0 & 0 & 0 & a & 1 & 1 & 1 \\
1 & 1 & 1 & \overline{a} & 1 & 1 & 1 \\
0 & 0 & 1 & 0 & 1 & 1 & 1 \\
1 & 1 & 0 & 0 & 1 & 1 & 1 \\
0 & 1 & 0 & 0 & 0 & 1 & 0 \\
1 & 0 & 1 & 0 & 1 & 0 & 0 \\
1 & 0 & 0 & 0 & 0 & 0 & 0 \\
0 & 1 & 1 & 0 & 1 & 1 & 1
\end{array}
\]
\end{proof}

\begin{proof}[Proof of Theorem~\ref{thm:ScVc-stab-gen}]
We determine, for each $(\clSc,\clVc)$\hyp{}clonoid $K$, the clones $C_1^K$ and $C_2^K$ such that $K C_1 \subseteq C_1$ if and only if $C_1 \subseteq C_1^K$ and $C_2 K \subseteq K$ if and only if $C_2 \subseteq C_2^K$.
We get the clones $C_1^K$ immediately from Theorem~\ref{thm:disc-clonoids:ScIc}.

For 42 out of the 123 $(\clSc,\clVc)$\hyp{}clonoids $K$, as indicated in Table~\ref{table:ScVc} (the ones with ``\cite{Lehtonen-SM}'' in the source column) and in Figure~\ref{fig:ScVc-clonoids} (the vertices drawn as hollow circles), the clones $C_2^K$ have been determined in \cite[Theorem~5.1, Table~1]{Lehtonen-SM}.
From now on, we assume that $K$ is one of the remaining $(\clSc,\clVc)$\hyp{}clonoids.

By Lemma~\ref{lem:CL-Lc-3.3}, left stability under composition with a clone $C_2$ can be tested with a generating set of $C_2$.
We are going to determine $C_2^K$ by testing whether $K$ is stable under left composition with certain functions.
Using this information together with Post's lattice, we can identify the right clone.
We clearly have $\clVc \subseteq C_2^K$ because $K$ is stable under left composition with $\clVc$.

\begin{claim}
\label{clm:ScVc-1}
For $a \in \{0,1\}$, $a \in C_2^K$ if and only if $\clVaka{a} \subseteq K$.
\end{claim}
\begin{pfclaim}[Proof of Claim~\ref{clm:ScVc-1}]
This follows immediately from Lemma~\ref{lem:C1C2Vak}\ref{lem:C1C2Vak:ii}.
\end{pfclaim}

\begin{claim}
\label{clm:ScVc-2}
$\mathord{\wedge} \notin C_2^K$.
\end{claim}
\begin{pfclaim}[Proof of Claim~\ref{clm:ScVc-2}]
Because one of the inclusions
\begin{gather*}
\clSmajOI \subseteq K \subseteq \clEioi \cup \clSmaj, \quad
\clSmajIO \subseteq K \subseteq \clEiio \cup \clSmaj, \quad
\clSmajII \subseteq K \subseteq \clEiii \cup \clSmaj \cup \clRefl, \\
\clOI \cup \clReflOO \subseteq K \subseteq \clXI \cup \clReflOO, \quad
\clIO \cup \clReflOO \subseteq K \subseteq \clIX \cup \clReflOO
\end{gather*}
holds,
it follows from Lemma~\ref{lem:Scstabapu}\ref{lem:Scstabapu:Lambda1} and \ref{lem:Scstabapu:Lambda2}
that $\{\mathord{\wedge}\} K \nsubseteq K$.
\end{pfclaim}

\begin{claim}
\label{clm:ScVc-3}
$\mu \notin C_2^K$.
\end{claim}
\begin{pfclaim}[Proof of Claim~\ref{clm:ScVc-3}]
Because one of the inclusions
\begin{gather*}
\clVako \subseteq K,
\quad
\clSmajII \cup \clReflII \subseteq K \subseteq \clSmaj \cup \clReflII,
\quad
\clSmajOI \cup \clII \subseteq K \subseteq \clSmajOI \cup \clIX,
\\
\clSmajIO \cup \clII \subseteq K \subseteq \clSmajIO \cup \clXI,
\quad
\clSmajII \cup \clOI \subseteq K \subseteq \clSmajIX \cup \clOI,
\\
\clSmajII \cup \clIO \subseteq K \subseteq \clSmajXI \cup \clIO
\end{gather*}
holds,
it follows 
from Lemma~\ref{lem:Scstabapu}\ref{lem:Scstabapu:SM1}, \ref{lem:Scstabapu:SM2}, and \ref{lem:Scstabapu:SM3} 
that $\{ \mu \} K \nsubseteq K$.
\end{pfclaim}

\begin{claim}
\label{clm:ScVc-4}
If $K$ is not any of the classes
$\clII \cup \clSmajOI$,
$\clII \cup \clSmajIO$,
$\clII \cup \clSmaj$,
$\clXI \cup \clSmajIO$,
$\clIX \cup \clSmajOI$,
then $\veewedge \notin C_2^K$.
\end{claim}
\begin{pfclaim}[Proof of Claim~\ref{clm:ScVc-4}]
If $K$ is not one of the five classes mentioned,
then one of the inclusions
\begin{gather*}
\clVako \subseteq K,
\quad
\clSmajII \cup \clReflII \subseteq K \subseteq \clSmaj \cup \clReflII,
\\
\clOI \cup \clSmajII \subseteq K \subseteq \clOI \cup \clSmajIX,
\quad
\clIO \cup \clSmajII \subseteq K \subseteq \clIO \cup \clSmajXI
\end{gather*}
holds,
and
it follows 
from Lemma~\ref{lem:Scstabapu}\ref{lem:Scstabapu:SM1}, \ref{lem:Scstabapu:McWinf1}, and \ref{lem:Scstabapu:McWinf2} 
that $\{ x_1 \vee (x_2 \wedge x_3) \} K \nsubseteq K$.
\end{pfclaim}

\begin{claim}
\label{clm:ScVc-5}
If $K$ is one of
$\clII \cup \clSmajOI$,
$\clII \cup \clSmajIO$,
$\clII \cup \clSmaj$,
$\clXI \cup \clSmajIO$,
$\clIX \cup \clSmajOI$,
then $\{ \mathord{\rightarrow}, \threshold{4}{2} \} \subseteq C_2^K$.
\end{claim}
\begin{pfclaim}[Proof of Claim~\ref{clm:ScVc-5}]
We prove first that $\{\mathord{\rightarrow}\} K \subseteq K$.
Let $f, g \in K$, and consider $f \rightarrow g$.
If $g$ is in $\clII$ ($\clXI$, $\clIX$, resp.), then $f \rightarrow g$ is in $\clII$ ($\clXI$, $\clIX$, resp.) and we are done.
If $g \in \clSmajOI$, then $f \rightarrow g$ is in $\clXI$. If $f \rightarrow g \in \clII$, then we are done, so assume that $f \rightarrow g \in \clOI$.
Because $g(\vect{a}) \vee g(\overline{\vect{a}}) = 1$ for all $\vect{a}$, we have
\begin{align*}
& (f \rightarrow g)(\vect{a}) \vee (f \rightarrow g)(\overline{\vect{a}})
= (f(\vect{a}) \rightarrow g(\vect{a})) \vee (f(\overline{\vect{a}}) \rightarrow g(\overline{\vect{a}}))
\\ & = (\overline{f(\vect{a})} \vee g(\vect{a})) \vee (\overline{f(\overline{\vect{a}})} \vee g(\overline{\vect{a}}))
= (\overline{f(\vect{a})} \vee \overline{f(\overline{\vect{a}})}) \vee (g(\vect{a}) \vee g(\overline{\vect{a}}))
= 1,
\end{align*}
so, in fact, $f \rightarrow g \in \clSmajOI$.
In a similar way, we can show that if $g \in \clSmajIO$, then $f \rightarrow g \in \clSmajIO$.
From these facts it follows that $\{ \mathord{\rightarrow} \} K \subseteq K$.

We now prove that $\{ \threshold{4}{2} \} K \subseteq K$.
Let $f_1, f_2, f_3, f_4 \in K$, and consider $\varphi := \threshold{4}{2}(f_1, f_2, f_3, f_4) \in \{\threshold{4}{2}\} K$.
If two of the functions $f_1$, $f_2$, $f_3$, $f_4$ are in $\clII$ ($\clXI$, $\clIX$, resp.), then $\varphi$ is in $\clII$ ($\clXI$, $\clIX$, resp.).
Assume thus that at most one of $f_1$, $f_2$, $f_3$, $f_4$ is in $\clII$ ($\clXI$, $\clIX$, resp.); without loss of generality, assume that this function, if any, is $f_4$.
Then $f_1, f_2, f_3$ are in $\clSmajOI$ ($\clSmajIO$, $\clSmaj$, resp.).
But then, for any $\vect{a} \in \{0,1\}^n$, there must exist $i, j \in \nset{3}$ such that $f_i(\vect{a}) = f_j(\vect{a}) = 1$ or $f_i(\overline{\vect{a}}) = f_j(\overline{\vect{a}})$; consequently, $\varphi(\vect{a}) = 1$ or $\varphi(\overline{\vect{a}}) = 1$.
Thus, $\varphi \in \clSmaj$.
Moreover, if $f_1, f_2, f_3 \in \clSmajOI$, then $\varphi \in \clSmajOI$; and if $f_1, f_2, f_3 \in \clSmajIO$, then $\varphi \in \clSmajIO$.
From these facts it follows that $\{ \threshold{4}{2} \} K \subseteq K$.
\end{pfclaim}

From the above claims we conclude that
$C_2^K = \clonegen{\mathord{\rightarrow}, \threshold{4}{2}} = \clWk{3}$ if and only if $K$ is one of the classes
$\clII \cup \clSmajOI$,
$\clII \cup \clSmajIO$,
$\clII \cup \clSmaj$,
$\clXI \cup \clSmajIO$,
$\clIX \cup \clSmajOI$.
In the case when $K$ is not one of these classes, we have that
$C_2^K = \clV$ if and only if $\clVak \subseteq K$;
$C_2^K = \clVo$ if and only if $\clVako \subseteq K$ but $\clVaki \cap K = \emptyset$;
$C_2^K = \clVi$ if and only if $\clVaki \subseteq K$ but $\clVako \cap K = \emptyset$;
and 
$C_2^K = \clVc$ if and only if $\clVak \cap K = \emptyset$.
This is presented explicitly in Table~\ref{table:ScVc}.
\end{proof}


\section{Final remarks}
\label{sec:remarks}

It remains out of the scope of this paper to consider $(C_1,C_2)$\hyp{}clonoids in the cases where the source clone $C_1$ in not a discriminator clone nor a monotone clone.
In such cases the $(C_1,\clIc)$\hyp{}clonoid lattice is uncountable, and it may not be possible to explicitly describe it.
Sparks's theorem (Theorem~\ref{thm:Sparks}) could nevertheless be extended by determining the cardinality of the lattice of $(C_1,C_2)$\hyp{}clonoids for arbitrary pairs $(C_1,C_2)$ of clones (in Theorem~\ref{thm:Sparks}, the source clone is fixed to be the clone of projections).
This remains a topic for further investigation.
Another direction for further further research lies in generalizing our current findings to arbitrary finite base sets.




\end{document}